\documentclass[a4paper,12pt]{scrartcl}
\usepackage{amsmath,amssymb,latexsym,amsthm,enumerate}
\usepackage{xcolor}
\usepackage{graphicx}
\usepackage{epstopdf}
\usepackage{changes}
\usepackage{bm}
\usepackage[nocompress]{cite}

\usepackage[T1]{fontenc}
\usepackage[utf8]{inputenc}
\usepackage[english]{babel}

\usepackage{hyperref}

\newcommand*{\wh}{\widehat}
\newcommand*{\wt}{\widetilde}
\newcommand*{\ol}{\overline}

\newcommand*{\PdM}{\mathcal{D}_M} 
\newcommand*{\PdW}{\mathcal{D}_W} 

\newcommand*{\bbN}{\mathbb N}

\newcommand*{\bbR}{\mathbb R}

\newcommand*{\cA}{\mathcal{A}}

\newcommand*{\cD}{\mathcal{D}}
\newcommand*{\cE}{\mathcal{E}}
\newcommand*{\cF}{\mathcal{F}}

\newcommand*{\N}{\mathbb{N}}

\newcommand*{\R}{\mathbb{R}}

\DeclareMathOperator*{\essinf}{ess\,inf}

\newcommand*{\eps}{\varepsilon}

\newcommand*{\const}{\mathrm{const}}

\newcommand*{\Leb}{\operatorname{Leb}}

\newcommand{\be}{\begin{eqnarray*}}
\newcommand{\ee}{\end{eqnarray*}}
\newcommand{\ben}{\begin{eqnarray}}
\newcommand{\een}{\end{eqnarray}}
\newcommand{\bi}{\begin{itemize}}
\newcommand{\ei}{\end{itemize}}

\newtheorem{theo}{Theorem}[section]

\newtheorem{lemma}[theo]{Lemma}
\newtheorem{propo}[theo]{Proposition}
\newtheorem{corollary}[theo]{Corollary}

\theoremstyle{definition}
\newtheorem{ex}[theo]{Example}

\newtheorem{remark}[theo]{Remark}

\title{C\`adl\`ag semimartingale strategies\\for optimal trade execution\\in stochastic order book models}

\author{Julia Ackermann\thanks{Institute of Mathematics, University of Gie{\ss}en, Arndtstr.~2, 35392 Gießen, Germany.
\emph{Email:} julia.ackermann@math.uni-giessen.de, \emph{Phone:} +49 (0)641 9932113.}
\and Thomas Kruse\thanks{Institute of Mathematics, University of Gie{\ss}en, Arndtstr.~2, 35392 Gießen, Germany.
\emph{Email:} thomas.kruse@math.uni-giessen.de, \emph{Phone:} +49 (0)641 9932102.}
\and Mikhail Urusov\thanks{Faculty of Mathematics, University of Duisburg-Essen, Thea-Leymann-Str.~9, 45127 Essen, Germany.
\emph{Email:} mikhail.urusov@uni-due.de, \emph{Phone:} +49 (0)201 1837428.
}}

\begin{document}

\maketitle

\begin{abstract}
We analyze an optimal trade execution problem in a financial market with stochastic liquidity.
To this end we set up a limit order book model in continuous time. Both order book depth and resilience are allowed to evolve randomly in time.
We allow for trading in both directions and for c\`adl\`ag semimartingales as execution strategies. 
We derive a quadratic BSDE that under appropriate assumptions characterizes minimal execution costs and identify conditions under which an optimal execution strategy exists. We also investigate qualitative aspects of optimal strategies such as, 
e.g., appearance of strategies with infinite variation or 
existence of block trades and discuss connections with the discrete-time formulation of the problem. Our findings are illustrated in several examples. 


\end{abstract}

\section{Introduction}
Liquidity in financial markets is not constant but evolves randomly in time. 
To better understand the effects of stochastic liquidity on trade execution strategies, the recent research articles \cite{almgren2012optimal,
schied2013control,
ankirchner2014bsdes,
cheridito2014optimal,
ankirchner2015optimal,
graewe2015non,
kruse2016minimal,
horst2016constrained,
graewe2017optimal,
bank2018linear,
graewe2018smooth,
horst2019multi,
popier2019second,
ankirchner2020optimal}
extend the market impact models of Almgren and Chriss \cite{almgren2001optimal} and Bertsimas and Lo \cite{bertsimas1998optimal} 
by allowing for stochastic liquidity parameters. 
In this article we take a different route and introduce a variant of the limit order book models of
Alfonsi, Fruth and Schied \cite{alfonsi2010optimal,alfonsi2008constrained},
Obizhaeva and Wang \cite{obizhaeva2013optimal}
and
Predoiu, Shaikhet and Shreve \cite{predoiu2011optimal},
where both order book depth and resilience are allowed to evolve randomly in time.

To this end we fix a time horizon $T\in (0,\infty)$ and consider a filtered probability space $(\Omega, \cF,  (\cF_t)_{t \in [0,T]}, P)$ which satisfies the usual conditions
and $\cF=\cF_T$. 
Let $M=(M_t)_{t\in [0,T]}$ be a continuous local martingale.
The \textit{price impact process} $\gamma=(\gamma_t)_{t\in [0,T]}$ evolves according to the stochastic dynamics 
\begin{equation}\label{eq:dyn_gamma}
d\gamma_t = \gamma_t\left( \mu_t\, d[M]_t + \sigma_t dM_t \right),\quad t\in [0,T], \quad \gamma_0>0,
\end{equation}
with progressively measurable coefficient processes $\mu=(\mu_t)_{t\in [0,T]}$ and
$\sigma=(\sigma_t)_{t\in [0,T]}$ satisfying appropriate integrability assumptions. 

Given an open position $x\in \R$ of a certain asset at time $t\in [0,T]$, an \emph{execution strategy} is a c\`adl\`ag semimartingale $X=(X_s)_{s\in [t,T]}$ satisfying $X_{t-}=x$ and $X_T=0$. 
A positive initial position $x>0$ means the trader has to sell an amount of $|x|$ shares, whereas $x<0$ requires to buy an amount of $|x|$ shares, having the time interval $[t,T]$ at disposal for trading. 
The terminal condition $X_T=0$ describes the liquidation constraint that at time $T$ the position has to be closed. 
For every $s\in [t,T]$ the quantity $X_{s-}$ reflects the remaining position to be closed during $[s,T]$.  
A jump at time $s$
(notation: $\Delta X_s=X_s-X_{s-}$)
is interpreted as a block trade at $s$. 
In particular, execution strategies will typically have a block trade at the beginning, i.e., $X_t$ often differs from the initial position $X_{t-}=x$.
In line with the previous sentence,
we always use the convention $[X]_{t-}=0$ and $[X]_t=(\Delta X_t)^2$
concerning the quadratic variation $[X]$ of the process $X=(X_s)_{s\in [t,T]}$
 at the initial time,
i.e.,
the quadratic variation process $[X]$ will typically jump at the initial time as well.\footnote{\label{ft:01042021a1}We stress that the conventions about possible jumps at the initial time apply also in the case when the initial time $t=0$, i.e.,
$X_0$ can differ from $X_{0-}=x$ and, consequently,
$[X]_0$ ($=(\Delta X_0)^2$) can differ from $[X]_{0-}$ ($=0$).}

Market illiquidity implies that trading according to a strategy $X$ impacts the asset price and induces a price deviation. 
We thus assume that the actual price of a share equals the unaffected price, which is the price of a share in absence of trading, plus a deviation that depends on the execution strategy. In our analysis, we focus on the price deviation 
(see, however, Remark~\ref{rem:01042021a1} for more detail) 
and model it by associating to each strategy $X$ a \textit{deviation process} $D=(D_s)_{s\in [t,T]}$ which evolves according to 
\begin{equation}\label{eq:deviation_dyn}
dD_s=-\rho_s D_{s} \,d[M]_s +\gamma_s\, dX_s+d[\gamma,X]_s, \quad s\in [t,T], \quad D_{t-}=d,
\end{equation}
where the progressively measurable coefficient process $\rho=(\rho_s)_{s\in [0,T]}$ is called the \textit{resilience process}. 
Notice that $D$ can have a jump at the initial time $t$, which corresponds to a jump of $X$, hence the initial condition on $D_{t-}$. 
Typically, the initial price deviation $d\in \R$ is assumed to be zero, but we allow for arbitrary values in our formulation and analysis.

Given an initial price deviation $d\in \R$ we denote by $\cA_t(x,d)$ the set of all c\`adl\`ag semimartingale execution strategies for closing an initial position $x\in \R$ during $[t,T]$ 
which in addition satisfy suitable integrability conditions (see (A1)--(A3) below). 
The expected execution costs for a strategy $X \in \cA_t(x,d)$
are given by
\footnote{\label{ft:01042021a2}
We follow the convention that,
for integrals of the forms $\int_{[t,T]} \ldots\,dX_s$ and $\int_{[t,T]}\ldots\,d[X]_s$,
jumps of the c\`adl\`ag integrators $X$ and $[X]$ at time $t$ contribute to the integrals.
In contrast, we write $\int_{(t,T]} \ldots\,dX_s$ and $\int_{(t,T]}\ldots\,d[X]_s$ when we do not include the jumps at time $t$ into the integrals.
In particular, a possible initial block trade at time $t$ contributes to both integrals in the cost functional $J$ in~\eqref{eq:cost_fct}.}
\begin{equation}\label{eq:cost_fct}
J_t(x,d,X)=E_t\left[\int_{[t,T]}D_{s-}dX_s+\int_{[t,T]}\frac{\gamma_s}{2}d[X]_s\right],
\end{equation}
where $E_t[\cdot]$ is a shorthand notation for
$E[\cdot|\cF_t]$. 
In this article we provide a purely probabilistic solution to the stochastic control problem of minimizing $J$ over $\cA_t(x,d)$. More precisely, under appropriate assumptions we characterize the value function
\begin{equation}\label{eq:def_value_fct}
V_t(x,d)=\essinf_{X \in \cA_t(x,d)}J_t(x,d,X), \quad x,d \in \R,\; t\in [0,T],
\end{equation}
of the control problem in terms of a quadratic backward stochastic differential equation (BSDE). Let $Y=(Y_s)_{s\in [0,T]}$ denote the first solution component of BSDE~\eqref{eq:bsde}, 
then we show in Theorem~\ref{thm:sol_val_fct} that the minimal expected costs amount to
\begin{equation*}
	V_t(x,d)= \frac{Y_{t}}{\gamma_t}\left(d-\gamma_t x\right)^2-\frac{d^2}{2\gamma_t} .
\end{equation*}
Moreover, we identify conditions under which an optimal strategy exists and give an explicit representation in terms of the solution of the BSDE.

Observe that the class of execution strategies $\cA_t(x,d)$ over which we optimize in \eqref{eq:def_value_fct} is a subclass of all c\`adl\`ag semimartingales.
In particular, the strategies are allowed to have infinite variation.
Only few research articles on optimal trade execution so far analyze problems where strategies are not restricted to have finite variation.
Infinite variation strategies are considered in \cite{lorenz2013drift}
in a setting related to ours
in order to investigate how optimal execution strategies react to a possible drift in the unaffected price process. 
The work \cite{becherer2019stability} explains how to go beyond finite variation strategies in another model for gains of a large investor.
Strategies of infinite variation also appear in \cite{carmona2019selffinancing} and \cite{garleanu2016dynamic}.
These articles study a generalization of the self-financing equation and an infinite horizon portfolio optimization problem under market frictions, respectively.
The recent work \cite{HorstKivman2021} considers a model with instantaneous price impact and stochastic resilience from \cite{graewe2017optimal} and \cite{horst2019multi}, and strategies of infinite variation emerge there in the limiting case of vanishing instantaneous price impact.

By allowing strategies of infinite variation we encounter several interesting effects, which we now discuss in more detail. In particular, this extension of the set of admissible controls requires an adjustment of conventional dynamics of the price deviation process and the cost functional as presented, e.g., in
\cite{alfonsi2014optimal},
\cite{bank2014optimal},
\cite{fruth2014optimal},
\cite{fruth2019optimal},
\cite{obizhaeva2013optimal},
\cite{predoiu2011optimal}
(see also the references therein).
In these papers, the trading is either constrained in one direction
or the execution strategies $X$ are assumed to be of finite variation,
which de facto translates into the dynamics
\begin{equation}\label{eq:13022020a1_intro}
d\wt D_s=-\rho_s\wt D_{s}\,d[M]_s+\gamma_s\,dX_s 
\end{equation}
for the deviation process and into the cost functional of the form
\begin{equation}\label{eq:13022020a2_intro}
\wt J_t(x,d,X)=E_t\left[\int_{[t,T]}\left(\wt D_{s-}+\frac{\gamma_s}{2}\Delta X_s\right)\,dX_s\right].
\end{equation}
Tildes in \eqref{eq:13022020a1_intro} and~\eqref{eq:13022020a2_intro} are to distinguish these from our setting. 
Let us recall that, for two c\`adl\`ag semimartingales $K=(K_s)_{s\in[t,T]}$ and $L=(L_s)_{s\in[t,T]}$, it holds for all $s\in[t,T]$ that
$[K,L]_s=\langle K^c,L^c\rangle_s+\sum_{u\in[t,s]}\Delta K_u\Delta L_u$
(see Theorem~I.4.52 in \cite{jacodshiryaev}),
where $K^c$ and $L^c$ denote the \emph{continuous martingale parts} of $K$ and $L$
(see Proposition~I.4.27 in \cite{jacodshiryaev}). 
In particular, $[X]_s=\langle X^c\rangle_s+\sum_{u\in[t,s]}(\Delta X_u)^2$ and, as $\gamma$ is continuous, $[\gamma,X]_s=\langle\gamma,X^c\rangle_s$, $s\in[t,T]$.
Therefore, if in our setting an execution strategy $X$ is monotone or, more generally, of finite variation,
then \eqref{eq:deviation_dyn} 
reduces to~\eqref{eq:13022020a1_intro},
while \eqref{eq:cost_fct} 
reduces to~\eqref{eq:13022020a2_intro}
(as $X^c=0$ in that case). 
In general, i.e., when $X$ is a c\`adl\`ag semimartingale, we have additional terms in the dynamics of the deviation process~\eqref{eq:deviation_dyn}
and in the cost functional~\eqref{eq:cost_fct} in comparison with the conventional setting of the problem.
To explain these additional terms, we now make the following comments:
\begin{itemize}
\item
Preserving the conventional dynamics~\eqref{eq:13022020a1_intro} and the cost functional~\eqref{eq:13022020a2_intro}
can result in an ill-posed optimization problem in our setting, 
see counterexamples in Sections \ref{sec:fromsec81offirstversion} and~\ref{sec:fromsec52offirstversion}, respectively.
\item
Specifically, using cost functional~\eqref{eq:13022020a2_intro} for strategies $X$ of infinite variation
can lead to arbitrarily big negative costs even with constant in time deterministic price impact $\gamma$
(in which case~\eqref{eq:deviation_dyn} 
and~\eqref{eq:13022020a1_intro} are the same),
see Section~\ref{sec:fromsec52offirstversion}.
With the right cost functional~\eqref{eq:cost_fct} 
we recover a well-posed problem,
see Section~\ref{sec:recoveringOW}.
\item
Furthermore, even with the right cost functional~\eqref{eq:cost_fct},  
the dynamics~\eqref{eq:13022020a1_intro}
for the deviation process can lead to arbitrarily big negative costs,
see Section~\ref{sec:fromsec81offirstversion}.
With the right dynamics~\eqref{eq:deviation_dyn} 
we again recover a well-posed problem,
see Section~\ref{sec:examplewithLambertWfct}.
\item
It is worth noting that the specific form of the corrections to \eqref{eq:13022020a1_intro} and~\eqref{eq:13022020a2_intro}
when allowing strategies $X$ to have infinite variation
can be justified by a limiting procedure from the discrete-time situation,
see Appendix~\ref{sec:formalderivationdeviationcostfct}.
\end{itemize}

The preceding discussion raises the question of whether it is really so necessary to try to include strategies $X$ of infinite variation into the picture.
The answer is affirmative: 
as we do not constrain the trading in one direction and allow for stochastic price impact and resilience processes
$\gamma$ and $\rho$, which, moreover, can have infinite variation, it is quite natural to expect that strategies of infinite variation come out.
The intuition is that the optimal strategy should react to changes in the exogenously given processes $\gamma$ and $\rho$,
and if $\gamma$ (or $\rho$) has infinite variation, then the optimal strategy should typically have infinite variation as well. 
But we actually discover a more surprizing effect:
even in the situation when the resilience $\rho$ is a deterministic constant 
and the price impact process $\gamma$ has $C^\infty$ paths 
(in particular, all exogenously given processes have finite variation), 
it can happen that the optimal strategy in our problem~\eqref{eq:def_value_fct}  
has infinite variation.
This and several other interesting qualitative effects are presented in Section~\ref{sec:furtherexamples}. 
We moreover
remark that we study a discrete-time version of the problem in
\cite{ackermann2020optimal},
but that study is concentrated on different questions, as the mentioned effects, being purely continuous-time features, cannot be discussed in the framework of \cite{ackermann2020optimal}.

From another perspective, we would like to mention \cite{carmona2019selffinancing}, who examine high-frequency trading in limit order books in general (not necessarily related with optimal trade execution).
It is very interesting that one of their conclusions is the empirical evidence for the infinite variation nature of trading strategies of high-frequency traders.

Finally, to complement the above discussion about the necessity of some adjustments in our setting, 
we discuss related literature from the viewpoint of the adjustments \eqref{eq:13022020a1_intro}$\to$\eqref{eq:deviation_dyn} and \eqref{eq:13022020a2_intro}$\to$\eqref{eq:cost_fct}:
\begin{itemize}
\item
Adjustments in the cost functional similar to \eqref{eq:13022020a2_intro}$\to$\eqref{eq:cost_fct} already appeared in related settings in \cite{lorenz2013drift} and \cite{garleanu2016dynamic}.
In both papers the terms in the cost functional containing the quadratic variation $[X]$ are justified via limiting arguments from discrete time.
Further, \cite{HorstKivman2021} proves that the limiting strategy in the case of vanishing instantaneous price impact (cf.\ the discussion above) can be viewed as the optimal strategy in a problem of optimal execution with semimartingale strategies, where the cost functional is in the spirit of~\eqref{eq:cost_fct}.
\item
On the contrary, the adjustment \eqref{eq:13022020a1_intro}$\to$\eqref{eq:deviation_dyn} in the dynamics of the deviation process is to the best of our knowledge new. It does not emerge in the aforementioned papers because they consider constant $\gamma$, in which case $[\gamma,X]\equiv0$. In order to see the need for the adjustment \eqref{eq:13022020a1_intro}$\to$\eqref{eq:deviation_dyn}, it is necessary to consider the price impact itself (i.e., the process $\gamma$) to be of infinite variation.
\end{itemize}

The remainder of this article is organized as follows. In Section~\ref{sec:problem_formulation} we formally introduce the trade execution problem as a stochastic optimization problem over semimartingales. 
Section~\ref{sec:main_results} presents the main results of this article. In Section~\ref{sec:reason_cost_fct} resp.\ Section~\ref{sec:reason_dynamics_of_deviation} we explain why the cost functional in~\eqref{eq:13022020a2_intro} has to be adjusted
to~\eqref{eq:cost_fct} 
resp.\ why the dynamics in~\eqref{eq:13022020a1_intro} have to be adjusted 
to~\eqref{eq:deviation_dyn}  
to obtain a well-posed optimization problem when minimizing over semimartingales.
We present several examples and describe qualitative effects of optimal strategies in Section~\ref{sec:furtherexamples}.
Section~\ref{sec:analysis_of_BSDE} establishes existence results for BSDE~\eqref{eq:bsde}. Section~\ref{sec:proofs_of_main_results} is devoted to the proofs of the results in Section~\ref{sec:main_results}.
Appendices \ref{sec:formalderivationdeviationcostfct}
and~\ref{sec:formalderivationbsde} 
additionally 
justify the form of the cost functional in~\eqref{eq:cost_fct}, 
the dynamics in~\eqref{eq:deviation_dyn} 
and BSDE~\eqref{eq:bsde} by deriving these objects as
continuous-time limits from the corresponding
discrete-time problem formulation.
Appendix~\ref{sec:comp_bsde} contains a certain comparison argument for BSDEs, which is used in Section~\ref{sec:analysis_of_BSDE}.

\section{Problem formulation}\label{sec:problem_formulation}

Throughout we consider a continuous local martingale
$M=(M_t)_{t\in [0,T]}$ 
and denote by $\PdM$  the Dol{\'e}ans measure associated to $M$
on 
$\left( \Omega\times[0,T] , \cF \otimes \mathcal{B}\left([0,T]\right) \right)$, i.e., 
$\PdM (C)= E\left[\int_0^T 1_C(\cdot,s) d[M]_s\right]$ 
for $C\in \cF \otimes \mathcal{B}\left([0,T]\right)$.

To model illiquidity, we need to specify three progressively measurable processes
$\mu=(\mu_s)_{s\in [0,T]}$,
$\sigma=(\sigma_s)_{s\in [0,T]}$
and $\rho=(\rho_s)_{s\in[0,T]}$
such that
$\int_0^T(|\rho_s|+|\mu_s|+\sigma^2_s)\,d[M]_s<\infty$ a.s. 
The first two inputs $\mu$ and $\sigma$ define the
\emph{price impact process}
$\gamma=(\gamma_s)_{s\in[0,T]}$,
which is a positive continuous adapted process 
satisfying~\eqref{eq:dyn_gamma}, i.e., 
$$
\gamma_s=\gamma_0\exp\left\{\int_0^s\left(\mu_u-\frac{\sigma^2_u}2\right)\,d[M]_u
+\int_0^s \sigma_u\,dM_u\right\},\quad s\in[0,T],
$$
where $\gamma_0>0$ is a positive $\cF_0$-measurable random variable.
It turns out to be useful also to introduce the process
$\alpha_s=\frac{1}{\gamma_s}$, $s\in [0,T]$,
which then satisfies
\begin{equation}\label{eq:dyn_alpha}
d\alpha_s = \alpha_s\left( -(\mu_s-\sigma_s^2)\,d[M]_s - \sigma_s\,dM_s \right),\quad s\in [0,T].
\end{equation}
The third input $\rho=(\rho_s)_{s\in [0,T]}$ above
is called the \emph{resilience process}
and is, together with the price impact process $\gamma$,
involved in the following dynamics.

Given $x\in\bbR$ and $t\in[0,T]$,
for any $d\in \R$ and any execution strategy $X=(X_s)_{s\in[t,T]}$
with initial position $X_{t-}=x$ (cf.\ the introduction), 
we define the \emph{deviation process}
$D=(D_s)_{s\in [t,T]}$ associated to $X$ as the unique solution 
of~\eqref{eq:deviation_dyn}, i.e., 
\begin{equation}\label{eq:deviation_def}
\begin{split}
D_s &= e^{-\int_t^s \rho_u d[M]_u} \left( d+\int_{[t,s]} e^{\int_t^r \rho_u d[M]_u} \gamma_r dX_r + \int_{[t,s]} e^{\int_t^r \rho_u d[M]_u} d[\gamma,X]_r \right), \quad s\in[t,T], \\
D_{t-} &= d.
\end{split}
\end{equation} 
For each $t\in [0,T]$, we formulate the conditions 
\begin{enumerate}
\item[(A1)] \qquad $E_t\left[ \sup_{s \in [t,T]} \left( \gamma_s^2 (X_s - \alpha_s D_s)^4 \right) \right] < \infty$ a.s.,
\item[(A2)] \qquad $E_t\left[  \left( \int_t^T \gamma_s^2 (X_s - \alpha_s D_s )^4 \sigma_s^2 d[M]_s \right)^{\frac12}  \right] < \infty$ a.s.,
\item[(A3)] \qquad $E_t\left[ \left( \int_t^T D_{s}^4 \alpha_s^2 \sigma_s^2 d[M]_s \right)^{\frac12} \right] < \infty$ a.s.,
\end{enumerate}
where $E_t[\cdot]$ is a shorthand notation for $E[\cdot|\cF_t]$.
Note that if $E_t\left[ \int_t^T \sigma_s^2 d[M]_s \right] < \infty$ a.s., then,
by the Cauchy-Schwarz inequality, (A2) follows from~(A1). 

For $x,d\in \R$ and $t\in [0,T]$, let $\cA_t(x,d)$ be the set of all
c\`adl\`ag semimartingales $X=(X_s)_{s\in [t,T]}$ with $X_{t-}=x$, $X_T=0$
(i.e., execution strategies)
satisfying conditions (A1), (A2) and~(A3).

For every $x,d \in \R$, $t\in [0,T]$ and $X \in \cA_t(x,d)$, 
we define 
the cost functional $J$ by~\eqref{eq:cost_fct}. 
Conditions under which the cost functional is well-defined
for all $x,d \in \R$, $t\in [0,T]$ and $X \in \cA_t(x,d)$
are provided in Theorem~\ref{thm:quadratic_exp_J} below. 
The control problem considered in this paper is
to minimize the cost functional over
the execution strategies $X\in\cA_t(x,d)$. 
The value function $V$ of the control problem is given by~\eqref{eq:def_value_fct}. 
If, for $x,d \in \R$, $t\in [0,T]$, there exists an execution strategy $X^*=(X^*_s)_{s\in [t,T]} \in \cA_t(x,d)$ such that $V_t(x,d)=J_t(x,d,X^*)$, we call this process $X^*$ an optimal execution strategy.

\begin{remark}
(i) An important special case of our setting is the situation where the continuous local martingale $M$ is an $(\cF_s)$-Brownian motion $W=(W_s)_{s\in[0,T]}$,
in which case we have $d[M]_s=ds$ and $\cD_W=P\times\Leb$ ($\Leb$ denotes the Lebesgue measure).

\smallskip
(ii) More generally, let $M$ be a continuous local martingale satisfying $ds \ll d[M]_s$ a.s.
Consider the situation when the dynamics of the price impact process $\gamma$ is given by
$d\gamma_s=\gamma_s\left( \wt{\mu}_s ds + \sigma_s dM_s \right)$, $s\in[0,T]$,
and, for $x,d \in \mathbb{R}$, $t\in[0,T]$ and $X\in\mathcal{A}_t(x,d)$,
the dynamics of the deviation process $D$ is specified by
$dD_s=-\wt{\rho}_s D_{s-} ds +\gamma_s dX_s+d[\gamma,X]_s$, $s\in[t,T]$, $D_{t-}=d$.
Then, the identifications
$\wt{\mu}\lambda=\mu$ and $\wt{\rho}\lambda=\rho$
with $\lambda_s = \frac{ds}{d[M]_s}$, $s\in[0,T]$,
reduce this situation to our formulation
(cf.\ \eqref{eq:dyn_gamma} and~\eqref{eq:deviation_dyn}).
\end{remark}

\begin{remark}\label{rem:01042021a1}
In the problem setting introduced above we focused on the price deviation only.
However, the considerations above also allow to explicitly include an unaffected price into the picture, provided that the unaffected price is a (local) martingale.

To this end, assume that the unaffected price is modelled via a c\`adl\`ag local martingale $S^0=(S^0_u)_{u\in[0,T]}$.
Fix an initial time instance $t\in[0,T]$, initial position $x\in\bbR$ and initial deviation $d\in\bbR$.
Consider an execution strategy $X=(X_u)_{u\in[t,T]}$, i.e., a c\`adl\`ag semimartingale satisfying $X_{t-}=x$ and $X_T=0$.
When we take the unaffected price into account, the execution costs generated by $X$ over $[t,T]$ are given by the formula
\begin{equation}\label{eq:02042021a1}
\int_{[t,T]}
S^0_{u-}\,dX_u
+\int_{[t,T]}
d[S^0,X]_u
+\int_{[t,T]}
D_{u-}\,dX_u
+\int_{[t,T]}
\frac{\gamma_u}2\,d[X]_u.
\end{equation}
The third and the fourth terms were extensively discussed above.
The first and the second cost terms in~\eqref{eq:02042021a1} are due to the unaffected price process $S^0$.
It was first observed in \cite{lorenz2013drift} via a limiting argument from discrete time that, in continuous time and for semimartingale strategies, the right expression for the cost terms due to the unaffected price is
\begin{equation}\label{eq:02042021a2}
\int_{[t,T]}
S^0_{u-}\,dX_u
+\int_{[t,T]}
d[S^0,X]_u
\end{equation}
(see Lemma~2.5 in \cite{lorenz2013drift}).\footnote{We notice that in the literature preceding \cite{lorenz2013drift} the execution strategies $X$ were always assumed to be of finite variation (often just monotone), while the part of execution costs coming from the unaffected price was given by the expression $\int_{[t,T]}S^0_u\,dX_u$. This is consistent with~\eqref{eq:02042021a2}, as
$
\int_{[t,T]}
S^0_{u-}\,dX_u
+\int_{[t,T]}
d[S^0,X]_u
=\int_{[t,T]}S^0_u\,dX_u
$
whenever $X$ is of finite variation
(see Proposition~I.4.49a in \cite{jacodshiryaev}).}
Using integration by parts for the semimartingales $X$ and $S^0$ together with $X_{t-}=x$ and $X_T=0$ we obtain that the expression in~\eqref{eq:02042021a2} equals\footnote{In the case $t=0$, the term $S^0_{0-}$ appearing in intermediate calculations in~\eqref{eq:02042021a3} can be an arbitrary $\cF_0$-measurable random variable (i.e, we can allow a jump at time $0$ in $S^0$ like we allow initial jumps in~$X$). Interestingly, we do not need any specific relation between $S^0_0$ and $S^0_{0-}$, as the term containing $S^0_{t-}$ in~\eqref{eq:02042021a3} is ultimately reduced with the initial jump in the stochastic integral.}
\begin{equation}\label{eq:02042021a3}
\begin{split}
X_TS^0_T-X_{t-}S^0_{t-}-\int_{[t,T]}X_{u-}\,dS^0_u
&=-xS^0_{t-}-\int_{[t,T]}X_{u-}\,dS^0_u\\
&=-xS^0_t-\int_{(t,T]}X_{u-}\,dS^0_u.
\end{split}
\end{equation}
It follows from the Burkholder-Davis-Gundy inequality that
$
E_t\left[\int_{(t,T]}X_{u-}\,dS^0_u\right]=0
$
whenever the condition
\begin{enumerate}
\item[(A4)] \qquad $E_t\left[ \left( \int_{(t,T]} X_{u-}^2\,d[S^0]_u \right)^{\frac12} \right] < \infty$ a.s.
\end{enumerate}
is satisfied. In particular, under~(A4), the expected (at time~$t$) costs due to the unaffected price are equal to $-xS^0_t$ and thus do not depend on the execution strategy, hence the minimization of the expected (at time~$t$) total costs in~\eqref{eq:02042021a1} reduces to the minimization of $J_t(x,d,X)$.

We summarize the discussion as follows.
In the paper we minimize $J_t(x,d,X)$ over strategies $X$ satisfying (A1)--(A3) (i.e., over $X\in\cA_t(x,d)$).
Given a local martingale unaffected price $S^0$,
a pertinent optimization problem is to minimize $J_t(x,d,X)$ over strategies $X$ satisfying (A1)--(A4).
Given an optimal strategy $X^*\in\cA_t(x,d)$ one thus needs additionally to examine whether $X^*$ satisfies~(A4).
Observe that (A4) need not be satisfied automatically, as it depends on the additional input $S^0$, which may have nothing to do with our other inputs (namely, the processes $M$, $\rho$, $\mu$ and~$\sigma$).
However, it is worth noting that, if $S^0$ is a square integrable martingale,
then, under the assumptions of Theorem~\ref{thm:sol_val_fct} below, the optimal strategy $X^*\in\cA_t(x,d)$ provided in~\eqref{eq:opt_strat_representation} satisfies~(A4).\footnote{Indeed, we will see in the proof that, under the assumptions of Theorem~\ref{thm:sol_val_fct}, for $X^*$ of~\eqref{eq:opt_strat_representation},
it holds that $E_t\left[\sup_{u\in[t,T]}(X^*_{u-})^2\right]<\infty$ a.s.
As $S^0$ is a square integrable martingale, we have $E\left([S^0]_T\right)<\infty$, hence
$E_t\left([S^0]_T-[S^0]_t\right)<\infty$ a.s.
Condition~(A4) for $X^*$ of~\eqref{eq:opt_strat_representation}
now follows from the Cauchy-Schwarz inequality.}
\end{remark}

\section{Main results}\label{sec:main_results}

We now present the main results,
which include an alternative representation of the cost functional,
a representation of the value function in terms of solutions to a certain BSDE,
a characterization of existence of an optimal strategy
and an explicit expression for the optimal strategy (when it exists).
The proofs are deferred to Section~\ref{sec:proofs_of_main_results}.

We often make use of the following assumption: 
\begin{enumerate}
\item[$\bm{\left(C_{>0}\right)}$]  \qquad $2\rho + \mu - \sigma^2>0$ $\PdM$-a.e. 
\end{enumerate}
In the case when $\bm{\left(C_{>0}\right)}$ is satisfied, we introduce the BSDE 
\begin{equation}\label{eq:bsde}
\begin{split}
Y_t & = \frac12 + \int_t^T f(s,Y_s,Z_s) d[M]_s - \int_t^T Z_s dM_s - \left( M^\perp_T - M^\perp_t\right), \quad t\in[0,T], 
\end{split}
\end{equation}
with the driver 
\begin{equation}\label{eq:driver_of_bsde}
f(s,Y_s,Z_s) = - \frac{\left( (\rho_s + \mu_s) Y_s + \sigma_s Z_s \right)^2}{ \sigma_s^2 Y_s + \frac12 (2\rho_s + \mu_s -\sigma_s^2) } 
+ \mu_s Y_s + \sigma_s Z_s
\end{equation}
and terminal condition $\frac12$. 
A solution of BSDE~\eqref{eq:bsde} is a triple $(Y,Z,M^\perp)$ of processes where 
\begin{itemize}
	\item $M^\perp$ is a c\`adl\`ag local martingale with $M^\perp_0 = 0$ and $[M^\perp,M]=0$,
	\item $Z$ is a progressively measurable process such that $\int_0^T Z_s^2 d[M]_s < \infty$ a.s.,
	\item $Y$ is an adapted c\`adl\`ag process, 
\end{itemize}	
such that \eqref{eq:bsde} is satisfied a.s. 
Notice that $Y$ is necessarily a \emph{special} semimartingale (see Section~I.4c in \cite{jacodshiryaev}).
We now introduce the assumption 
\begin{enumerate}
	\item[$\bm{\left(C_{\textbf{BSDE}}\right)}$]  
	There exists a solution $(Y,Z,M^\perp)$ of BSDE~\eqref{eq:bsde} such that $Y$ is $[0,1/2]$-valued, $E\left[ [M^\perp]_T \right] < \infty$  and $E\left[ \int_0^T Z_s^2 d[M]_s \right] < \infty$.
\end{enumerate}
To explain the role of condition $\bm{\left(C_{>0}\right)}$ for BSDE~\eqref{eq:bsde} and in $\bm{\left(C_{\textbf{BSDE}}\right)}$ it is worth noting that, under $\bm{\left(C_{>0}\right)}$, the denominator in the first term in \eqref{eq:driver_of_bsde} stays strictly positive whenever $Y$ is nonnegative. 
Furthermore, we make the following comments: 
\begin{itemize}
	\item In many situations below (Proposition~\ref{propo:withoutresiliencecloseimmediately},  Sections~\ref{sec:recoveringOW}, \ref{sec:examplewithLambertWfct} and \ref{sec:furtherexamples}) $\bm{\left(C_{\textbf{BSDE}}\right)}$ is satisfied.
	\item Two broad subsettings of our general setting where $\bm{\left(C_{\textbf{BSDE}}\right)}$ is satisfied are described in Section~\ref{sec:analysis_of_BSDE}. 
	\item In our general setting $\bm{\left(C_{\textbf{BSDE}}\right)}$ is motivated by the discrete-time version of the stochastic control problem~\eqref{eq:def_value_fct} (see Appendix~\ref{sec:formalderivationbsde}). 
\end{itemize}

In Remark~\ref{rem:interpretation_Y} below we present an interpretation of the solution component $Y$ of~\eqref{eq:bsde} as a saving factor describing the benefits of using an optimal execution strategy compared to an immediate position closure.

If $\bm{\left(C_{>0}\right)}$ and $\bm{\left(C_{\textbf{BSDE}}\right)}$ hold we define the process $\wt{\beta} = (\wt{\beta}_s)_{s\in [0,T]}$ 
pertaining to $(Y,Z)$ from $\bm{\left(C_{\textbf{BSDE}}\right)}$ 
by 
\begin{equation}\label{eq:def_beta_tilde}
\wt{\beta}_s = \frac{(\rho_s+\mu_s)Y_s+\sigma_s Z_s}{\sigma_s^2 Y_s + \frac12 (2\rho_s + \mu_s - \sigma_s^2 )}, \quad s\in [0,T].
\end{equation}

\begin{theo}\label{thm:quadratic_exp_J}
	Let $\bm{\left(C_{>0}\right)}$ and $\bm{\left(C_{\textbf{BSDE}}\right)}$ be satisfied. For all $x,d\in \R$, $t\in [0,T]$ and $X\in \cA_t(x,d)$ it holds that the cost functional \eqref{eq:cost_fct} is well-defined and admits the representation 
	\begin{equation}\label{eq:quad_expansion_J}
	\begin{split}
	&J_t(x,d,X)=\frac{Y_{t}}{\gamma_t}\left(d-\gamma_t x\right)^2-\frac{d^2}{2\gamma_t}\\
	&\quad+E_t\left[\int_t^T  \frac{1}{\gamma_s} 
	\left( \wt{\beta}_s (\gamma_s X_s - D_s ) + D_s  \right)^2
	\left( \sigma_s^2 Y_s + \frac12 (2\rho_s + \mu_s - \sigma_s^2 ) \right)
	d[M]_s \right]\;\;\text{a.s.}
	\end{split}
	\end{equation}
	In particular, for all $x,d\in \R$ and $t\in [0,T]$ it holds 
	\begin{equation}\label{eq:lower_bound_for_value_fct}
	V_t(x,d)\geq \frac{Y_{t}}{\gamma_t}\left(d-\gamma_t x\right)^2-\frac{d^2}{2\gamma_t}\;\;\text{a.s.}
	\end{equation}
\end{theo}

\begin{remark}
	Note that $\bm{\left(C_{\textbf{BSDE}}\right)}$ does only postulate existence but not uniqueness of a solution triple $(Y,Z,M^\perp)$ of  BSDE~\eqref{eq:bsde}. By assuming $\bm{\left(C_{\textbf{BSDE}}\right)}$ we always mean that we fix some solution triple $(Y,Z,M^\perp)$ of  BSDE~\eqref{eq:bsde} that satisfies the properties in $\bm{\left(C_{\textbf{BSDE}}\right)}$ and we use this solution in all subsequent statements. 
	In particular, $\wt \beta$ is then understood as the process defined by \eqref{eq:def_beta_tilde} using $(Y,Z)$ from this solution. 
	Observe that \eqref{eq:quad_expansion_J} and \eqref{eq:lower_bound_for_value_fct} hold for any such solution $(Y,Z,M^\perp)$. 
	Using the first part of Theorem~\ref{thm:sol_val_fct} we provide in Proposition~\ref{propo:bsdeuniquenessfrommaintheo} a uniqueness result for BSDE~\eqref{eq:bsde}.	 
\end{remark}

For $t\in [0,T]$ we use the notation $\PdM|_{[t,T]}$ for the restriction of the Dol{\'e}ans measure $\PdM$ to $\left( \Omega\times [t,T], \cF \otimes \mathcal{B}\left( [t,T] \right) \right)$. 

We proceed with describing the solution to our optimization problem~\eqref{eq:def_value_fct}. 
The case $x=\frac{d}{\gamma_t}$ is easy, and we treat it first:

\begin{lemma}\label{lem:optstratifzero}
	Let $\bm{\left(C_{>0}\right)}$ and $\bm{\left(C_{\textbf{BSDE}}\right)}$ be satisfied. 
	Suppose that $t\in [0,T]$ and $x,d \in \R$ with  $x=\frac{d}{\gamma_t}$. 
	Then, the value function is $V_t(x,d)= -\frac{d^2}{2\gamma_t}$, 
	and the strategy 
	$X^*=(X^*_s)_{s\in [t,T]}$ defined by $X^*_{t-}=x$, $X^*_s = 0$, $s \in [t,T]$, 
	which closes the position immediately, is optimal in $\mathcal A_t(x,d)$. 
	Moreover, this optimal strategy is unique up to $\PdM|_{[t,T]}$-null sets.
\end{lemma}

In order to describe the solution beyond the case $x=\frac{d}{\gamma_t}$, 
we introduce the condition 
\begin{enumerate}
\item[$\bm{\left(C_{[M]}\right)}$]  \qquad for all $c\in(0,\infty)\colon$ $E\left[ \exp( c \ [M]_T ) \right] < \infty$ .
\end{enumerate}  
Note that if $M=W$ is an $\left(\cF_s\right)$-Brownian motion, then $\bm{\left(C_{[M]}\right)}$ is trivially satisfied.

For a continuous semimartingale $Q=(Q_s)_{s\in[0,T]}$ 
we denote by $\mathcal{E}(Q)=\left( \mathcal{E}(Q)_s \right)_{s\in [0,T]}$ its stochastic exponential, i.e., 
$\mathcal{E}(Q)_s= \exp\left( Q_s - Q_0 - \frac12 [Q]_s \right),$ $s\in[0,T]$. 
For $t\leq s$ in $[0,T]$ we also use the notation 
$\mathcal{E}(Q)_{t,s}  = \frac{\mathcal{E}(Q)_s}{\mathcal{E}(Q)_t}
= \exp\left( Q_s - Q_t - \frac12 \left( [Q]_s - [Q]_t \right) \right)$.

\begin{theo}\label{thm:sol_val_fct}
Let $\bm{\left(C_{>0}\right)}$, $\bm{\left(C_{\textbf{BSDE}}\right)}$ and $\bm{\left(C_{[M]}\right)}$ be satisfied. 
Suppose furthermore that $\rho$, $\mu$ and $\wt{\beta}$ (defined by \eqref{eq:def_beta_tilde}) are $\PdM$-a.e. bounded. 

\noindent
\begin{enumerate}
\item[(i)] (Representation of the value function) 
	
	For all $x,d \in \R$ and $t\in[0,T]$ it holds 
	\begin{equation*}
	V_t(x,d)= \frac{Y_{t}}{\gamma_t}\left(d-\gamma_t x\right)^2-\frac{d^2}{2\gamma_t} \;\;\text{a.s.}
	\end{equation*}
	
\item[(ii)]
(Characterization for the existence of the optimizer)
	
	Let $x,d \in \R$ and assume $x\neq \frac{d}{\gamma_0}$. Then
	there exists an optimal strategy $X^*=(X^*_s)_{s\in [0,T]}\in \mathcal A_0(x,d)$ if and only if there exists a c\`adl\`ag semimartingale $\beta=(\beta_s)_{s\in [0,T]}$ such that $\wt{\beta} = \beta$ $\PdM$-a.e. 
	
	In this case, the optimal strategy is unique up to $\PdM$-null sets.

\item[(iii)]
(Representations for the optimal strategy and deviation process when the optimizer exists)

Consider the case that there exists a c\`adl\`ag semimartingale $\beta=(\beta_s)_{s\in [0,T]}$ such that $\wt{\beta} = \beta$ $\PdM$-a.e. Define
\begin{equation}\label{eq:01052020a1}
	Q_s = - \int_0^s \beta_r \sigma_r dM_r - \int_0^s \beta_r (\mu_r + \rho_r - \sigma_r^2) d[M]_r, \quad s\in [0,T].
\end{equation}
Let $x,d \in \R$ and $t\in [0,T]$.
Then the optimal strategy
$\left( X^*_s \right)_{s \in [t,T]} \in \cA_t(x,d)$
and the associated deviation process
$(D^*_s)_{s\in[t,T]}$
(both unique up to $\PdM|_{[t,T]}$-null sets)
are given by the formulas $X^*_{t-}=x$, $D^*_{t-}=d$,
\begin{align}
X^*_s
&=
\left(x-\frac{d}{\gamma_t}\right)\cE(Q)_{t,s}\,(1-\beta_s),
\quad s\in [t,T),
\label{eq:opt_strat_representation}\\
D^*_s
&=
\left(x-\frac{d}{\gamma_t}\right)\cE(Q)_{t,s}\,(-\gamma_s\beta_s),
\quad s\in [t,T),
\label{eq:opt_dev_representation}
\end{align}
and $X^*_T=0$,
$D^*_T=\left(x-\frac{d}{\gamma_t}\right)\cE(Q)_{t,T}\,(-\gamma_T)$.
\end{enumerate}
\end{theo}

\begin{remark}\label{rem:interpretation_Y}
(a) \emph{(Economic interpretation of the process~$Y$)}
Given a unit open position $x=1$ and a deviation $d=0$ at time $t\in [0,T]$, a possible strategy is to close the position immediately at time $t$ and keep it closed in the remaining period (i.e., $X_s=0$ for all $s\in [t,T]$). The cost associated to this strategy is given by $J_t(1,0,X)=\gamma_t/2$. Theorem~\ref{thm:sol_val_fct} shows that under appropriate assumptions the minimal costs in this situation are given by $V_t(1,0)=\gamma_tY_t$. Therefore, we obtain for all $t\in [0,T]$ that $2Y_t=V_t(1,0)/J_t(1,0,X)$ and thus the random variable $2Y_t\colon \Omega \to [0,1]$ describes to which percentage the costs of selling one unit immediately at time $t$ can be reduced by executing the position optimally.

\smallskip
(b) \emph{(On the boundedness assumptions in Theorem~\ref{thm:sol_val_fct})}
In addition to $\bm{\left(C_{>0}\right)}$ and $\bm{\left(C_{\textbf{BSDE}}\right)}$, which are already assumed in Theorem~\ref{thm:quadratic_exp_J},
we need $\bm{\left(C_{[M]}\right)}$ and the boundedness of $\rho$, $\mu$ (hence, due to $\bm{\left(C_{>0}\right)}$, also $\sigma^2$) and $\wt\beta$ in Theorem~\ref{thm:sol_val_fct}.
These are strong sufficient conditions for the validity of Theorem~\ref{thm:sol_val_fct} and can be replaced by appropriate integrability assumptions.
For instance, a straightforward generalization of part~(iii) obtained along the lines of our proofs (cf.\ Remarks \ref{rem:26032021a1} and~\ref{rem:26032021a2} below) is as follows: Assume
\begin{enumerate}
\item\label{it:26032021a4}
$\bm{\left(C_{>0}\right)}$,
$\bm{\left(C_{\textbf{BSDE}}\right)}$,

\item\label{it:26032021a3}
the existence of a c\`adl\`ag semimartingale $\beta=(\beta_s)_{s\in[0,T]}$ such that $\wt\beta=\beta$ $\PdM$-a.e.,

\item\label{it:26032021a1}
$E\left[\int_0^T\beta_s^4\sigma_s^2\,d[M]_s\right]<\infty$,

\item\label{it:26032021a2}
$E\left[ \exp\left\{ c \int_0^T\chi_s\,d[M]_s \right\}\right] < \infty$
for all $c>0$,
where $\chi=\sigma^2+\beta^2\sigma^2+(\mu-2\beta\mu-2\beta\rho)^+$.
\end{enumerate}
Then the value function is given by the formula in part~(i) of Theorem~\ref{thm:sol_val_fct},
and the unique optimal strategy and the associated deviation process are given by the formulas in part~(iii) of Theorem~\ref{thm:sol_val_fct}.
Notice that the integrability conditions \ref{it:26032021a1} and~\ref{it:26032021a2} above are satisfied whenever we assume $\bm{\left(C_{[M]}\right)}$ and the boundedness of $\rho$, $\mu$ and $\wt\beta$.

It is possible also to get the remaining messages of Theorem~\ref{thm:sol_val_fct}
(namely, part~(ii) and, in the case when condition~\ref{it:26032021a3} above is not satisfied, part~(i))
under certain integrability-type conditions instead of the boundedness-type ones. Such integrability conditions, however, look more cumbersome than \ref{it:26032021a1} and~\ref{it:26032021a2} above.
 Given that in all examples (with some novel effects) discussed below 
the boundedness assumptions are satisfied, we formulate Theorem~\ref{thm:sol_val_fct} under the boundedness assumptions and restrain from further technical
discussions of this point.
\end{remark}

We now observe that the optimal strategy and the optimal deviation process are dynamically consistent.

\begin{corollary}[Dynamic consistency when the optimizer exists]\label{cor:07052020a1}
Under the assumptions of Theorem~\ref{thm:sol_val_fct}
consider the case that there exists a c\`adl\`ag semimartingale $\beta=(\beta_s)_{s\in [0,T]}$ such that $\wt{\beta} = \beta$ $\PdM$-a.e. Define the process $Q$ as in~\eqref{eq:01052020a1}.
Let $x,d \in \R$ and $t\in [0,T]$.
Then, for the optimal strategy and deviation process given in
\eqref{eq:opt_strat_representation}--\eqref{eq:opt_dev_representation}
and for any $u\in(t,T)$, we have
\begin{align*}
X^*_s
&=
\left(X^*_{u-}-\frac{D^*_{u-}}{\gamma_u}\right)\cE(Q)_{u,s}\,(1-\beta_s),
\quad s\in [u,T),
\\
D^*_s
&=
\left(X^*_{u-}-\frac{D^*_{u-}}{\gamma_u}\right)\cE(Q)_{u,s}\,(-\gamma_s\beta_s),
\quad s\in [u,T),
\end{align*}
and $X^*_T=0$,
$D^*_T=\left(X^*_{u-}-\frac{D^*_{u-}}{\gamma_u}\right)\cE(Q)_{u,T}\,(-\gamma_T)$.
\end{corollary}

In the case of vanishing resilience it turns out that it is always optimal to close the position immediately and then stay inactive. 
We formally state this as 

\begin{propo}\label{propo:withoutresiliencecloseimmediately}
	Assume $\bm{\left(C_{>0}\right)}$ and $\rho \equiv 0$. 
	Then, for all $t\in [0,T]$ and $x,d \in \R$, 
	we have $V_t(x,d) = -x\left( d-\frac{\gamma_t}{2}x \right)$ 
	and there exists a unique (up to $\PdM|_{[t,T]}$-null sets) optimal strategy $\left( X^*_s \right)_{s \in [t,T]} \in \cA_t(x,d)$ and it is given by $X^*_{t-}=x$, $X^*_s=0$, $s\in [t,T]$.
\end{propo}

It is worth noting that we need not assume $\bm{\left(C_{\textbf{BSDE}}\right)}$ and $\bm{\left(C_{[M]}\right)}$ 
in Proposition~\ref{propo:withoutresiliencecloseimmediately}.  
We will see, however, that $\bm{\left(C_{\textbf{BSDE}}\right)}$ is always satisfied in this case.

We close the section with a uniqueness result for BSDE~\eqref{eq:bsde}.

\begin{propo}\label{propo:bsdeuniquenessfrommaintheo}
Assume $\bm{\left(C_{>0}\right)}$, $\bm{\left(C_{[M]}\right)}$ and that $\rho$, $\mu$ are $\PdM$-a.e.\ bounded. 
Let $\left( Y^{(i)}, Z^{(i)}, M^{\perp, (i)} \right)$, $i=1,2$, be solutions of BSDE~\eqref{eq:bsde} 
such that $Y^{(i)}$ are $[0,1/2]$-valued, $E\left[ [M^{\perp, (i)}_T] \right] < \infty$ and $E\left[ \int_0^T \left( Z_s^{(i)} \right)^2 d[M]_s \right] < \infty$, $i=1,2$ (cf.\ $\bm{\left(C_{\textbf{BSDE}}\right)}$), and such that the processes 
$\wt{\beta}^{(i)} = (\wt{\beta}^{(i)}_s)_{s\in [0,T]}$ defined by~\eqref{eq:def_beta_tilde} for $(Y^{(i)},Z^{(i)})$, $i=1,2$,  i.e., 
\begin{equation*}
	\wt{\beta}^{(i)}_s = \frac{(\rho_s+\mu_s)Y^{(i)}_s+\sigma_s Z^{(i)}_s}{\sigma_s^2 Y^{(i)}_s + \frac12 (2\rho_s + \mu_s - \sigma_s^2 )},  \quad s\in [0,T], i=1,2,
\end{equation*}  
are $\PdM$-a.e. bounded. 
Then: 
\begin{itemize}
	\item $Y^{(1)}$ and $Y^{(2)}$ are indistinguishable, 
	\item $Z^{(1)}=Z^{(2)}$ $\PdM$-a.e., 
	\item $M^{\perp, (1)}$ and $M^{\perp, (2)}$ are indistinguishable.
\end{itemize}
\end{propo}

\section{Reason for adjusting the cost functional}\label{sec:reason_cost_fct}

\subsection{Counterexample}\label{sec:fromsec52offirstversion}

In this subsection we show that minimizing the cost functional  
\begin{equation}\label{eq:cost_fct_tilde}
\wt J_t(x,d,X)=E_t\left[\int_{[t,T]}\left(D_{s-}+\frac{\gamma_s}{2}\Delta X_s\right)	dX_s\right], \quad 
x,d \in\mathbb{R}, \; t\in[0,T], \; X\in \mathcal{A}_t(x,d), 
\end{equation}  
(cf.\ \eqref{eq:13022020a2_intro} above or formula~(5) in \cite{fruth2019optimal}) over $\mathcal{A}_t(x,d)$ might constitute an ill-posed problem. 
More precisely, we construct an example, where 
$\essinf_{X \in \cA_t(x,d)} \wt J_t(x,d,X) = - \infty $.

Consider the setting, where $M=W$ is an $(\cF_s)$-Brownian motion and
the price impact $\gamma>0$ and the resilience $\rho>0$ are positive deterministic constants
(that is, $\mu =\sigma \equiv 0$ in terms of our model parameters).

We consider the starting time $t=0$ and assume (for simplicity) that the $\sigma$-field $\mathcal{F}_0$ is trivial.
As $\gamma$ is constant, for all $X\in \cA_0(x,d)$ the associated deviation process $D$ satisfies 
$$
dD_s = -\rho_sD_{s}\,ds + \gamma_s\,dX_s +d[\gamma,X]_s = -\rho D_{s}\,ds + \gamma\,dX_s,\quad s\in[0,T].
$$
In particular, in this setting the dynamics of $D$ is of type~\eqref{eq:13022020a1_intro} (cf.\ formula~(2) in \cite{fruth2019optimal}).

Fix the initial position $x=0$ and the initial deviation $d=0$.
For $\nu \in \bbR$ consider the execution strategy $X^{(\nu)}=( X_s^{(\nu)} )_{s\in[0,T]}$ defined by $X^{(\nu)}_{0-}=X^{(\nu)}_{0}=0$, $X^{(\nu)}_s=\nu W_s$ for $s\in[0,T)$ and $X^{(\nu)}_T=0$,
i.e., the strategy follows a scaled Brownian motion
on $[0,T)$ and has a block trade at time $T$.
For each $\nu \in \bbR$, let $D^{(\nu)}=( D_s^{(\nu)} )_{s\in[0,T]}$ be the deviation process associated to $X^{(\nu)}$, i.e., $d D^{(\nu)}_s = -\rho D^{(\nu)}_{s}\,ds + \gamma \nu \,dW_s$, $s\in[0,T]$, and $D^{(\nu)}_{0-}=0$. Note that $D^{(\nu)}$ is an Ornstein-Uhlenbeck process. 
One can therefore show that (A1) is satisfied, and due to $\sigma\equiv0$, (A2) and (A3) are satisfied as well, thus $X^{(\nu)}\in\cA_0(0,0)$ for all $\nu \in \bbR$. 
Since $\int_0^{\cdot} D^{(\nu)}_s dW_s$ is a martingale and $X^{(\nu)}$ is continuous on $(0,T)$ and $\Delta X^{(\nu)}_T=-X^{(\nu)}_{T-}$, it holds for all $\nu \in \bbR$ that 
\begin{equation*}
\begin{split}
& \wt J_0(0,0,X^{(\nu)})  
= E\left[ \int_{[0,T)} D^{(\nu)}_{s-} dX^{(\nu)}_s + D^{(\nu)}_{T-} \Delta X^{(\nu)}_T + \frac{\gamma}{2} \left( \Delta X^{(\nu)}_T \right)^2 \right] \\
& \quad  = E\left[ \nu \int_0^T D^{(\nu)}_s dW_s - D^{(\nu)}_{T-} X^{(\nu)}_{T-} + \frac{\gamma}{2} \nu^2 W_T^2 \right] 
= - E\left[ D^{(\nu)}_{T-} X^{(\nu)}_{T-} \right] + \frac{\gamma}{2} \nu^2 T .
\end{split}
\end{equation*} 
We have that 
\begin{equation*}
d (D^{(\nu)} X^{(\nu)})_s = \nu D^{(\nu)}_s dW_s - \rho X^{(\nu)}_s D^{(\nu)}_s ds + \gamma \nu^2 W_s dW_s + \gamma \nu^2 ds, \quad s\in [0,T), 
\end{equation*}  
and hence 
\begin{equation*}
E\left[ X^{(\nu)}_s D^{(\nu)}_s \right] = - \rho \int_0^s E\left[ X^{(\nu)}_u D^{(\nu)}_u \right] du + \gamma \nu^2 s , \quad s\in [0,T).
\end{equation*}
It follows that  
\begin{equation*}
E\left[ X^{(\nu)}_s D^{(\nu)}_s \right] = \frac{\gamma \nu^2}{\rho} \left( 1 - e^{-\rho s} \right) , \quad s \in [0,T).
\end{equation*}
Therefore, we obtain for all $\nu \in \bbR$
\begin{equation*}
\wt J_0(0,0,X^{(\nu)})  = -  \frac{\gamma \nu^2}{\rho} \left( 1 - e^{-\rho T} \right)  + \frac{\gamma}{2} \nu^2 T 
= \frac{\gamma \nu^2}{\rho} \left( e^{-\rho T} - 1 + \frac{\rho T}{2} \right).
\end{equation*} 
Now we see that, if $\rho>0$ is chosen in the way that
$e^{-\rho T} - 1 + \frac{\rho T}{2} <0$
(it is enough to take $\rho \in (0,1/T)$), then
\begin{equation*}
\wt J_0(0,0,X^{(\nu)}) \to - \infty \quad \text{ as } |\nu| \to \infty .
\end{equation*}
Thus, the cost functional $\wt J$ leads to an ill-posed optimization problem.

\subsection{Solution in our framework}\label{sec:recoveringOW}

In the setting above (see the second paragraph in Section~\ref{sec:fromsec52offirstversion}),
we recover a well-posed optimization problem when we use the cost functional~\eqref{eq:cost_fct}.
Let us verify that Theorem~\ref{thm:sol_val_fct} applies and present an explicit formula for the optimal strategy in $\cA_t(x,d)$,
for any $t\in[0,T]$, $x,d\in\bbR$.

In this setting, $\bm{\left(C_{>0}\right)}$ and $\bm{\left(C_{[M]}\right)}$ are trivially satisfied. BSDE~\eqref{eq:bsde} takes the form
\begin{equation}\label{eq:ODEmu0sigma0rhoconst}
dY_s  = \rho Y_s^2 ds + Z_sdW_s +dM^\perp_s,\quad  s \in [0,T], \quad 
Y_{T}  =\frac{1}{2}.
\end{equation} 
The solution of \eqref{eq:ODEmu0sigma0rhoconst} is 
\begin{equation}
Z\equiv 0, \quad M^\perp \equiv 0, \quad Y_s = \frac{1}{2+(T-s)\rho}, \quad  s \in [0,T].
\end{equation}
Observe that $\wt{\beta}=Y$ (see \eqref{eq:def_beta_tilde} for the definition of $\wt{\beta}$) in this setting and that $Y$ is deterministic, increasing, continuous and $(0,1/2]$-valued. 
In particular, $\wt{\beta}$ is bounded, and it is a semimartingale.
Hence, Theorem~\ref{thm:sol_val_fct} applies, and there exists a unique optimal strategy $X^*=(X^*_s)_{s\in[t,T]}\in\cA_t(x,d)$, which is given by the formulas
\begin{equation}\label{eq:optstratobiwang}
\begin{split}
X^*_{t-}&=x, \quad X_T^*=0,\\
X^*_s & = \left(x-\frac{d}{\gamma}\right)
\exp\left\{ - \int_t^s  \frac{\rho}{2+ (T-r)\rho}\,dr \right\}
\frac{1+ (T-s)\rho}{2+(T-s)\rho}
\\
& =  \left(x-\frac{d}{\gamma}\right)\frac{1+ (T-s) \rho }{2+ (T-t) \rho}, \quad s\in [t,T).
\end{split}
\end{equation}
In the context of optimal trade execution in a limit order book model,
this setting is considered in the pioneering work \cite{obizhaeva2013optimal},
and the optimal strategy $X^*$ of~\eqref{eq:optstratobiwang} (for $d=0$)
appears in Proposition~3 of \cite{obizhaeva2013optimal},
where the cost functional $\wt J$ of~\eqref{eq:cost_fct_tilde} is minimized over the strategies of finite variation.
We stress again that we obtain optimality of~\eqref{eq:optstratobiwang} in this setting as a result of a different optimization problem
(minimization of the cost functional $J$ of~\eqref{eq:cost_fct} over c\`adl\`ag semimartingales).

Notice that the optimal strategy $X^*$ of~\eqref{eq:optstratobiwang}
is deterministic, has jumps at times $t$ and $T$ (i.e., block trades in the beginning and in the end)
and is continuous on $(t,T)$.
It is worth noting that the associated deviation process
$D^*$ is constant on $(t,T)$
(but, clearly, has jumps at times $t$ and $T$).
In the case $d=0$ the strategy $X^*$ is monotone.
In general, the strategy is monotone only on $(t,T]$.
Global monotonicity can fail because of the block trade in the beginning
(the size of the block trade depends not only on $x$ but also on~$d$).

\section{Reason for adjusting the dynamics of the deviation process}\label{sec:reason_dynamics_of_deviation}

\subsection{Counterexample}\label{sec:fromsec81offirstversion}

In this subsection 
we illustrate that the covariation term $[\gamma,X]$ in the definition of the deviation process $D$ (see~\eqref{eq:deviation_dyn}) can be necessary to obtain a well-posed optimization problem. 
More precisely, we construct an example, where $\essinf_{X \in \cA_t(x,d)} J_t(x,d,X)=-\infty$ when the deviation process $D$ associated to $X \in \mathcal{A}_t(x,d)$ follows the dynamics
\begin{equation}\label{eq:05052020a1}
dD_s = - \rho_s D_{s}\,d[M]_s + \gamma_s\,dX_s,\quad s\in[t,T]
\end{equation}
(cf.\ \eqref{eq:13022020a1_intro} above or formula~(2) in \cite{fruth2019optimal}).

Consider the setting, where $M=W$
is an $(\cF_s)$-Brownian motion,
$\mu\equiv0$,
$\sigma>0$ and $\rho>0$ are positive deterministic constants
such that $2\rho-\sigma^2>0$, i.e., $\bm{\left(C_{>0}\right)}$ holds.
In particular, in our current setting the price impact process
$\gamma$ is a geometric Brownian motion
$\gamma_s=\gamma_0\exp\{\sigma W_s-\frac{\sigma^2}2 s\}$,
$s\in[0,T]$.

We consider the starting time $t=0$ and assume (for simplicity) that the $\sigma$-field $\cF_0$ is trivial.
We further fix some initial position $x\in\bbR\setminus\{0\}$
and the initial deviation $d=0$.
For $\nu \in \bbR$ define the execution strategy $(X^{(\nu)}_s)_{s\in[0,T]}$ by 
$X^{(\nu)}_{0-}=X^{(\nu)}_0=x$, 
$dX^{(\nu)}_s = \nu X^{(\nu)}_s dW_s$ for $s\in[0,T)$ and  
$X^{(\nu)}_T = 0$,
i.e., the strategy follows a geometric Brownian motion on $[0,T)$
and has a block trade at time $T$.
For each $\nu \in \bbR$, let $D^{(\nu)}=( D_s^{(\nu)} )_{s\in[0,T]}$ be the deviation process associated to $X^{(\nu)}$
according to dynamics~\eqref{eq:05052020a1}, which is 
\begin{equation*}
	\begin{split}
		& dD^{(\nu)}_s 
		= -\rho_s D^{(\nu)}_{s}ds+\gamma_sdX^{(\nu)}_s 
		= -\rho D^{(\nu)}_s ds + \nu \gamma_s X^{(\nu)}_sdW_s 
		, \quad s\in[0,T),\\
		& D^{(\nu)}_{0-}=0, \quad D^{(\nu)}_{T} = D^{(\nu)}_{T-} - \gamma_T X^{(\nu)}_{T-} .
	\end{split}
\end{equation*} 
In particular,
$D^{(\nu)}_s = \int_0^s \nu e^{-\rho(s-r)} \gamma_r X^{(\nu)}_r\,dW_r$ for $s\in[0,T)$.

We first verify that $X^{(\nu)}\in \mathcal{A}_0(x,0)$ for all $\nu\in\bbR$.
Notice that in the current setting we have for all $p\in[1,\infty)$ and $\nu\in\bbR$
\begin{equation}\label{eq:exa_gamma_alpha_X_finite_sup}
E\left[ \sup_{s\in [0,T]} \gamma_s^p \right] <\infty, \quad 
E\left[ \sup_{s\in [0,T]} \alpha_s^p \right] <\infty, \text{ and } 
E\left[ \sup_{s\in [0,T]} |X^{(\nu)}_s|^p \right] <\infty.
\end{equation}
(see, e.g., Lemma~\ref{lem:expectation_sup_N_finite}). 
This, the Burkholder-Davis-Gundy inequality and the H{\"o}lder inequality imply that it holds for all $p\in[2,\infty)$ and $\nu \in \bbR$ that there exists $c\in [1,\infty)$ such that 
\begin{equation*}
\begin{split}
E\left[ \sup_{s\in [0,T)} |D^{(\nu)}_s|^p \right] 
& \leq
c E\left[ \left( \int_0^T \nu^2 e^{-2\rho(T-r)} \gamma_r^2 \left(X^{(\nu)}_r\right)^2\,dr \right)^{p/2}\,\right]
\\
& \leq
c |\nu|^p T^{p/2} E\left[ \sup_{r\in [0,T]} \gamma_r^p |X^{(\nu)}_r|^p \right] 
< \infty.
\end{split}
\end{equation*}
Furthermore, as $D^{(\nu)}_T=D^{(\nu)}_{T-}+\gamma_T\Delta X^{(\nu)}_T=D^{(\nu)}_{T-}-\gamma_T X^{(\nu)}_{T-}$,
we also get $D^{(\nu)}_T\in L^p$, hence
\begin{equation}\label{eq:exa_D_finite_sup}
E\left[ \sup_{s\in [0,T]} |D^{(\nu)}_s|^p \right]<\infty
\end{equation}
for all $p\in[1,\infty)$ and $\nu \in \bbR$.
It now follows from the H{\"o}lder inequality, the Minkowski inequality, \eqref{eq:exa_gamma_alpha_X_finite_sup} and  \eqref{eq:exa_D_finite_sup} that (A1) is satisfied. 
Since $\sigma^2$ is a deterministic constant, (A2) then also holds true. 
Furthermore, the H{\"o}lder inequality, \eqref{eq:exa_gamma_alpha_X_finite_sup} and \eqref{eq:exa_D_finite_sup} prove that (A3) is satisfied. 
Hence, it holds $X^{(\nu)}\in \mathcal{A}_0(x,0)$ for all $\nu \in \bbR$.

We next consider the cost functional $J$ defined by~\eqref{eq:cost_fct} and obtain, for any $\nu\in\bbR$,
\begin{equation}\label{eq:costfctinexadynamicsD}
\begin{split}
& J_0(x,0,X^{(\nu)}) 
= E\left[ \int_{[0,T)} D^{(\nu)}_{s-}dX^{(\nu)}_s + D^{(\nu)}_{T-} \Delta X^{(\nu)}_{T} + \int_{[0,T)} \frac{\gamma_s}{2} \nu^2 (X^{(\nu)}_s)^2 ds + \frac{\gamma_T}{2} \left(\Delta  X^{(\nu)}_{T}\right)^2 \right] \\
&   = \nu E\left[ \int_0^T D^{(\nu)}_s X^{(\nu)}_s dW_s \right] - E\left[ D^{(\nu)}_{T-} X^{(\nu)}_{T-} \right] + \frac{\nu^2}{2} \int_0^T E\left[ \gamma_s (X^{(\nu)}_s)^2 \right] ds + \frac12 E\left[ \gamma_{T-} (X^{(\nu)}_{T-})^2 \right] .
\end{split}
\end{equation}
By the Burkholder-Davis-Gundy inequality, the H{\"o}lder inequality, \eqref{eq:exa_gamma_alpha_X_finite_sup} and \eqref{eq:exa_D_finite_sup}, the stochastic integral $\int_0^{\cdot} D_s^{(\nu)} X^{(\nu)}_s dW_s$ is a martingale,
hence its expectation vanishes.
Further, it holds that 
$d(X^{(\nu)}_s)^2  = 2\nu (X^{(\nu)}_s)^2 dW_s + \nu^2  (X^{(\nu)}_s)^2 ds,$ 
$s \in [0,T)$. 
This yields that 
\begin{equation*}
d\left( \gamma_s (X^{(\nu)}_s)^2 \right) 
= \gamma_s (X^{(\nu)}_s)^2\left( (\nu^2 +2\sigma \nu) ds + (2\nu + \sigma) dW_s \right), \quad s \in [0,T), 
\end{equation*}
and hence 
\begin{equation}\label{eq:expectationgammaX2}
E\left[ \gamma_s (X^{(\nu)}_s)^2 \right] = \gamma_0 x^2 e^{(\nu^2 +2\sigma \nu)s}, \quad s\in[0,T).
\end{equation}
Besides this, we have that for all $s\in[0,T)$
\begin{equation}\label{eq:DXproddiffrepresentation}
d\left( D^{(\nu)}_s X^{(\nu)}_s \right) 
= - \rho D^{(\nu)}_s X^{(\nu)}_s ds  + \nu \gamma_s (X^{(\nu)}_s)^2 dW_s + \nu D^{(\nu)}_s X^{(\nu)}_s dW_s + \nu^2 \gamma_s (X^{(\nu)}_s)^2 ds.
\end{equation}
Again by the Burkholder-Davis-Gundy inequality, the H{\"o}lder inequality, \eqref{eq:exa_gamma_alpha_X_finite_sup} and \eqref{eq:exa_D_finite_sup}, one can show that $\int_0^{\cdot} \gamma_s (X^{(\nu)}_s)^2 dW_s$ is a martingale. 
Therefore, it follows from
\eqref{eq:expectationgammaX2}
and~\eqref{eq:DXproddiffrepresentation}
that for all $\nu \in \bbR\setminus\{0\}$ and $s\in[0,T)$
\begin{equation}\label{eq:expectationDXprod}
\begin{split}
E\left[ D^{(\nu)}_s X^{(\nu)}_s \right] 
& = -\rho \int_0^s E\left[ D^{(\nu)}_u X^{(\nu)}_u \right] du 
+ \nu^2 \int_0^s E\left[ \gamma_u (X^{(\nu)}_u)^2 \right] du \\
& = -\rho \int_0^s E\left[ D^{(\nu)}_u X^{(\nu)}_u \right] du 
+ \frac{\nu^2 \gamma_0 x^2 }{ \nu^2 + 2\sigma \nu } \left( e^{(\nu^2+2\sigma\nu)s}-1\right) .
\end{split}
\end{equation}
We thus obtain that 
\begin{equation*}
E\left[ D^{(\nu)}_s X^{(\nu)}_s \right] 
= e^{-\rho s} \frac{\nu^2 \gamma_0 x^2 }{ \rho + \nu^2 + 2\sigma \nu } \left( e^{(\rho + \nu^2+2\sigma\nu)s}-1\right) , \quad s\in [0,T).
\end{equation*}
Together with \eqref{eq:expectationgammaX2} the cost functional \eqref{eq:costfctinexadynamicsD} becomes 
\begin{equation*}
\begin{split}
& J_0(x,0,X^{(\nu)}) \\
& = - e^{-\rho T} \frac{\nu^2 \gamma_0 x^2 }{ \rho + \nu^2 + 2\sigma \nu } \left( e^{(\rho + \nu^2+2\sigma\nu)T}-1\right) 
+ \frac{\nu^2}{2} \int_0^T \gamma_0 x^2 e^{(\nu^2 + 2\sigma \nu)s} ds 
+ \frac{\gamma_0 x^2}{2} e^{(\nu^2 + 2\sigma\nu)T}  \\
& = \frac{\gamma_0 x^2}{2} \Bigg( e^{(\nu^2 + 2\sigma \nu)T} \left( \frac{\nu^2}{\nu^2+2\sigma\nu} -\frac{2\nu^2}{\rho + \nu^2+2\sigma\nu} + 1 \right) 
- \left( \frac{\nu^2}{\nu^2+2\sigma\nu} - \frac{2\nu^2 e^{-\rho T}}{\rho + \nu^2+2\sigma\nu} \right) \Bigg) \\
& = \frac{\gamma_0 x^2}{2} \left( I_1(\nu) - I_2(\nu) \right),
\end{split}
\end{equation*}
where
$I_1(\nu)=e^{(\nu^2 + 2\sigma \nu)T} \left( \frac{\nu^2}{\nu^2+2\sigma\nu} -\frac{2\nu^2}{\rho + \nu^2+2\sigma\nu} + 1 \right)$ and $I_2(\nu)=\frac{\nu^2}{\nu^2+2\sigma\nu} - \frac{2\nu^2 e^{-\rho T}}{\rho + \nu^2+2\sigma\nu}$.
Observe that 
\begin{equation}\label{eq:factorincostfctexadynamicsD}
\begin{split}
\frac{\nu^2}{\nu^2+2\sigma\nu} -\frac{2\nu^2}{\rho + \nu^2+2\sigma\nu} + 1 
& = \frac{1}{1+2\frac{\sigma}{\nu}} - \frac{2}{\frac{\rho}{\nu^2} +1 + 2\frac{\sigma}{\nu}} + 1 \\ 
& = \frac{2}{\nu}\frac{\sigma + \frac{\rho}{\nu} + \frac{\rho\sigma}{\nu^2}+2\frac{\sigma^2}{\nu}}{\left( 1+2\frac{\sigma}{\nu} \right) \left( \frac{\rho}{\nu^2} +1+2\frac{\sigma}{\nu} \right)} ,
\end{split}
\end{equation}
i.e., this term behaves as $\frac{2\sigma}{\nu}$ in the limit $\nu\to-\infty$
(in particular, this term is strictly negative provided $\nu<0$ and $|\nu|$ is sufficiently large).
It follows that $I_1(\nu)\to-\infty$ as $\nu\to-\infty$,
whereas, clearly, $I_2(\nu)\to1-2e^{-\rho T}$ as $\nu\to-\infty$,
hence
\begin{equation*}
J_0(x,0,X^{(\nu)})  \to -\infty \quad \text{ as } \nu \to -\infty .
\end{equation*}
Thus, dynamics~\eqref{eq:05052020a1} leads to an ill-posed optimization problem.

\subsection{Solution in our framework}\label{sec:examplewithLambertWfct}

In the setting above (see the second paragraph in Section~\ref{sec:fromsec81offirstversion}),
we recover a well-posed optimization problem
when we use dynamics~\eqref{eq:deviation_dyn}
instead of~\eqref{eq:05052020a1}
for the deviation process $D$.
Let us verify that Theorem~\ref{thm:sol_val_fct} applies
and present explicit formulas for the optimal strategy $X^*$
in $\cA_t(x,d)$ and the associated deviation process $D^*$,
for any $t\in[0,T]$, $x,d\in\bbR$.

In this setting, $\bm{\left(C_{[M]}\right)}$ is trivially satisfied, while $\bm{\left(C_{>0}\right)}$ holds true due to our assumption $2\rho-\sigma^2>0$. BSDE~\eqref{eq:bsde} takes the form 
\begin{equation}\label{eq:odeforspecialdetermincase}
	dY_s  = \left( \frac{(\rho Y_s + \sigma Z_s)^2}{\sigma^2 Y_s+ \rho - \frac{\sigma^2}{2}} -\sigma Z_s \right)ds + Z_sdW_s + dM^\perp_s ,\quad  s \in [0,T], \quad 
	Y_{T}  =\frac{1}{2},
\end{equation}
and has a deterministic solution
\begin{equation}\label{eq:solforspecialdetermincase}
Z \equiv 0, \quad
M^\perp \equiv 0, \quad
Y_s = \frac{\rho-\frac{\sigma^2}{2}}{\sigma^2}\,
\mathcal{W}\left(  \frac{\rho-\frac{\sigma^2}{2}}{\sigma^2} e^{\kappa - \frac{\rho^2}{\sigma^2}s} \right)^{-1}, \quad s \in [0,T],
\end{equation}
where $\mathcal{W}$ denotes the Lambert $W$ function and 
$\kappa = \log(2) + \frac{1}{\sigma^2} (2\rho - \sigma^2 + \rho^2 T) .$ 
Observe that in this setting
$$
\wt{\beta}_s = \frac{\rho Y_s}{\sigma^2 Y_s + \rho - \frac{\sigma^2}{2}}, \quad s \in [0,T],
$$
and that both $Y$ and $\wt\beta$
are deterministic increasing continuous $(0,1/2]$-valued functions.
In particular, $\wt\beta$ is bounded and it is a semimartingale.
Hence, Theorem~\ref{thm:sol_val_fct} applies,
and, for all $t\in[0,T]$, $x,d\in\bbR$,
the unique optimal strategy
$X^*=(X^*_s)_{s\in[t,T]}\in\cA_t(x,d)$
and its associated deviation process
$D^*=(D^*_s)_{s\in[t,T]}$
are given by the formulas
$X^*_{t-}=x$, $D^*_{t-}=d$,
\begin{align*}
X^*_s
&=
\left(x-\frac{d}{\gamma_t}\right)\cE(Q)_{t,s}\,(1-\wt\beta_s),
\quad s\in [t,T),
\\
D^*_s
&=
\left(x-\frac{d}{\gamma_t}\right)\cE(Q)_{t,s}\,(-\gamma_s\wt\beta_s),
\quad s\in [t,T),
\end{align*}
and $X^*_T=0$,
$D^*_T=\left(x-\frac{d}{\gamma_t}\right)\cE(Q)_{t,T}\,(-\gamma_T)$,
where 
\begin{equation*}
Q_s = - \sigma \int_0^s \wt{\beta}_r dW_r - (\rho-\sigma^2) \int_0^s \wt{\beta}_r dr, \quad s\in [0,T].
\end{equation*}

We, finally, discuss some properties of the optimal strategy in the case $x\ne\frac d{\gamma_t}$
(there is nothing to discuss in the remaining case $x=\frac d{\gamma_t}$).
Like in the situation of Section~\ref{sec:recoveringOW},
$X^*$ has jumps at times $t$ and $T$ and is continuous on $(t,T)$.
As $1-\wt\beta$ is positive,
here, again, $X^*$ has the same sign as $x-\frac d{\gamma_t}$ on $(t,T]$.
Now in contrast to Section~\ref{sec:recoveringOW},
the associated deviation process $D^*$
is no longer constant on $(t,T)$. Further,
as $1-\wt\beta$ is nonvanishing and has finite variation on $[t,T]$,
while $\cE(Q)_{t,\cdot}$, almost surely, has infinite variation on all subintervals of $[t,T]$,
we get that, in contrast to Section~\ref{sec:recoveringOW},
$X^*$, almost surely, has infinite variation on all subintervals of $[t,T]$
(in particular, $X^*$ is in no way monotone on any subinterval). 
See Figure~\ref{fig:optStratLambertWDouble} for an illustration.

\begin{figure}[!htb]
	\includegraphics[scale=0.52]{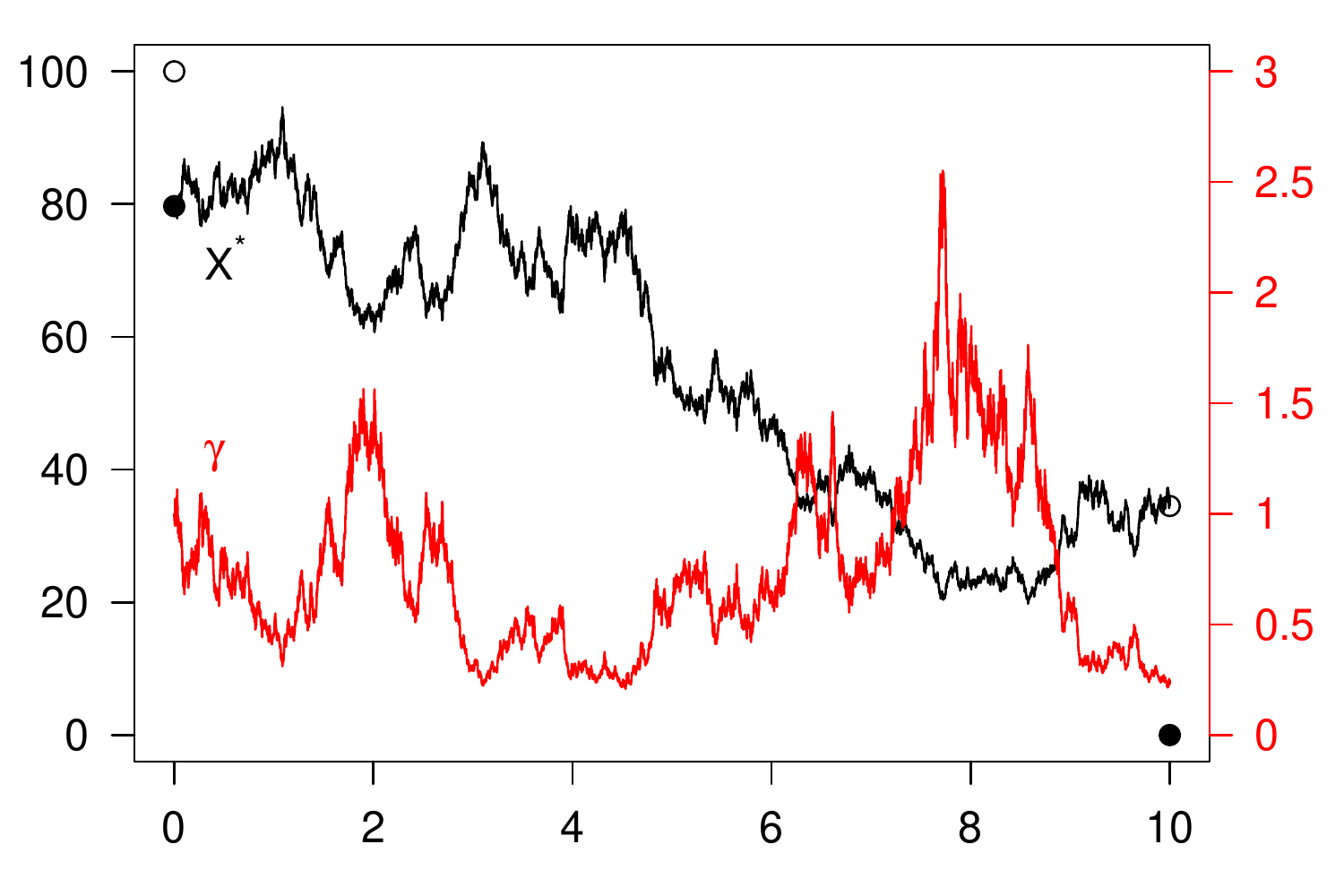}
	\includegraphics[scale=0.52]{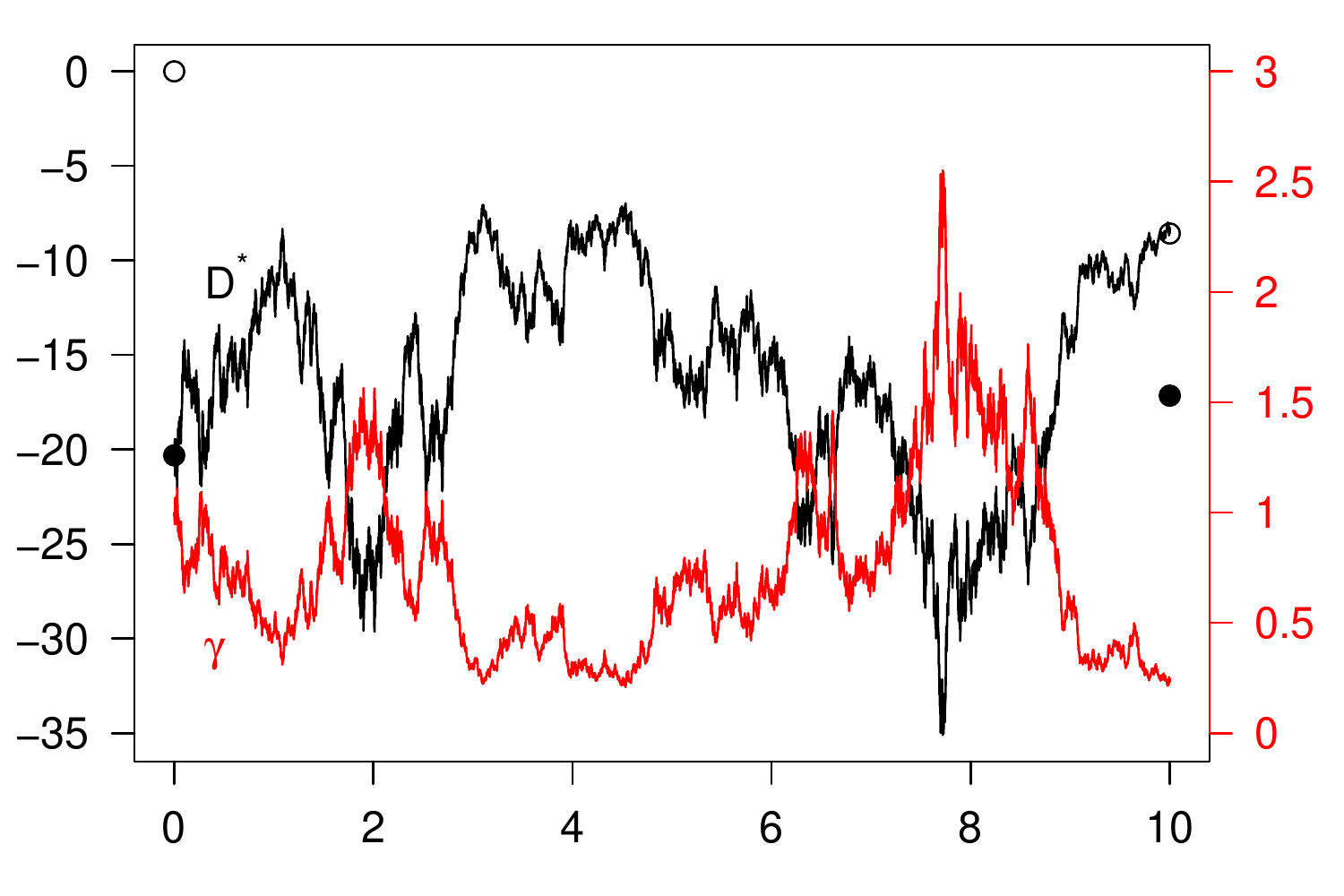}
	\caption{Left: A simulation of the optimal strategy $X^*$ (black) and the price impact $\gamma$ (red) in the setting of Section~\ref{sec:examplewithLambertWfct} for $T=10$, $x=100$, $d=0$, $\gamma_0=1$, $\rho=0.5$ and $\sigma=0.8$. Note the difference in scales. Right: The associated deviation process $D^*$ (black) and the price impact $\gamma$ (red) for the same situation.}
	\label{fig:optStratLambertWDouble}
\end{figure}

\section{Examples}\label{sec:furtherexamples}

In this section, we present several interesting qualitative effects that can arise in our framework. 
To do this, we consider several subsettings of the following common set-up: 

Let $M=W$ be a Brownian motion and $\cF_s=\cF_s^W$ for all $s\in[0,T]$. Let $t=0$, $x,d \in \R$ with $x\neq \frac{d}{\gamma_0}$ (for the case $x=\frac{d}{\gamma_0}$, see Lemma~\ref{lem:optstratifzero}). 
The resilience here is taken to be a deterministic constant $\rho \in \R\setminus\{0\}$ (for the case $\rho = 0$, see Proposition~\ref{propo:withoutresiliencecloseimmediately}). 
We consider the price impact $\gamma$ from~\eqref{eq:dyn_gamma} with $\sigma\equiv 0$, i.e., 
\begin{equation*}
\gamma_s = \gamma_0 \exp\left( \int_0^s \mu_r dr \right), \quad s\in [0,T].
\end{equation*}
In particular, $\gamma$ is continuous and of finite variation. 
We assume that there exist deterministic constants $\eps,\ol\mu\in(0,\infty)$ such that
\begin{equation}\label{eq:08032020a1}
2\rho+\mu\ge\eps\;\;\cD_W\text{-a.e.}
\quad\text{and}\quad
\mu\le\ol\mu\;\;\cD_W\text{-a.e.}
\end{equation}
(in particular, $\bm{\left(C_{>0}\right)}$ is satisfied).
Our current set-up is a special case of the settings considered in Sections \ref{sec:bsde1} and~\ref{sec:bsde2} below.
Therefore, it follows from Proposition~\ref{propo:existence_bsde_general_filtration}
(alternatively, from Proposition~\ref{propo:existence_bsde_continuous_filtration})
that $\bm{\left(C_{\textbf{BSDE}}\right)}$ is satisfied.

As before, $(Y,Z,M^\perp)$ denotes a solution to BSDE~\eqref{eq:bsde} satisfying the requirements in $\bm{\left(C_{\textbf{BSDE}}\right)}$.
We notice that $M^\perp\equiv0$ in our current set-up because,
due to the martingale representation theorem (Theorem~V.3.4 in \cite{revuzyor}),
on the Brownian filtration any local martingale $M^\perp$ with $M^\perp_0=0$ and $[M^\perp,W]=0$ is indistinguishable from zero.
For the process $\wt{\beta}$ defined in~\eqref{eq:def_beta_tilde} we obtain 
\begin{equation}\label{eq:betatildespecialcaseBrownianmotionsigmazero}
\wt{\beta}_s = \frac{\rho+\mu_s}{2\rho+\mu_s}\,2Y_s
=\left(1-\frac\rho{2\rho+\mu_s}\right)2Y_s
, \quad s \in [0,T].
\end{equation}
Notice that, by~\eqref{eq:08032020a1}, $\wt\beta$ is bounded.
That is, in our current set-up, including~\eqref{eq:08032020a1}, the assumptions of Theorem~\ref{thm:sol_val_fct} are satisfied.

What varies between the examples in this section is the choice of $\mu$, i.e., the price impact process $\gamma$.
In the examples below, we distinguish
between the following two situations.

\smallskip\noindent
\emph{Situation~1:}
There exists a c\`adl\`ag semimartingale
$\beta=(\beta_s)_{s\in[0,T]}$ such that
\begin{equation}\label{eq:30042020a1}
\wt\beta=\beta\quad\cD_W\text{-a.e.}
\end{equation}

\noindent
\emph{Situation~2:}
There is no c\`adl\`ag semimartingale $\beta$ such that \eqref{eq:30042020a1} is satisfied.

\smallskip\noindent
As we know from Theorem~\ref{thm:sol_val_fct},
in Situation~1 there exists a unique
(up to $\cD_W$-null sets)
optimal strategy $X^*=(X^*_s)_{s\in[0,T]}\in\cA_0(x,d)$,
and it is given by the formulas
$X^*_{0-}=x$, $X^*_T=0$ and
\begin{equation}\label{eq:30042020a2}
X^*_s=\left(x-\frac{d}{\gamma_0}\right)
\exp\left\{-\int_0^s  \beta_r (\mu_r + \rho)\,dr\right\}
(1-\beta_s),\quad s\in [0,T),
\end{equation}
while in Situation~2 there does not exist an optimal strategy.

Before turning to specific examples we notice that the multiplier
$$
\left(x-\frac{d}{\gamma_0}\right)
\exp\left\{ - \int_0^s  \beta_r (\mu_r + \rho)\,dr \right\},\quad s\in[0,T],
$$
in~\eqref{eq:30042020a2} is a nonvanishing continuous process of finite variation.
Therefore, if, in Situation~1, we want to obtain the optimal strategy
$X^*$ of infinite variation on $[0,T]$
and/or with jumps inside $(0,T)$,
it is enough to construct $\beta$
(see \eqref{eq:betatildespecialcaseBrownianmotionsigmazero}
and~\eqref{eq:30042020a1})
of infinite variation on $[0,T]$
and/or with jumps inside $(0,T)$.

\begin{ex}\label{ex:opt_strat_of_infinite_variation}
Let $\mu$ be a continuous process of finite variation satisfying~\eqref{eq:08032020a1} such that
\begin{equation}\label{eq:08032020a3}
\text{a.s.\ the function }s\mapsto\rho+\mu_s\text{ is nonvanishing on }[0,T].
\end{equation}
Observe that for a fixed $\omega \in \Omega$, the unique solution to the Bernoulli ODE
	\begin{equation*}
	d\ol Y_s(\omega) = \left( \frac{2\left(\rho+\mu_s(\omega)\right)^2 \ol Y_s(\omega)^2}{2\rho + \mu_s(\omega)} -\mu_s(\omega)\ol Y_s(\omega) \right) ds,\quad  s \in [0,T], \quad 
	\ol Y_{T}(\omega)  =\frac{1}{2},
	\end{equation*}
which is BSDE~\eqref{eq:bsde} without the martingale part, is given by the formula
\begin{equation}\label{eq:02052020a2}
	\ol Y_s(\omega) =  e^{\int_s^T \mu_r(\omega) dr} \left( \int_s^T \frac{2\left(\rho+\mu_r(\omega)\right)^2}{2\rho + \mu_r(\omega)} e^{\int_r^T \mu_u(\omega) du} dr + 2  \right)^{-1}, \quad s \in [0,T].
\end{equation}
It follows that it is possible to choose $\mu$ such that $\ol Y$
is not adapted.
Choosing $\mu$ in such a way we conclude that the solution
$(Y,Z,M^\perp\equiv0)$ of BSDE~\eqref{eq:bsde}
satisfies $\PdW(Z\neq 0) > 0$.
This yields that, with positive probability, $Y$ has infinite variation on $[0,T]$.
Define $\varphi_s=\frac{2(\rho + \mu_s)}{2\rho + \mu_s}$, $s\in[0,T]$, which is a nonvanishing (recall~\eqref{eq:08032020a3})
continuous process of finite variation.
Hence, $\wt{\beta} =  \varphi Y$ is a continuous semimartingale that, with positive probability, has infinite variation on $[0,T]$.
Thus, we are in Situation~1 with $\beta\equiv\wt\beta$,
and the optimal strategy $X^*$,
which is given by~\eqref{eq:30042020a2},
has, with positive probability, infinite variation on $[0,T]$.

In contrast to the situation in Section~\ref{sec:examplewithLambertWfct},
where the infinite variation in $X^*$ was caused by the infinite variation in the exogenous process $\gamma$, in this example all exogenous processes (i.e., $\gamma$ and $\rho$) 
have finite variation\footnote{As easily seen, it is even possible to choose $\mu$ with $C^\infty$ paths. Then $\gamma$ also has $C^\infty$ paths. The process $\rho$ is even constant.}, but the optimal strategy has infinite variation,
i.e., ``oscillates much more quickly'' than the exogenous processes do.
This is due to the incoming information that is reflected in the endogenous process~$Y$ (which turns out to have infinite variation).
\end{ex}

\begin{ex}\label{ex:opt_strat_jumps_in_between}
Optimal strategies we have seen so far have jumps
(block trades) at times $0$ and $T$ only.
In order to construct an optimal strategy with jumps inside $(0,T)$
it is enough to take
$$
\text{a c\`adl\`ag semimartingale }\mu\text{ satisfying~\eqref{eq:08032020a1} that exhibits jumps in }(0,T),
$$
i.e., with positive probability,
$\{s\in(0,T):\Delta\mu_s\ne0\}\ne\emptyset$, such that
\begin{equation}\label{eq:30042020a3}
\text{the corresponding process }Y\text{ is nonvanishing.}
\end{equation}
Indeed, in this case, $\wt\beta$ is a c\`adl\`ag semimartingale,
so we are in Situation~1 with $\beta\equiv\wt\beta$.
Moreover, as $Y$ is continuous and nonvanishing,
we readily see
from~\eqref{eq:betatildespecialcaseBrownianmotionsigmazero} that
$$
\Delta\mu_s\ne0
\;\;\Longleftrightarrow\;\;
\Delta\wt\beta_s\ne0,
$$
hence the optimal strategy $X^*$,
which is given by~\eqref{eq:30042020a2},
contains block trades inside $(0,T)$.

\smallskip
To show a specific example of this kind,
we consider, for some $t_0 \in (0,T)$,
a deterministic $\mu$ given by the formula
$\mu_s=1_{[t_0,T]}(s)$, $s\in[0,T]$.
Observe that \eqref{eq:08032020a1} is satisfied whenever $\rho>0$,
so we take some $\rho>0$ in this example.
BSDE~\eqref{eq:bsde} here takes the form
	\begin{equation*}
	\begin{split}
	dY_s & = \rho Y_s^2 ds + Z_sdW_s + dM^\perp_s, \quad s\in [0,t_0], \\
	dY_s & = \left( \frac{2(\rho+1)^2 Y_s^2}{2\rho + 1} - Y_s \right) ds +Z_sdW_s + dM^\perp_s,\quad  s \in [t_0,T], \quad 
	Y_{T}  =\frac{1}{2},
	\end{split}
	\end{equation*} 
and has a deterministic solution $(Y,Z\equiv0,M^\perp\equiv0)$ given by
	\begin{equation}\label{eq:29042020a1}
	\begin{split}
	Y_s & 
	= (2\rho +1)
	\left( 2(\rho +1)^2 - 2\rho^2 e^{s-T} \right)^{-1}
	, \quad s \in [t_0,T],	\\
	Y_s & = \frac{1}{Y_{t_0}^{-1} + (t_0 - s) \rho }, \quad s \in [0,t_0).
	\end{split}
	\end{equation}
Notice that $Y$ is continuous, strictly increasing and $(0,1/2]$-valued.
In particular, \eqref{eq:30042020a3} is satisfied, and what is stated after~\eqref{eq:30042020a3} applies.
Observe that, in this specific example,
\begin{equation}\label{eq:29042020a2}
\beta_s=\begin{cases}
Y_s, & s \in [0,t_0),\\
Y_s\left(1+\frac1{2\rho+1}\right),&s\in[t_0,T],
\end{cases}
\end{equation}
which is a deterministic strictly increasing $(0,1)$-valued c\`adl\`ag function
with the only jump at time $t_0$:
$\Delta\beta_{t_0} = \frac{Y_{t_0}}{2\rho +1}>0$. 
From
\eqref{eq:29042020a1}
and~\eqref{eq:29042020a2} 
we compute that 
\begin{equation}\label{eq:15052020a1}
\exp\left\{ - \int_0^s  \beta_r (\mu_r + \rho)\,dr \right\}  = 
\begin{cases} 
Y_0 Y_s^{-1}, & s \in [0,t_0),\\
e^{t_0-s} Y_0 Y_s^{-1}, & s \in [t_0,T],
\end{cases}
\end{equation}
which, together with
\eqref{eq:29042020a1}
and~\eqref{eq:29042020a2},
provides the optimal strategy in closed form (see~\eqref{eq:30042020a2}). 
However, the qualitative structure of the optimal strategy $X^*$,
in fact, follows from~\eqref{eq:30042020a2}
even without calculating~\eqref{eq:15052020a1}: 

First, $X^*$ is deterministic, and, due to $\beta$ being  strictly increasing and $(0,1)$-valued, $X^*$ is monotone on $(0,T]$.  
Moreover, the facts that $\beta<1$, $\Delta\beta_{t_0}>0$ and $x\neq \frac{d}{\gamma_0}$ together with~\eqref{eq:30042020a2} imply that the optimal strategy necessarily has block trades in the end and at time $t_0$. Their signs are opposite to the sign of $x-\frac d{\gamma_0}$. 
Whether or not $X^*$ has a block trade in the beginning depends on the value of the initial deviation $d$. Likewise, we claim the monotonicity of $X^*$ only on $(0,T]$ because whether or not $X^*$ is monotone on $[0,T]$ also depends on~$d$.\footnote{\label{ft:02042021a1}Namely, $X^*$ has a block trade in the beginning if and only if $x\ne(x-\frac d{\gamma_0})(1-\beta_0)$, i.e., if and only if $d\ne-\frac{\beta_0}{1-\beta_0}\gamma_0x$. Likewise, $X^*$ is monotone on $[0,T]$ if and only if either $x\ge0$, $d\ge-\frac{\beta_0}{1-\beta_0}\gamma_0x$ holds or $x\le0$, $d\le-\frac{\beta_0}{1-\beta_0}\gamma_0x$ holds. (In particular, if $d=0$, then $X^*$ is monotone on $[0,T]$.)}

Between the block trades the associated deviation process $D^*$ is constant: 
It follows from~\eqref{eq:opt_dev_representation}, \eqref{eq:29042020a2} and \eqref{eq:15052020a1} that 
$D_s^* = (d-\gamma_0 x) Y_0$, $s\in [0,t_0)$, and $D_s^* = (d-\gamma_0 x) Y_0 \left(1+\frac{1}{2\rho+1}\right)$, $s\in [t_0,T)$.

Figure~\ref{fig:optStratJump} is an illustration for specific parameter values.

\begin{figure}[!htb]
	\includegraphics[scale=0.52]{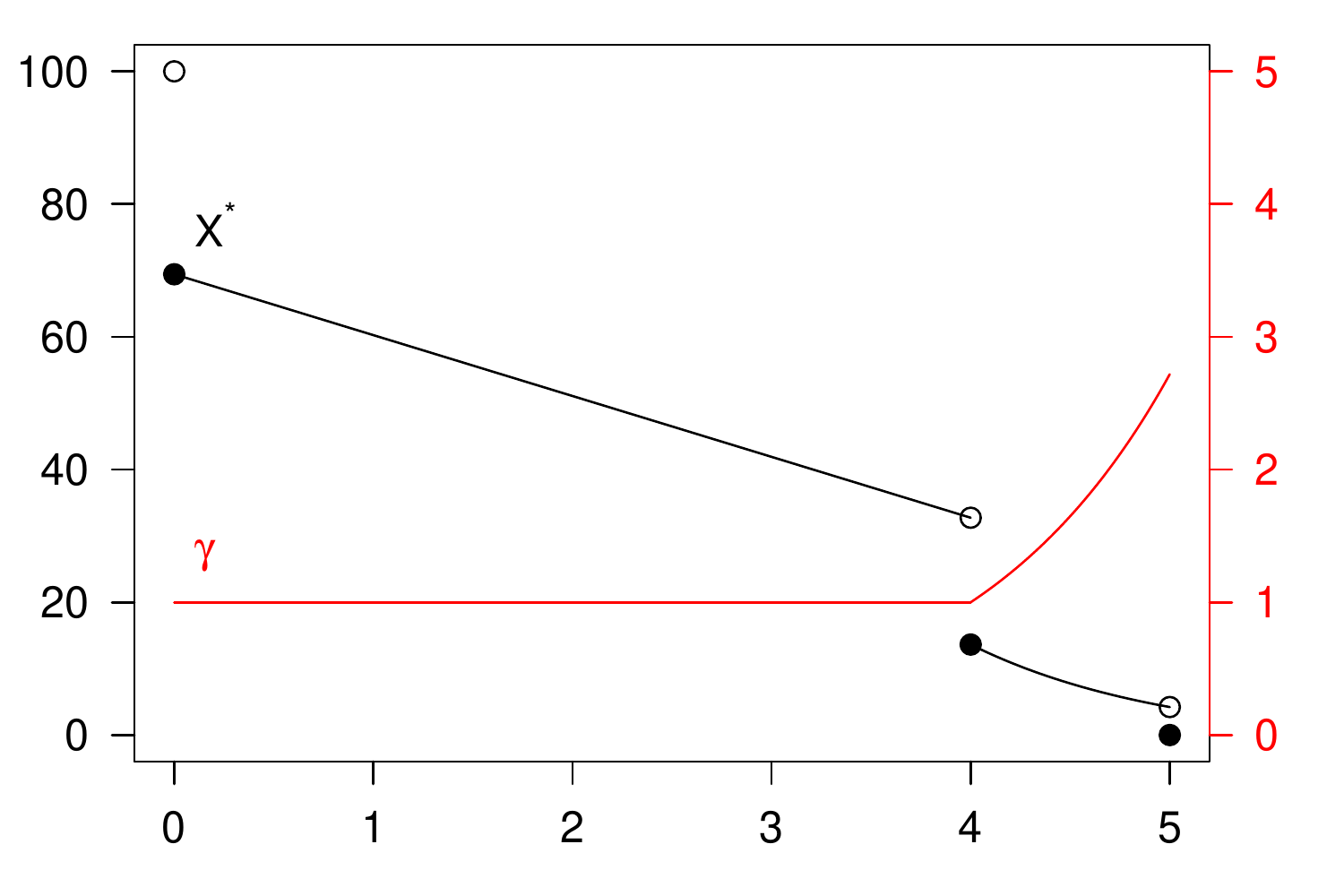}
	\includegraphics[scale=0.52]{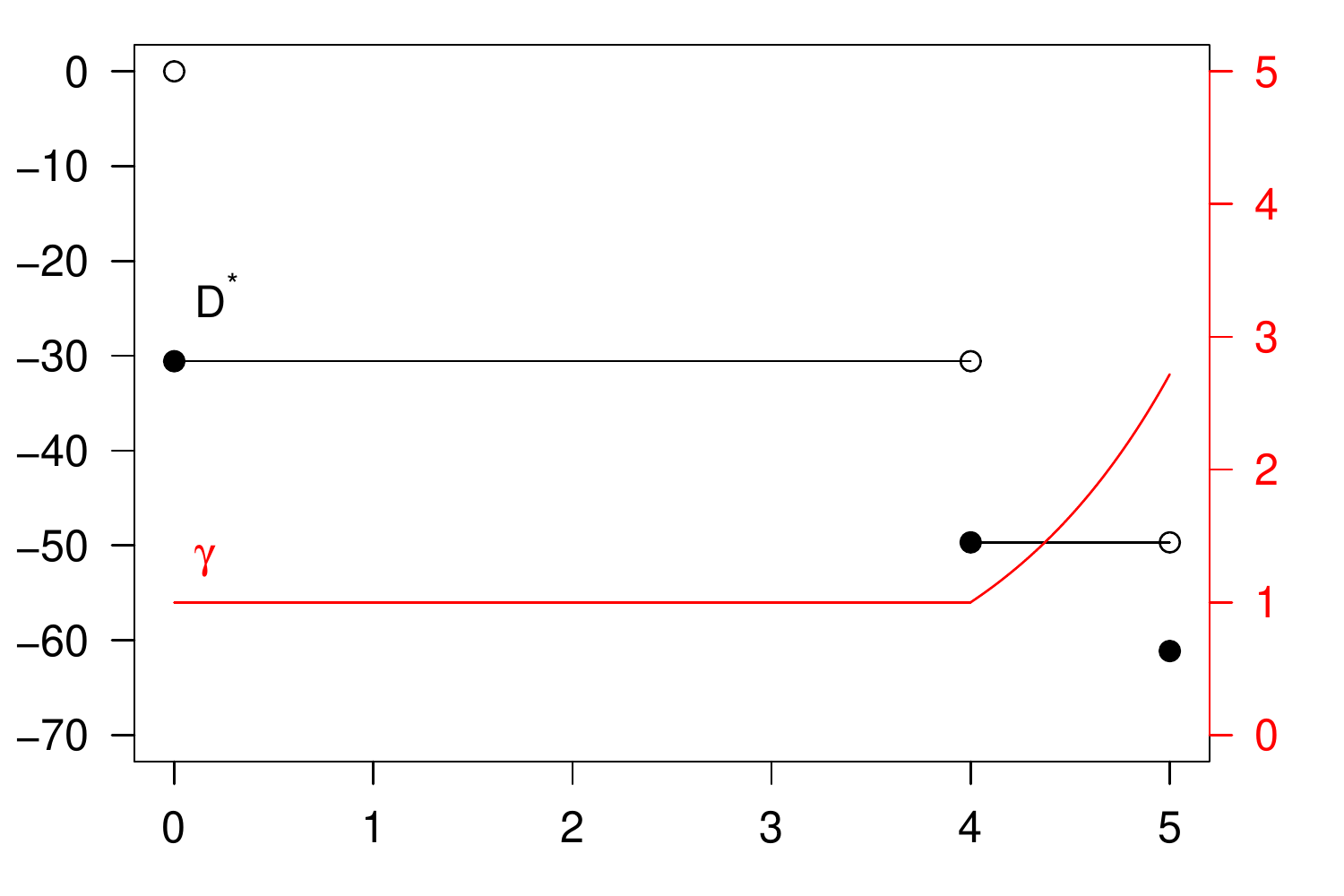}
	\caption{Left: The optimal strategy $X^*$ (black) and the price impact $\gamma$ (red) in the setting of Example~\ref{ex:opt_strat_jumps_in_between} with $\mu_s=1_{[t_0,T]}(s)$, $s\in[0,T]$, and for $T=5$, $x=100$, $d=0$, $\gamma_0=1$, $\rho=0.3$ and $t_0=4$. 
	Note the difference in scales. 
	Right: The associated deviation process $D^*$ (black) and the price impact $\gamma$ (red) for the same situation.}
	\label{fig:optStratJump}
\end{figure}

Observe that the reaction of the optimal strategy to changes in the price impact is rather sensitive:
here only $\mu$ jumps at time $t_0$
(not the price impact $\gamma$ itself),
but this already causes a jump in $X^*$ at time $t_0$.
Finally, it is worth noting that a model with deterministically
time-varying price impact and resilience was considered
in Section~8 of \cite{fruth2014optimal},
but examples of such type are not possible in their framework
because the smoothness assumption in Section~8
of \cite{fruth2014optimal}
excludes the possibility of block trades inside $(0,T)$
(cf.\ Theorem~8.4 in \cite{fruth2014optimal}).
\end{ex}

\begin{ex}\label{ex:neg_res}
In models of price impact that include resilience it is commonly assumed that resilience is positive.
But negative resilience also has a natural interpretation, as it models self-exciting behaviour of the price impact, where trading activities of the large investor stimulate other market participants to trade in the same direction.
In this example we discuss a basic effect of negative resilience in our model.
To this end, we consider some $\rho<0$ and take a deterministic constant $\mu>-2\rho\;(>0)$, which ensures~\eqref{eq:08032020a1}.
Here, again, BSDE~\eqref{eq:bsde} has a deterministic solution
$(Y,Z\equiv0,M^\perp\equiv0)$, which is given by
$$
Y_s=\frac12\mu(2\rho+\mu)\left(
(\rho+\mu)^2-\rho^2e^{\mu(s-T)}
\right)^{-1},\quad s\in[0,T].
$$
It follows that
$$
\wt\beta_s=\mu(\rho+\mu)\left(
(\rho+\mu)^2-\rho^2e^{\mu(s-T)}
\right)^{-1},\quad s\in[0,T],
$$
which is a deterministic positive continuous increasing function, in particular, a semimartingale. Thus, we are in Situation~1 with $\beta\equiv\wt\beta$. 
Notice that
$$
\beta_s>\frac{\mu(\rho+\mu)}{(\rho+\mu)^2}=\frac\mu{\rho+\mu}>1,\quad s\in[0,T],
$$
i.e., in contrast to Example~\ref{ex:opt_strat_jumps_in_between}, $\beta$ is now $(1,\infty)$-valued. 
We set
$$
\lambda_s=\left(x-\frac d{\gamma_0}\right)\exp\left\{-(\rho+\mu)\int_0^s\beta_r\,dr\right\},\quad s\in[0,T],
$$
and have by \eqref{eq:opt_strat_representation}--\eqref{eq:opt_dev_representation} that
\begin{equation}\label{eq:02052020a1}
	X^*_s=\lambda_s(1-\beta_s)
	\quad\text{and}\quad
	D^*_s=-\lambda_s\gamma_s\beta_s,\quad s\in[0,T).
\end{equation}
The fact that $\beta$ is $(1,\infty)$-valued makes the factor $1-\beta$ in~\eqref{eq:02052020a1} negative and means that the optimal strategy $X^*$ is not monotone on $[0,T]$ even for $d=0$ because of the block trade at time $0$
(this is contrary to Example~\ref{ex:opt_strat_jumps_in_between}
and Section~\ref{sec:recoveringOW}, where the optimal strategy is monotone on $[0,T]$, once $d=0$).
Indeed, let, for the moment, $d=0$ and, say, $x>0$ (the objective to \emph{sell} shares). Then, in the first block trade, more than $x$ shares are sold and the sell-program is thus changed into the buy-program. This is done to profit from the negative resilience that drives the deviation process $D^*$ associated to $X^*$ down also after the initial block trade and allows to profit from the subsequent buy-program. 

It can, however, be shown that $X^*$ is monotone on $(0,T]$. To this end, we first prove monotonicity on $(0,T)$: 
BSDE~\eqref{eq:bsde} for $Y$ (just a Bernoulli ODE in this case) and~\eqref{eq:betatildespecialcaseBrownianmotionsigmazero} imply that $\beta$ satisfies the (Bernoulli) ODE
\begin{equation}\label{eq:dotbetanegres}
	\dot\beta_s=(\rho+\mu)\beta_s^2-\mu\beta_s,\quad s\in[0,T].
\end{equation}
Also observe that $\dot\gamma_s=\mu\gamma_s$ and $\dot\lambda_s=-(\rho+\mu)\lambda_s\beta_s$ for all $s\in[0,T]$. It now follows from~\eqref{eq:02052020a1} that
$$
\dot X^*_s=\lambda_s\left(-(\rho+\mu)\beta_s(1-\beta_s)-(\rho+\mu)\beta_s^2+\mu\beta_s\right)=-\rho\lambda_s\beta_s,\quad s\in(0,T),
$$
which has the same sign as $x-\frac d{\gamma_0}$. 
This shows that $X^*$ is monotone on $(0,T)$.  
Further, note that for the final block trade we have
$\Delta X^*_T=-X^*_{T-}=\lambda_T(\beta_T-1)$. Since this also 
has the same sign as $x-\frac d{\gamma_0}$, we conclude that $X^*$ is even monotone on $(0,T]$. 

Moreover, observe that it follows from~\eqref{eq:02052020a1} and~\eqref{eq:dotbetanegres} that
$$
\dot D^*_s=\lambda_s\left((\rho+\mu)\gamma_s\beta_s^2-\mu\gamma_s\beta_s-\gamma_s\left((\rho+\mu)\beta_s^2-\mu\beta_s\right)\right)=0,\quad s\in(0,T),
$$
i.e., the trading is performed in the way that $D^*\equiv\const$ on $(0,T)$. 

The optimal strategy and deviation process for specific parameter values are shown in Figure~\ref{fig:optStratRhoNeg}.

\begin{figure}[!htb]
	\includegraphics[scale=0.52]{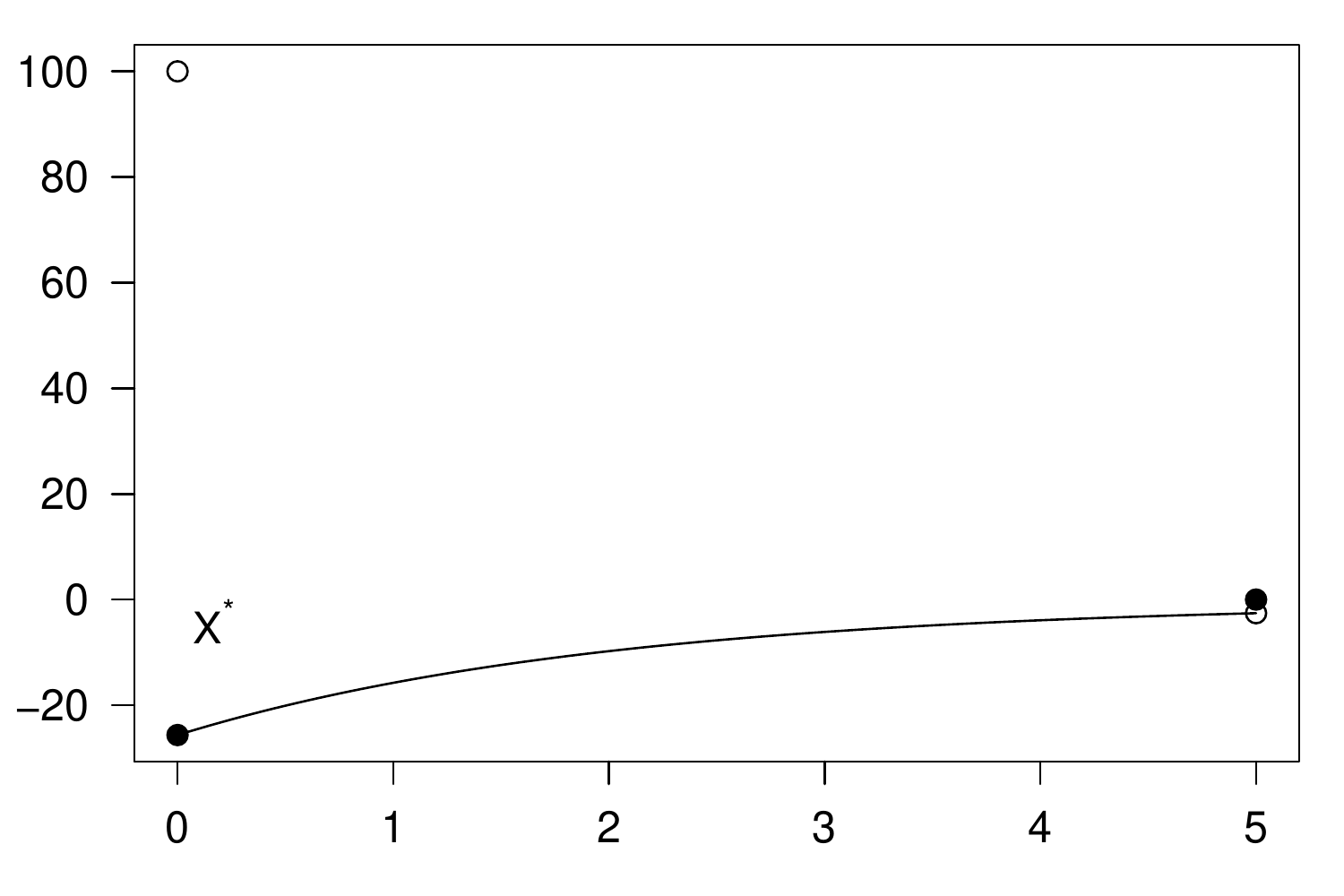} 
	\includegraphics[scale=0.52]{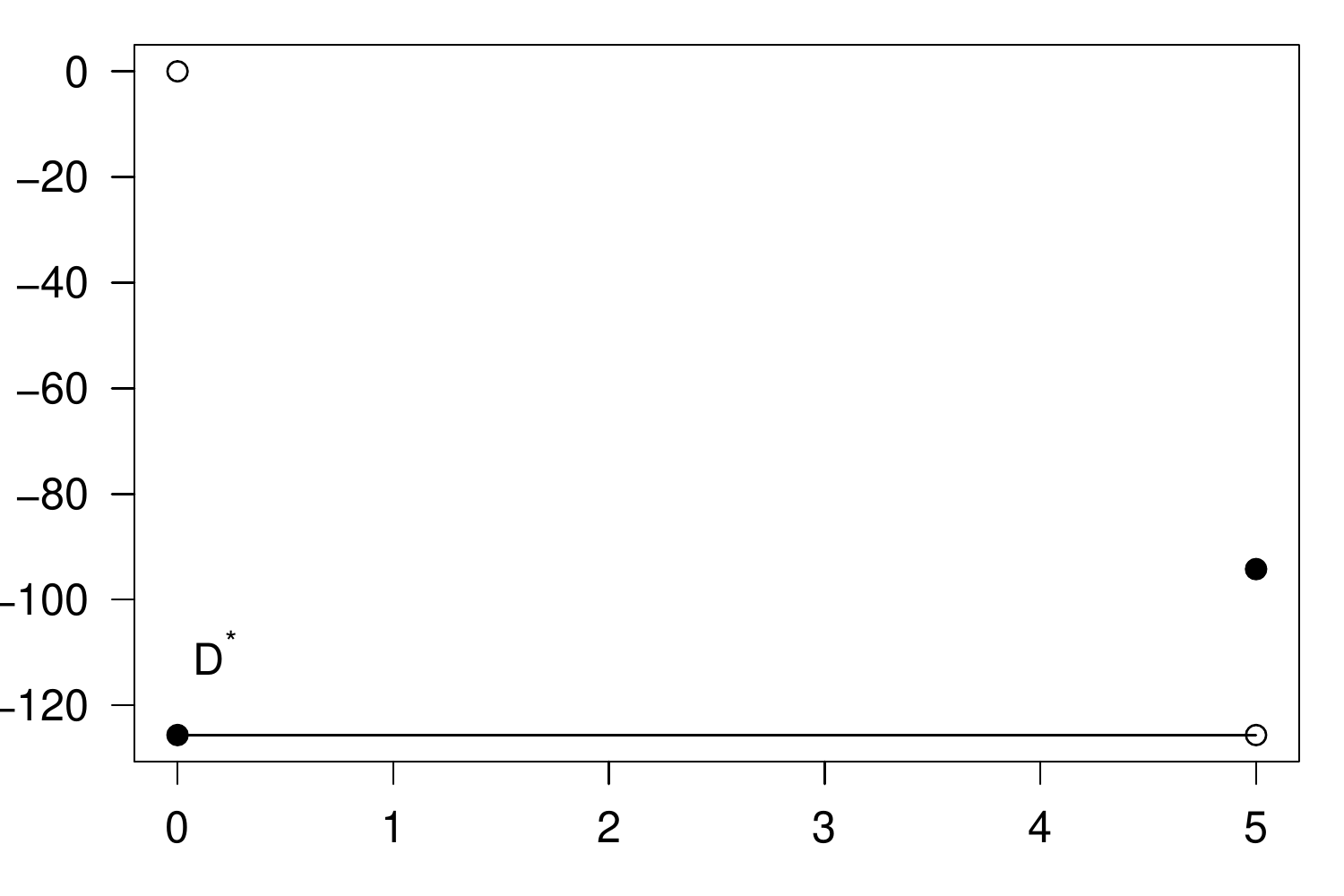} 
	\caption{Left: The optimal strategy $X^*$ in the setting of Example~\ref{ex:neg_res} for  
	$T=5$, $x=100$, $d=0$, $\gamma_0=1$, $\rho=-0.1$ and $\mu=0.5$. To plot $X^*$ we compute that $\lambda_s = (x-d/\gamma_0) e^{-\mu s} Y_s^{-1} Y_0$, $s \in [0,T]$. 
	Right: The associated deviation process $D^*$ for the same situation.}
	\label{fig:optStratRhoNeg}
\end{figure}
\end{ex}

\begin{ex}\label{ex:opt_strat_does_not_exist}
Finally, in order to construct an example of Situation~2 it suffices to take any deterministic c\`adl\`ag function $\mu$ 
such that there exists 
$\delta \in (0,T)$
with $\mu$ having infinite variation on 
$[0,T-\delta]$. 
E.g., one could take $\mu$ to be the Weierstrass function or the function $s\mapsto(s\sin\frac1s)1_{(0,T]}(s)$, $s\in[0,T]$. We also take $\rho\in\bbR\setminus\{0\}$ such that \eqref{eq:08032020a1} is satisfied.

Notice that, in this deterministic framework, the process $Y$ is a deterministic continuous function of finite variation explicitly given by~\eqref{eq:02052020a2}. In particular, $Y$ is nonvanishing.

To formally prove that we are in Situation~2, assume by contradiction that there exists a c\`adl\`ag semimartingale $\beta=(\beta_s)_{s\in [0,T]}$ such that $\wt\beta=\beta$ $\cD_W$-a.e.\ ($\beta$ can be stochastic). Then it follows from~\eqref{eq:betatildespecialcaseBrownianmotionsigmazero} and the fact that $Y$ is nonvanishing that
\begin{equation}\label{eq:02052020a3}
\frac\rho{2\rho+\mu}=1-\frac\beta{2Y}\quad\cD_W\text{-a.e.}
\end{equation}
Set $S=1-\frac\beta{2Y}$ and notice that it is a c\`adl\`ag semimartingale. As both sides in~\eqref{eq:02052020a3} are c\`adl\`ag, they are even indistinguishable on $[0,T)$, i.e., almost surely, it holds
\begin{equation}\label{eq:02052020a4}
\frac\rho{2\rho+\mu_s}=S_s,\quad s\in [0,T),
\end{equation}
hence $S\ne0$ and $S_-\ne0$ on $[0,T)$, which implies that $\frac1S$ is also a semimartingale on $[0,T)$. Now \eqref{eq:02052020a4} yields that, almost surely,
$$
\mu_s=\frac\rho{S_s}-2\rho,\quad s\in [0,T).
$$
Thus, $\mu$ is itself a semimartingale on $[0,T)$. As $\mu$ is deterministic, this means that $\mu$ has finite variation on each compact subinterval of $[0,T)$, in particular, on 
$[0,T-\delta]$.  
The obtained contradiction proves that we are in Situation~2.

This example thus shows that an optimal strategy can fail to exist even when the value function is finite.
\end{ex}

\section{Existence for the BSDE in two subsettings}\label{sec:analysis_of_BSDE}

In this section we establish, in two subsettings, existence of a solution $(Y,Z,M^\perp)$ of BSDE~\eqref{eq:bsde} with driver~\eqref{eq:driver_of_bsde} such that $\bm{\left(C_{\textbf{BSDE}}\right)}$ holds. 

We suppose in both subsettings that the following two conditions are satisfied:
\begin{enumerate}
	\item[$\bm{\left(C_{\geq \varepsilon}\right)}$]  \qquad 
	there exists
	$\varepsilon \in (0,\infty)$: \, $2\rho+\mu-\sigma^2\geq \varepsilon$ $\PdM$-a.e., 
\end{enumerate}
\begin{enumerate}
	\item[$\bm{\left(C_{\text{bdd}}\right)}$]  \qquad 
	there exist 
	 $\ol{\rho}, \ol{\mu} \in (0,\infty)$:\,  $|\rho|\leq \ol{\rho}$, $|\mu|\leq \ol{\mu}$  $\PdM$-a.e.
\end{enumerate}

In the first setting we do not impose restrictions on the filtration but assume $\sigma\equiv0$ in order to meet the Lipschitz condition in some place.

Subsequently, we consider a setting with a general $\sigma$, 
where we assume that $(\cF_t)_{t \in [0,T]}$ is a \emph{continuous filtration} in the sense that any $(\cF_t)_{t \in [0,T]}$-martingale is continuous.
This condition is for example satisfied for a Brownian filtration.

\subsection{General filtration and $\sigma\equiv0$}\label{sec:bsde1}

\begin{propo}\label{propo:existence_bsde_general_filtration}
	Let $\sigma\equiv 0$ and assume that $\bm{\left(C_{\geq\varepsilon}\right)}$, $\bm{\left(C_{\text{bdd}}\right)}$ and $\bm{\left(C_{[M]}\right)}$ are satisfied. Then  $\bm{\left(C_{\textbf{BSDE}}\right)}$ holds. 
\end{propo}

\begin{proof}
	We define the truncation function $L\colon \mathbb{R} \to [0,1/2]$ by $L(y)=(y\vee 0) \wedge \frac12$, $y\in\mathbb{R}$, and consider BSDE~\eqref{eq:bsde} with the truncated driver 
	\begin{equation*}
	\begin{split}
	\ol{f} \colon \Omega\times [0,T] \times \mathbb{R} \to \mathbb{R}, \quad
	& \ol{f}(s,y) = -\frac{2(\rho_s+\mu_s)^2 L(y)^2}{2\rho_s + \mu_s} +\mu_s L(y) , \quad s\in [0,T], \, y\in\mathbb{R}, 
	\end{split}
	\end{equation*}
	instead of $f$ defined in~\eqref{eq:driver_of_bsde}. 
	Our aim is to first obtain a solution $(Y,Z,M^\perp)$ of the BSDE with truncated driver via \cite[Theorem 3.5]{papapantoleon2018nineyards} and then show that $Y$ is $[0,1/2]$-valued, i.e., $(Y,Z,M^\perp)$ is also a solution of BSDE~\eqref{eq:bsde} with driver~\eqref{eq:driver_of_bsde}. 
	
	Due to $\bm{\left(C_{\geq\varepsilon}\right)}$, $\bm{\left(C_{\text{bdd}}\right)}$ and the definition of $L$ it holds true that for all $y,y'\in\mathbb{R}$ 
	\begin{equation*}
	|\ol{f}(s,y)-\ol{f}(s,y')| \leq 
	\left( \frac{2(\rho_s + \mu_s)^2}{2\rho_s+\mu_s} + |\mu_s| \right) 
	|y-y'| 
	\leq \left( \frac{6(\ol{\rho}^2+\ol{\mu}^2)}{\varepsilon} + \ol{\mu} \right) |y-y'| \quad \PdM\text{-a.e.}
	\end{equation*}
	Therefore, assumption~(F3) in \cite{papapantoleon2018nineyards} is satisfied. It further follows from $\bm{\left(C_{[M]}\right)}$ that (F2) holds true. The fact that $\ol{f}(s,0)=0$ for all $s\in [0,T]$ yields (F5).  
	Since $M$ is continuous, (F4) is satisfied for all $\Phi>0$. 
	Thus, by \cite[Theorem 3.5]{papapantoleon2018nineyards} (see also Corollary~3.6 therein) there exists a solution $(Y,Z,M^\perp)$ of BSDE~\eqref{eq:bsde} with driver $\ol{f}$. In particular, the norm in \cite[Theorem 3.5]{papapantoleon2018nineyards} being finite implies that $E\left[ [M^\perp]_T\right]<\infty$ and $E\left[ \int_0^T Z_s^2 d[M]_s \right] <\infty$. 
	
	In order to show that $Y$ is $[0,1/2]$-valued, observe that 
	$(\wt{Y},\wt{Z},\wt M^\perp)=\left( \frac{1}{2},0,0 \right)$ (resp. $(\wt{Y},\wt{Z},\wt M^\perp)=(0,0,0)$) solves the BSDE 
	\begin{equation*}
	d\wt{Y}_s = \wt{Z}_s dM_s + d\wt M^\perp_s, \quad s\in[0,T], \quad \wt{Y}_T=\frac12 \quad (\text{resp. }\wt{Y}_T=0),
	\end{equation*} 
	with vanishing driver 
	and that 
	\begin{equation*}
	\ol{f}\left(s,\frac12\right) = \frac{-\rho_s^2}{2\left( 2\rho_s + \mu_s \right)} \leq 0 \quad (\text{resp. } \ol{f}(s,0)=0), \quad s\in[0,T].
	\end{equation*}
Finally, it is possible to verify that a comparison principle holds, which yields that $Y\leq \frac12$ and $Y\geq 0$.\footnote{Although the comparison is performed with standard techniques, we could not locate a precise reference that applies in this situation. Therefore, we present the argument in Appendix~\ref{sec:comp_bsde}.}
\end{proof}

\begin{remark}\label{rem:28032021a1}
	Note that the setting in \cite{papapantoleon2018nineyards} is much more general than ours. Amongst others, the BSDE may include jumps and the Lipschitz continuity of the driver is allowed to be stochastic. 
	E.g., we could replace our conditions $\bm{\left(C_{\geq\varepsilon}\right)}$, $\bm{\left(C_{\text{bdd}}\right)}$ and $\bm{\left(C_{[M]}\right)}$ by $\bm{\left(C_{> 0}\right)}$ together with the more abstract assumption that there exists a predictable stochastic process $R$ such that for all $y,y'\in\mathbb{R}$, $|\ol{f}(\omega,s,y) - \ol{f}(\omega,s,y')|\leq R_s(\omega) |y-y'|$ $\PdM$-a.e.    
	and for all $c\in (0,\infty)$, $E\left[ \exp\left( c\int_0^T R_s d[M]_s \right) \right] <\infty$. Notice, however, that we still need to assume $\sigma \equiv 0$ to obtain, possibly stochastic, Lipschitz continuity. 
	Observe furthermore that the assumptions $\bm{\left(C_{\geq\varepsilon}\right)}$, $\bm{\left(C_{\text{bdd}}\right)}$ and $\bm{\left(C_{[M]}\right)}$ 
	in Proposition~\ref{propo:existence_bsde_general_filtration} seem reasonable in light of the requirements in our main Theorem~\ref{thm:sol_val_fct}.
\end{remark}

\subsection{General $\sigma$ and continuous filtration}\label{sec:bsde2}

\begin{propo}\label{propo:existence_bsde_continuous_filtration}
	Assume that $(\cF_t)_{t \in [0,T]}$ is \emph{continuous}
	 in the sense that any $(\cF_t)_{t \in [0,T]}$-martingale is continuous
	and that $[M]_T\leq c_1$ a.s.\ for some deterministic $c_1\in(0,\infty)$. 
	Suppose  $\bm{\left(C_{\geq\varepsilon}\right)}$ and  $\bm{\left(C_{\text{bdd}}\right)}$. Then  $\bm{\left(C_{\textbf{BSDE}}\right)}$ holds. 
\end{propo}

\begin{proof}
	We first consider BSDE~\eqref{eq:bsde} with the truncated driver 
	\begin{equation*}
	\begin{split}
	& \ol{f} \colon \Omega\times [0,T] \times \mathbb{R} \times \mathbb{R} \to \mathbb{R}, \\
	& \ol{f}(s,y,z) = -\frac{\left( (\rho_s+\mu_s) L(y) + \sigma_s z \right)^2}{\sigma_s^2 L(y) + \frac12 (2\rho_s + \mu_s - \sigma_s^2 )} +\mu_s L(y) + \sigma_s z , \quad s\in [0,T], \, y,z\in\mathbb{R}, 
	\end{split}
	\end{equation*}
	where 
	$L\colon \mathbb{R} \to [0,1/2]$, $L(y)=(y\vee 0) \wedge \frac12$, $y\in\mathbb{R}$. 
	
	Note that $\bm{\left(C_{\geq\varepsilon}\right)}$ and  $\bm{\left(C_{\text{bdd}}\right)}$ imply that $\sigma^2 \leq 2\ol{\rho} + \ol{\mu} - \varepsilon$. 
	Moreover, by $\bm{\left(C_{\geq\varepsilon}\right)}$ and $L\geq 0$ we have $\sigma^2 L(y) + \frac12 (2\rho + \mu - \sigma^2 ) \geq \frac{\varepsilon}{2}$ for all $y\in \mathbb{R}$. 
	Together with $0\leq L\leq \frac12$ and the boundedness of $\rho,\mu$ and $\sigma$ it is thus possible to show that there exist deterministic constants $c_2,c_3 \in (0,\infty)$ such that for all  $y,z \in \mathbb{R}$
	\begin{equation*}
	\begin{split}
		| \ol{f}(s,y,z) | & \leq \frac12 \ol{\mu} + \frac{2}{\varepsilon} \, \Big| (\rho_s + \mu_s)^2 L(y)^2 + 2 (\rho_s + \mu_s ) L(y) \sigma_s z + \sigma_s^2 z^2 + \sigma_s^3 L(y) z \\
		& \quad + \frac12 (2\rho_s + \mu_s - \sigma_s^2) \sigma_s z \Big| \\
		& \leq c_2 + \frac{c_3}{2} z^2 \quad \PdM\text{-a.e.}
	\end{split}
	\end{equation*}
	Furthermore, it holds that $\int_0^T c_2 d[M]_s \leq c_1 c_2$. Hence, assumption $(H_1')$ in \cite{morlais2009quadratic} is satisfied. 
	Observe moreover that $\ol{f}$ is continuous in $(y,z)$. 
	Step 3 and 4 in the proof of \cite[Theorem 2.5]{morlais2009quadratic} show that there exists a solution $(Y,Z,M^\perp)$ of BSDE~\eqref{eq:bsde} with driver $\ol{f}$ 
	and it satisfies $E\left[ \int_0^T Z_s^2 d[M]_s \right]<\infty$, $E\left[ [M^\perp]_T \right] <\infty$ and that $Y$ is bounded. 
	
	In the remainder we prove that $Y$ is $[0,1/2]$-valued, which implies that $(Y,Z,M^\perp)$ is also a solution of BSDE~\eqref{eq:bsde} with driver $f$.
	
	For the upper bound, let $\wh Y=\frac{1}{2}-Y$, $\wh Z=-Z$ and $\wh M^\perp = - M^\perp$. Then it holds that 
	\begin{equation*}
	d\wh Y_t = - \wh f(t,\wh Y_t, \wh Z_t)d[M]_t + \wh Z_t dM_t + d\wh M^\perp _t, \quad t\in[0,T], \quad \wh Y_T = 0,
	\end{equation*}
	where 
	\begin{equation*}
		\begin{split}
		& \wh f(t,\wh Y_t, \wh Z_t) = \frac{
			((\rho_t + \mu_t)L(Y_t)+\sigma_t Z_t)^2}{\sigma_t^2L(Y_t)+\frac{1}{2}\left(2\rho_t +\mu_t-\sigma^2_t\right)}-\frac{1}{2}\mu_t+\wh Y_t \mu_t \frac{L(Y_t)-\frac{1}{2}}{Y_t-\frac{1}{2}}+\sigma_t \wh Z_t \\
		&=\frac{ \wh Y_t \frac{L(Y_t)-\frac{1}{2}}{Y_t-\frac{1}{2}}\left(\frac{\sigma_t^2\mu_t}{2}-(2\rho_t+\mu_t)\mu_t(L(Y_t)+\frac{1}{2})\right)
			-2\wh Z_t \sigma_t (\rho_t+\mu_t)L(Y_t)+\rho^2_tL(Y_t)^2+\sigma_t^2Z_t^2
		}{\sigma_t^2L(Y_t)+\frac{1}{2}\left(2\rho_t +\mu_t-\sigma^2_t\right)}\\
		&\quad+\wh Y_t \mu_t \frac{L(Y_t)-\frac{1}{2}}{Y_t-\frac{1}{2}}+\sigma_t \wh Z_t \\
		& =  \wh Y_t \, \frac{L(Y_t)-\frac{1}{2}}{Y_t-\frac{1}{2}} \left( \frac{\frac{\sigma_t^2\mu_t}{2}-(2\rho_t+\mu_t)\mu_t(L(Y_t)+\frac{1}{2})}{\sigma_t^2L(Y_t)+\frac{1}{2}\left(2\rho_t +\mu_t-\sigma^2_t\right)}  + \mu_t \right) \\
		& \quad + \wh Z_t \left( \frac{-2\sigma_t(\rho_t+\mu_t)L(Y_t)}{\sigma_t^2L(Y_t)+\frac{1}{2}\left(2\rho_t +\mu_t-\sigma^2_t\right)} + \sigma_t \right) 
		 + \frac{\rho^2_tL(Y_t)^2+\sigma_t^2Z_t^2}{\sigma_t^2L(Y_t)+\frac{1}{2}\left(2\rho_t +\mu_t-\sigma^2_t\right)}, 
		 \quad t\in [0,T],
		\end{split}
		\end{equation*}
		with the convention that $0/0:=0$. 
		Denote $\wh f(t,y,z) = y \psi_t + z \eta_t + \varphi_t$, $t\in[0,T]$, $y,z\in\mathbb{R}$, and observe that $\psi$ and $\eta$ are bounded.  
		Since it holds that $\rho^2L(Y)^2+\sigma^2Z^2\ge 0$ and $\sigma^2L(Y)+\frac{1}{2}(2\rho +\mu-\sigma^2)\ge \frac{\varepsilon}{2} >0$, we have that $\varphi \geq 0$ $\PdM$-a.e. 
		Define the process $\Gamma=(\Gamma_t)_{t\in[0,T]}$ by $d\Gamma_t = \Gamma_t \psi_t d[M]_t + \Gamma_t \eta_t dM_t$, $t\in[0,T]$, $\Gamma_0=1$. 
		One can then show 
		that $\wh Y$ has the representation 
		\begin{equation*}
		\wh Y_t = \Gamma_t^{-1} E_t\left[ \int_t^T \Gamma_s \varphi_s d[M]_s \right], \quad t\in[0,T],
		\end{equation*}  
		and hence $\wh Y\ge 0$, i.e., $Y\le \frac{1}{2}$.
		
		Next, we show that $Y$ is nonnegative. To this end we first choose $\delta \in (0,\infty)$ such that $\frac{\delta}{2}\ge \frac{2\sigma^2}{2\rho+\mu-\sigma^2}$ $\PdM$-a.e. Let $h\colon \R \to \R$ be the function $h(y)=1-e^{-\delta y}$, $y\in \R$, and let $\wt Y=(\wt Y_t)_{t\in [0,T]}$ be the process  $\wt Y_t=h(Y_t)$, $t\in [0,T]$. Then it holds for all $t\in [0,T]$ 
		\begin{equation}\label{eq:d_wt_Y_in_proof_cont_filtr}
		\begin{split}
		d\wt Y_t  = h'(Y_t)dY_t+\frac{1}{2}h''(Y_t)d[Y]_t 
		& =-\left\{\ol{f}(t,Y_t,Z_t) h'(Y_t)-\frac{Z^2_th''(Y_t)}{2}\right\}d[M]_t \\
		& \quad +\frac{1}{2}h''(Y_t)d[M^\perp]_t +h'(Y_t)Z_tdM_t+h'(Y_t)dM^\perp_t.
		\end{split}
		\end{equation}
		Let $\wt Z=(\wt Z_t)_{t\in [0,T]}$, $\wt{M}^\perp=(\wt M^\perp_t)_{t\in [0,T]}$ and $A=(A_t)_{t\in [0,T]}$ be the processes  $\wt Z_t=h'(Y_t)Z_t$, $\wt M^\perp_t=\int_0^th'(Y_s)dM^\perp_s$ and $A_t=-\frac{1}{2}\int_0^th''(Y_s)d[M^\perp]_s$, $t\in [0,T]$.
		Observe that it holds $h'(y)=\delta e^{-\delta y}=\delta(1-h(Y_t))$ and $h''(y)=-\delta^2e^{-\delta y}=-\delta h'(y)$ for all $y\in \R$. In particular, the process $A$ is nondecreasing. We obtain from~\eqref{eq:d_wt_Y_in_proof_cont_filtr} that for all $t\in[0,T]$ 
		\begin{equation*}
		\begin{split}
		d\wt Y_t&=-\Bigg\{-\frac{
			\delta (\rho_t + \mu_t)^2L(Y_t)^2(1-\wt Y_t)+2\sigma_t(\rho_t + \mu_t)L(Y_t) \wt Z_t}{\sigma_t^2L(Y_t)+\frac{1}{2}\left(2\rho_t +\mu_t-\sigma^2_t\right)}+\delta \mu_t L(Y_t)(1-\wt Y_t)+\sigma_t \wt Z_t\\
		&\quad +Z^2_th'(Y_t)\left(\frac{\delta}{2}-\frac{\sigma_t^2}{\sigma_t^2L(Y_t)+\frac{1}{2}\left(2\rho_t +\mu_t-\sigma^2_t\right)} \right)\Bigg\}d[M]_t 
		-dA_t +\wt Z_tdM_t+d\wt M^\perp_t\\
		&=-\Bigg\{\wt Y_t \frac{\delta L(Y_t)(1-\wt Y_t)}{\wt Y_t} \left( \mu_t -\frac{
			(\rho_t + \mu_t)^2L(Y_t)}{\sigma_t^2L(Y_t)+\frac{1}{2}\left(2\rho_t +\mu_t-\sigma^2_t\right)}\right)\\
		&\quad+\sigma_t \wt Z_t \left(1-\frac{2(\rho_t + \mu_t)L(Y_t)}{\sigma_t^2L(Y_t)+\frac{1}{2}\left(2\rho_t +\mu_t-\sigma^2_t\right)}\right) \\ &\quad +Z^2_th'(Y_t)\left(\frac{\delta}{2}-\frac{\sigma_t^2}{\sigma_t^2L(Y_t)+\frac{1}{2}\left(2\rho_t +\mu_t-\sigma^2_t\right)} \right)\Bigg\}d[M]_t 
		 -dA_t +\wt Z_tdM_t+d\wt M^\perp_t.
		\end{split}
		\end{equation*}
		Denote the coefficients of $\wt Y$ resp. $\wt Z$ by 
		\begin{equation*}
		\begin{split}
		\wt \psi_t & = \frac{\delta L(Y_t)(1-\wt Y_t)}{\wt Y_t} \left( \mu_t -\frac{
			(\rho_t + \mu_t)^2L(Y_t)}{\sigma_t^2L(Y_t)+\frac{1}{2}\left(2\rho_t +\mu_t-\sigma^2_t\right)}\right), \quad t\in [0,T],	\\
		\wt \eta_t & = \sigma_t \left(1-\frac{2(\rho_t + \mu_t)L(Y_t)}{\sigma_t^2L(Y_t)+\frac{1}{2}\left(2\rho_t +\mu_t-\sigma^2_t\right)}\right), \quad t\in [0,T],		
		\end{split}
		\end{equation*}
		an define $\wt \Gamma$ by $d\wt \Gamma_t = \wt \Gamma_t \wt\psi_t d[M]_t + \wt \Gamma_t \wt\eta_t dM_t$, $t\in[0,T]$, $\wt\Gamma_0 = 1$. 
		Note that the process $\frac{\delta L(Y_t)(1-\wt Y_t)}{\wt Y_t}=\frac{\delta L(Y_t)e^{-\delta Y_t}}{1-e^{-\delta Y_t}}$, $t\in [0,T]$, is bounded. 
		Together with $\bm{\left(C_{\geq\varepsilon}\right)}$,   $\bm{\left(C_{\text{bdd}}\right)}$ and $0\leq L\leq \frac12$ it follows that $\wt \psi$ and $\wt \eta$ are bounded. 
		One can then show that 
		\begin{equation}\label{eq:representation_of_supersol_Y}
		\begin{split}
		\wt Y_t \wt\Gamma_t & = E_t\bigg[ \wt Y_T \wt \Gamma_T + \int_t^T \wt\Gamma_s Z_s^2  h'(Y_s)\left(\frac{\delta}{2}-\frac{\sigma_s^2}{\sigma_s^2L(Y_s)+\frac{1}{2}\left(2\rho_s +\mu_s-\sigma^2_s\right)} \right) d[M]_s \\
		& \qquad + \int_t^T \wt\Gamma_s dA_s \bigg] , 
		\quad t \in [0,T].
		\end{split}
		\end{equation}		
		Due to the choice of $\delta$ we have that $Z^2h'(Y)\left(\frac{\delta}{2}-\frac{\sigma^2}{\sigma^2L(Y)+\frac{1}{2}\left(2\rho +\mu-\sigma^2\right)} \right) \geq 0$ $\PdM$-a.e.  
		Since furthermore $A$ is nondecreasing and $\wt Y$ has nonnegative terminal value $1-e^{-\frac{\delta}{2}}$, it follows from~\eqref{eq:representation_of_supersol_Y} that $\wt Y\geq 0$ and hence $Y\geq 0$.
\end{proof}

\section{Proofs of results from Section~\ref{sec:main_results}}\label{sec:proofs_of_main_results}

We first present a technical lemma that is used in the proof of Theorem~\ref{thm:quadratic_exp_J}. 

\begin{lemma}\label{lem:Y_Tminus_equals_Y_T}
	Let $\left( Y, Z, M^\perp \right)$ be a solution of BSDE~\eqref{eq:bsde} as described in $\bm{\left(C_{\textbf{BSDE}}\right)}$. 
	Then $Y_{T-}=\frac12$ a.s., i.e., $Y$ does not jump at terminal time.
\end{lemma}

\begin{proof}
	We have, with $f$ defined in \eqref{eq:driver_of_bsde}, that 
	\begin{equation}\label{eq:17022020a1}
	Y_t = \frac12 + E_t\left[ \int_t^T f(s,Y_s,Z_s) d[M]_s \right] 
	= \frac12 + E_t[A_T] - A_t, \quad t\in [0,T],
	\end{equation}
	where $A_t = \int_0^t f(s,Y_s,Z_s) d[M]_s$, $t \in [0,T]$. 
	As $A=\left( A_t \right)_{t\in [0,T]}$ is a continuous process, it holds $\lim_{t \uparrow T} A_t = A_T$, hence $A_T$ is $\cF_{T-}$-measurable. Therefore, 
	\begin{equation*}
	\lim_{t \uparrow T} E_t[A_T] = E\left[A_T | \cF_{T-} \right] = A_T  \text{ a.s.}
	\end{equation*}
	The result now follows from~\eqref{eq:17022020a1}.
\end{proof}

We furthermore introduce the following lemma that we employ  
in the proofs of Theorem~\ref{thm:quadratic_exp_J}, 
Lemma~\ref{lem:uniqueness_of_opt_strat_up_to_Dm_null_sets}, Lemma~\ref{lem:Xifbetancadlagsemimart}, and  Theorem~\ref{thm:sol_val_fct}. 
It provides helpful representations for the dynamics of the process $A=X-\alpha D$ where $X$ is an execution strategy and $D$ its deviation.

\begin{lemma}\label{lem:unifieddAcalc}
	Let $x,d \in \R$ and $t \in [0,T]$. 
	Suppose that $X=(X_s)_{s\in[t,T]}$ is a c\`adl\`ag semimartingale with $X_{t-}=x$ and $X_T=0$, and let $D=(D_s)_{s\in[t,T]}$ be the associated deviation process given by \eqref{eq:deviation_dyn}. 
	It then holds for $A=(A_s)_{s\in[t,T]}$ defined by $A_s=X_s-\alpha_sD_s$, $s \in [t,T]$, that 
	\begin{equation*}
		\begin{split}
			dA_s & = -D_s d\alpha_s + \alpha_s \rho_s D_sd[M]_s 
			= (A_s-X_s) \left(\frac{d\alpha_s}{\alpha_s} -  \rho_s d[M]_s\right), \quad s \in [t,T].
		\end{split}
	\end{equation*} 
\end{lemma}

\begin{proof}
	Note first that it follows from \eqref{eq:dyn_gamma}, \eqref{eq:dyn_alpha} and \eqref{eq:deviation_dyn} that $d[\alpha,D]_s = -\alpha_s d[\gamma,X]_s$, $s \in [t,T]$. 
	By integration by parts and \eqref{eq:deviation_dyn} it then holds for all $s\in[t,T]$ that 
	\begin{equation*}	
		\begin{split}
			dA_s & = dX_s - D_{s-} d\alpha_s - \alpha_s dD_s - d[\alpha,D]_s \\
			& = dX_s - D_sd\alpha_s - \alpha_s\rho_sD_s d[M]_s - \alpha\gamma_s dX_s - \alpha_s d[\gamma,X]_s + \alpha_s d[\gamma,X]_s \\	
			& = -D_sd\alpha_s + \alpha_s \rho_s D_s d[M]_s. 
		\end{split}
	\end{equation*}
	The second equality in the claim now follows from the fact that $-D_s=\frac{A_s-X_s}{\alpha_s}$, $s \in [t,T]$, by definition of $A$.
\end{proof}

\begin{proof}[Proof of Theorem \ref{thm:quadratic_exp_J}]
We fix $x,d\in \R$, $t\in [0,T]$ and $X\in \cA_t(x,d)$ throughout the proof.
First observe that it follows from \eqref{eq:deviation_dyn} and the fact that $\gamma=1/\alpha$ that for all $s\in [t,T]$ it holds $d[\alpha,D]_s=-\alpha_sd[\gamma,X]_s$ and $d[D]_s=\frac{1}{\alpha^2_s}d[X]_s$.
This shows that
\begin{equation}\label{eq:alpha_T_mult_D_T_squared}
\begin{split}
\alpha_TD_{T-}^2&=\alpha_td^2+\int_{[t,T)}\alpha_sdD_s^2+\int_{[t,T)}D_s^2d\alpha_s+\int_{[t,T)}d[\alpha,D^2]_s\\
&=\alpha_td^2+2\int_{[t,T)}\alpha_sD_{s-} dD_s+\int_{[t,T)}\alpha_s d[D]_s+\int_{[t,T)}D_s^2d\alpha_s+2\int_{[t,T)}D_{s-} d[\alpha,D]_s\\
&=\alpha_td^2-2\int_{[t,T)}\rho_s\alpha_sD^2_{s-} d[M]_s
+2\int_{[t,T)}D_{s-} dX_s
+2\int_{[t,T)} \alpha_sD_{s-} d[\gamma, X]_s\\
&\quad+\int_{[t,T)}\frac{1}{\alpha_s} d[X]_s+\int_{[t,T)}D_s^2d\alpha_s-2\int_{[t,T)}\alpha_sD_{s-}d[\gamma,X]_s\\
& = \alpha_t d^2 - \int_t^T \alpha_s D_s^2 \left( 2 \rho_s + \mu_s - \sigma_s^2 \right) d[M]_s 
+ 2 \int_{[t,T)} D_{s-} dX_s 
+ \int_{[t,T)} \gamma_s d[X]_s \\
& \quad    - \int_t^T \alpha_s D_s^2 \sigma_s dM_s .
\end{split}
\end{equation}
The first equality in Lemma~\ref{lem:unifieddAcalc} and~\eqref{eq:dyn_alpha} prove for all $s\in [t,T]$ that
\begin{equation*}
	d\left(X_s-\alpha_sD_s\right)
	= \alpha_s D_{s}(\rho_s +\mu_s -\sigma^2_s)d[M]_s+\sigma_s \alpha_s D_{s} dM_s.
\end{equation*}
In particular, the process $X-\alpha D$ has continuous sample paths. Moreover, it follows for all $s\in [t,T]$ that
\begin{equation}\label{eq:dyn_cont_proc}
\begin{split}
&d\left(X_s-\alpha_sD_s\right)^2=
2\left(X_s-\alpha_sD_s\right)d\left(X_s-\alpha_sD_s\right)+d\left[X-\alpha D\right]_s\\
&=
2\left(X_s-\alpha_sD_s\right)\left(\alpha_s D_{s}(\rho_s +\mu_s-\sigma^2_s )d[M]_s+\sigma_s \alpha_s D_{s}dM_s \right)+\sigma^2_s \alpha^2_s D^2_{s}d[M]_s\\
&=
\alpha_s D_{s}\left[2\left(X_s-\alpha_sD_s\right)(\rho_s +\mu_s-\sigma^2_s )+\sigma^2_s \alpha_s D_{s}\right]d[M]_s +2\sigma_s \alpha_s D_{s}\left(X_s-\alpha_sD_s\right)dM_s.
\end{split}
\end{equation}
Next observe that \eqref{eq:dyn_gamma} and \eqref{eq:bsde} imply for all $s\in [t,T]$ that
\begin{equation*}
\begin{split}
d(\gamma_s Y_s)&=\gamma_sY_s(\mu_sd[M]_s+\sigma_sdM_s) +
\gamma_s\left[\frac{
((\rho_s + \mu_s)Y_s+\sigma_sZ_s)^2}{\sigma_s^2Y_s+\frac{1}{2}\left(2\rho_s +\mu_s-\sigma^2_s\right)}-\mu_s Y_s -\sigma_s Z_s\right]d[M]_s\\
&\quad +\gamma_s Z_sdM_s +\gamma_s dM^{\perp}_s+\gamma_s\sigma_sZ_sd[M]_s\\
&=
\frac{\gamma_s
((\rho_s + \mu_s)Y_s+\sigma_sZ_s)^2}{\sigma_s^2Y_s+\frac{1}{2}\left(2\rho_s +\mu_s-\sigma^2_s\right)}d[M]_s+\gamma_s(\sigma_sY_s+Z_s) dM_s+\gamma_s dM^{\perp}_s.
\end{split}
\end{equation*}
This and \eqref{eq:dyn_cont_proc} prove that
\begin{equation}\label{eq:prodgammaYwithsquaredcontproc}
\begin{split}
&\gamma_TY_{T-}\left(X_{T-}-\alpha_TD_{T-}\right)^2\\
&=\gamma_tY_{t}\left(x-\alpha_t d\right)^2
+\int_{(t,T)}\gamma_sY_{s-}d(X_s-\alpha_sD_s)^2
+\int_{(t,T)}
\left(X_{s}-\alpha_s D_s\right)^2d(\gamma_sY_s)\\
&\quad+[\gamma Y,(X-\alpha D)^2]_{T-}\\
& = \gamma_t Y_t (x-\alpha_td)^2 
+ \int_t^T \Bigg(  D_s Y_s \left( 2(X_s-\alpha_s D_s) (\rho_s + \mu_s - \sigma_s^2) + \sigma_s^2 \alpha_s D_s \right) \\
& \qquad + (X_s - \alpha_s D_s)^2 \gamma_s  \frac{
((\rho_s + \mu_s)Y_s+\sigma_sZ_s)^2}{\sigma_s^2 Y_s+\frac{1}{2}\left(2\rho_s +\mu_s-\sigma^2_s\right)} 
+ 2 \sigma_s (\sigma_s Y_s + Z_s) D_s (X_s - \alpha_s D_s) \Bigg) d[M]_s \\
& \quad + \int_t^T \left( 2\sigma_s D_s Y_s (X_s - \alpha_s D_s) + \gamma_s (X_s - \alpha_s D_s)^2 (\sigma_s Y_s + Z_s) \right) dM_s \\
& \quad + \int_{(t,T)} \gamma_s (X_s - \alpha_s D_s )^2 dM^\perp_s .
\end{split}
\end{equation}
Since $Y_{T-}=\frac12$ by Lemma~\ref{lem:Y_Tminus_equals_Y_T}, it holds that 
\begin{equation*}
-\left(D_{T-}-\frac{1}{2\alpha_T}X_{T-}\right)X_{T-}=\gamma_TY_{T-}\left(X_{T-}-\alpha_T D_{T-}\right)^2-\frac{\alpha_TD_{T-}^2}{2}.
\end{equation*}
This, \eqref{eq:alpha_T_mult_D_T_squared}, \eqref{eq:prodgammaYwithsquaredcontproc} and the fact that $X_T=0$ show that 
\begin{equation*}
\begin{split}
& \int_{[t,T]} D_{s-} dX_s + \int_{[t,T]} \frac{\gamma_s}{2} d[X]_s \\
& = \int_{[t,T)}D_{s-}dX_s+\int_{[t,T)}\frac{\gamma_s}{2}d[X]_s+\gamma_TY_{T-}\left(X_{T-}-\alpha_T D_{T-}\right)^2-\frac{\alpha_TD_{T-}^2}{2} \\
& = \gamma_tY_{t}\left(x-\alpha_td\right)^2-\frac{\alpha_t d^2}{2} 
+ \int_t^T  \Bigg( 
D_s Y_s \left( 2(X_s-\alpha_s D_s) (\rho_s + \mu_s - \sigma_s^2) + \sigma_s^2 \alpha_s D_s \right) \\
& \qquad + (X_s - \alpha_s D_s)^2 \gamma_s  \frac{
((\rho_s + \mu_s)Y_s+\sigma_sZ_s)^2}{\sigma_s^2 Y_s+\frac{1}{2}\left(2\rho_s +\mu_s-\sigma^2_s\right)} 
+ 2 \sigma_s (\sigma_s Y_s + Z_s) D_s (X_s - \alpha_s D_s)  \\
& \qquad + \alpha_s D_s^2 \frac{2\rho_s + \mu_s - \sigma_s^2}{2}
\Bigg) d[M]_s
+ \int_t^T \Bigg( 
2\sigma_s D_s Y_s (X_s - \alpha_s D_s) + \gamma_s (X_s - \alpha_s D_s)^2 \\
& \qquad  \cdot (\sigma_s Y_s + Z_s) + \alpha_s D_s^2 \frac{\sigma_s}{2}
\Bigg) dM_s
+ \int_{(t,T)} \gamma_s (X_s - \alpha_s D_s )^2 dM^\perp_s.
\end{split}
\end{equation*}
We thus have, with  $\wt{\beta}$ defined by \eqref{eq:def_beta_tilde} and $J$ given by \eqref{eq:cost_fct}, that 
\begin{equation*}
\begin{split}
J_t(x,d,X) &  = 
 \frac{Y_t}{\gamma_t}\left(d-\gamma_t x\right)^2-\frac{d^2}{2 \gamma_t}  \\
& \quad + E_t\Bigg[ \int_t^T  
\frac{1}{\gamma_s} 
\left( \wt{\beta}_s (\gamma_s X_s - D_s ) + D_s  \right)^2
\left( \sigma_s^2 Y_s + \frac12 (2\rho_s + \mu_s - \sigma_s^2 ) \right)
d[M]_s \\
& \quad + \int_t^T \left( 
2\sigma_s D_s Y_s (X_s - \alpha_s D_s) + \gamma_s (X_s - \alpha_s D_s)^2 (\sigma_s Y_s + Z_s) + \alpha_s D_s^2 \frac{\sigma_s}{2}
\right) dM_s \\
& \quad + \int_{(t,T)} \gamma_s (X_s - \alpha_s D_s )^2 dM^\perp_s 
\Bigg] .
\end{split}
\end{equation*}
It therefore remains to show that 
\begin{equation}\label{eq:stoch_int_wrt_M_is_0}
E_t\left[ \int_t^T \left( 
2\sigma_s D_s Y_s (X_s - \alpha_s D_s) + \gamma_s (X_s - \alpha_s D_s)^2 (\sigma_s Y_s + Z_s) + \alpha_s D_s^2 \frac{\sigma_s}{2}
\right) dM_s  \right] = 0
\end{equation}
and 
\begin{equation}\label{eq:stoch_int_wrt_Mperp_is_0}
E_t\left[ \int_{(t,T)} \gamma_s (X_s - \alpha_s D_s )^2 dM^\perp_s \right] = 0 .
\end{equation}
Consider first the stochastic integral $\int_t^T \gamma_s (X_s - \alpha_s D_s)^2 Z_s dM_s$. By the Burkholder-Davis-Gundy inequality it holds that for some constant $c \in (0,\infty)$, 
\begin{equation*}
 E_t\left[ \sup_{r \in [t,T]} \left| \int_t^r \gamma_s (X_s - \alpha_s D_s)^2 Z_s dM_s \right|  \right] 
 \leq c  E_t\left[ \left( \int_t^{T} \gamma_s^2 (X_s - \alpha_s D_s)^4 Z_s^2 d[M]_s \right)^{1/2} \right] .
\end{equation*}
Since $E_t\left[ \int_t^T Z_s^2 d[M]_s \right] < \infty$ and (A1) hold true, it follows from the Cauchy-Schwarz inequality that 
\begin{equation*}
\begin{split}
&  E_t\left[ \sup_{r \in [t,T]} \left| \int_t^r \gamma_s (X_s - \alpha_s D_s)^2 Z_s dM_s \right|  \right] \\
& \leq c E_t\left[ \sup_{s \in [t,T]} \left( \gamma_s (X_s - \alpha_s D_s)^2 \right) \cdot \left( \int_t^{T} Z_s^2 d[M]_s \right)^{1/2} \right] \\
& \leq c \left( E_t\left[ \sup_{s \in [t,T]} \left( \gamma_s^2 (X_s - \alpha_s D_s)^4 \right) \right] \right)^{1/2} \left( E_t\left[ \int_t^{T} Z_s^2 d[M]_s \right] \right)^{1/2} 
 < \infty.
\end{split}
\end{equation*}
Therefore, $\int_t^{\cdot} \gamma_s (X_s - \alpha_s D_s)^2 Z_s dM_s$ is a true martingale, and $E_t\left[ \int_t^{T} \gamma_s (X_s - \alpha_s D_s)^2 Z_s dM_s \right]$ $ = 0$. 
Similarly, $E_t\left[ [M^\perp ]_T \right] < \infty$ and (A1) imply \eqref{eq:stoch_int_wrt_Mperp_is_0}. 
Furthermore, we obtain from (A2) and the fact that $Y$ is bounded that 
$E_t\left[ \left( \int_t^T \gamma_s^2 (X_s - \alpha_s D_s)^4  \sigma_s^2 Y_s^2 d [M]_s \right)^{1/2} \right] < \infty$ 
and hence 
$E_t\left[ \int_t^T \gamma_s (X_s - \alpha_s D_s)^2  \sigma_s Y_s d M_s \right] = 0.$ 
To show 
$E_t\left[ \int_t^T 2  \sigma_s D_{s} Y_s (X_s-\alpha_sD_s) dM_s \right] = 0$ 
observe that by Young's inequality $D_{s}^2 (X_s - \alpha_s D_s)^2 \leq \frac{1}{2} \left(D_{s}^4 \alpha_s^2 + \gamma_s^2 (X_s - \alpha_s D_s)^4\right)$, $s\in [t,T]$. 
This together with $0\leq Y \leq \frac12$, (A3) and (A2) yields 
\begin{equation*}
\begin{split}
& E_t\left[ \left( \int_t^T 4 \sigma_s^2 D_{s}^2 Y_s^2 (X_s-\alpha_sD_s)^2 d[M]_s \right)^{1/2} \right] \\
& \leq \frac{1}{\sqrt{2}} E_t\left[ \left( \int_t^T \sigma_s^2 D_{s}^4 \alpha_s^2 d[M]_s \right)^{1/2} \right] + \frac{1}{\sqrt{2}} E_t\left[ \left( \int_t^T \sigma_s^2 \gamma_s^2 (X_s-\alpha_sD_s)^4 d[M]_s \right)^{1/2} \right] 
 < \infty .
\end{split}
\end{equation*}
Moreover, it follows from (A3) that also 
$E_t\left[ \int_t^T \alpha_s D_s^2  \sigma_s d M_s \right] = 0.$ 
We thus have established \eqref{eq:stoch_int_wrt_M_is_0} and \eqref{eq:stoch_int_wrt_Mperp_is_0}, which completes the proof.
\end{proof}

The following uniqueness result for optimal strategies, which relies on the representation of the cost functional~\eqref{eq:quad_expansion_J}, is applied in the proofs of Lemma~\ref{lem:optstratifzero}, Theorem~\ref{thm:sol_val_fct}, and Proposition~\ref{propo:withoutresiliencecloseimmediately}.

\begin{lemma}\label{lem:uniqueness_of_opt_strat_up_to_Dm_null_sets}
	Suppose $\bm{\left(C_{>0}\right)}$ and $\bm{\left(C_{\textbf{BSDE}}\right)}$ 
	and fix a solution $(Y,Z,M^\perp)$ of BSDE~\eqref{eq:bsde} that satisfies the properties in $\bm{\left(C_{\textbf{BSDE}}\right)}$. Let $\wt\beta$ be the process defined by~\eqref{eq:def_beta_tilde} pertaining to $(Y,Z)$. 
	Let $x,d \in \R$ and $t\in[0,T]$. 
	Assume that $V_t(x,d)= \frac{Y_{t}}{\gamma_t}\left(d-\gamma_t x\right)^2-\frac{d^2}{2\gamma_t}$ 
	and that there exists an optimal strategy in $\cA_t(x,d)$. 
	Then the optimal strategy is unique up to $\PdM|_{[t,T]}$-null sets.
\end{lemma} 

\begin{proof}
	Let $X^*,X \in \cA_t(x,d)$ be two optimal strategies with associated deviation processes $D^*$ and $D$, respectively. 
	Combine the assumption $V_t(x,d)= \frac{Y_{t}}{\gamma_t}\left(d-\gamma_t x\right)^2-\frac{d^2}{2\gamma_t}$ with Theorem~\ref{thm:quadratic_exp_J} to obtain that 
	\begin{equation*}
	E_t\left[\int_t^T  \frac{1}{\gamma_s} 
	\left( \wt{\beta}_s (\gamma_s X^*_s - D^*_s ) + D^*_s  \right)^2
	\left( \sigma_s^2 Y_s + \frac12 (2\rho_s + \mu_s - \sigma_s^2 ) \right)
	d[M]_s \right] 
	= 0 \text{ a.s.}
	\end{equation*}
	By taking expectations, it follows that 
	\begin{equation*}
	E\left[\int_t^T  \frac{1}{\gamma_s} 
	\left( \wt{\beta}_s (\gamma_s X^*_s - D^*_s ) + D^*_s  \right)^2
	\left( \sigma_s^2 Y_s + \frac12 (2\rho_s + \mu_s - \sigma_s^2 ) \right)
	d[M]_s \right] 
	= 0 . 
	\end{equation*}
	This implies 
	\begin{equation}\label{eq:25022020a1}
	\wt{\beta} (\gamma X^* - D^* ) + D^* = 0 
	\quad \PdM|_{[t,T]}\text{-a.e.}
	\end{equation}
	By Lemma~\ref{lem:unifieddAcalc} this  
	further yields for the process $A^*=(A^*_s)_{s\in[t,T]}$ defined by 
	$A^*_s = X^*_s - \alpha_s D^*_s$, $s \in [t,T],$ that 
	\begin{equation*}
			dA^*_s = \wt{\beta}_s A^*_s\left(\frac{d\alpha_s}{\alpha_s} -  \rho_s d[M]_s\right), \quad s \in [t,T] . 
	\end{equation*}
	For $X$, $D$ and $A=X -\alpha D$ we analogously obtain \eqref{eq:25022020a1} 
	and 
	\begin{equation*}
	dA_s = \wt{\beta}_s A_s\left(\frac{d\alpha_s}{\alpha_s} -  \rho_s d[M]_s\right), \quad s \in [t,T] . 
	\end{equation*}	
	Hence, $A$ and $A^*$ satisfy the same dynamics and have the same starting point $A_t = x-\alpha_t d = A_t^*$. 
	It follows that $A$ and $A^*$ are indistinguishable. 
	Together with \eqref{eq:25022020a1} this yields that $D = -\wt{\beta} \gamma A = -\wt{\beta} \gamma A^* = D^*$ $\PdM|_{[t,T]}$-a.e. 
	Finally, it follows from the definition of $A$ and $A^*$ that $X=X^*$ $\PdM|_{[t,T]}$-a.e.
\end{proof}

\begin{proof}[Proof of Lemma \ref{lem:optstratifzero}]
Suppose that  $x=\frac{d}{\gamma_t}$. Let $X^*=(X^*_s)_{s\in [t,T]}$ be defined by $X^*_{t-}=x$, $X^*_s = 0$, $s \in [t,T]$. Then, $X^*$ is a c\`adl\`ag semimartingale with $X^*_{t-}=x$ and $X^*_T=0$. The associated deviation process $D^*=(D^*_s)_{s\in [t,T]}$ satisfies $D^*_t=d+\Delta D^*_t=d+\gamma_t \Delta X^*_t=d-\gamma_tx=0$ and hence $D^*_s=0$ for all $s\in [t,T]$. 
It follows that $X^*_s - \alpha_s D^*_s = 0$, $s\in [t,T]$, and thus conditions (A1), (A2) and (A3) are satisfied, i.e., $X^* \in \mathcal{A}_t(x,d)$. Since $D^*_s=0$ and $\gamma_s X^*_s - D^*_s = 0$ for all $s\in [t,T]$, Theorem  \ref{thm:quadratic_exp_J} yields that $X^*$ is optimal and that $V_t(x,d) = \frac{Y_{t}}{\gamma_t}\left(d-\gamma_t x\right)^2-\frac{d^2}{2\gamma_t} = -\frac{d^2}{2\gamma_t}$. 
Uniqueness up to $\PdM|_{[t,T]}$-null sets follows from Lemma~\ref{lem:uniqueness_of_opt_strat_up_to_Dm_null_sets}.
\end{proof}

For the proof of Lemma~\ref{lem:Xifbetancadlagsemimart} we need the following technical lemma. 
It provides conditions which ensure that the conditional expectation of the supremum of a process with a certain exponential structure
is a.s. finite (see also Remark~\ref{rem:26032021a1} below).

\begin{lemma}\label{lem:expectation_sup_N_finite}
Suppose that $\bm{\left(C_{[M]}\right)}$ is satisfied. 
Let $\eta=(\eta_s)_{s\in [0,T]}$ and $\nu=(\nu_s)_{s\in [0,T]}$ be progressively measurable processes 
such that $|\eta|\leq c_{\eta}$ and $|\nu| \leq c_{\nu}$ $\PdM$-a.e. 
 for some constants $c_{\eta}, c_{\nu} \in (0,\infty)$. 
 Let $t \in [0,T]$ and define $N=\left( N_s \right)_{s\in[t,T]}$ by 
\begin{equation*}
N_s = \exp\left( \int_t^s \eta_r dM_r + \int_t^s \nu_r d[M]_r \right), \quad s\in [t,T].
\end{equation*}
It then holds that 
\begin{equation*}
E_t\left[ \sup_{s\in[t,T]} N_s \right] 
 \leq 2 \left( E_t\left[ e^{ 6 c_{\eta}^2 \left( [M]_T - [M]_t \right)} \right]  \right)^{1/4} \left( E_t\left[ e^{( 2 c_{\nu} + c_{\eta}^2 ) \left( [M]_T - [M]_t \right)} \right]  \right)^{1/2} 
< \infty \text{ a.s.}
\end{equation*}
\end{lemma}

\begin{proof}
We introduce the process $L=\left( L_s \right)_{s\in[t,T]}$ defined by 
\begin{equation*}
L_s = \exp\left( \int_t^s \eta_r dM_r - \frac12 \int_t^s \eta_r^2 d[M]_r \right), \quad s\in [t,T].
\end{equation*}
We then have 
\begin{equation*}
N_s = L_s \exp\left( \int_t^s \left( \nu_r + \frac{\eta_r^2}{2} \right) d[M]_r \right), \quad s\in [t,T], 
\end{equation*}
and thus by the Cauchy-Schwarz inequality 
\begin{equation}\label{eq:supN_CS_L_squared}
E_t\left[ \sup_{s\in[t,T]} N_s \right] \leq 
\left( E_t\left[ \sup_{s\in[t,T]} L_s^2 \right]  \right)^{1/2}
\left( E_t\left[ \sup_{s\in[t,T]} \exp\left( \int_t^s \left( 2\nu_r + \eta_r^2 \right) d[M]_r \right) \right] \right)^{1/2} .
\end{equation}
Since $2\nu + \eta^2$ is bounded by $2c_{\nu} + c_{\eta}^2$, 
it holds that 
\begin{equation}\label{eq:sup_exp_2nu_plus_eta_squared_inequ}
E_t\left[ \sup_{s\in[t,T]} \exp\left( \int_t^s \left( 2\nu_r + \eta_r^2 \right) d[M]_r \right) \right] 
\leq E_t\left[ e^{( 2 c_{\nu} + c_{\eta}^2 ) \left( [M]_T - [M]_t \right)} \right] .
\end{equation}
Next, observe that 
\begin{equation*}
E_t\left[ \exp\left( \frac12 \int_t^T \eta_r^2 d[M]_r  \right) \right] < \infty \text{ a.s.}
\end{equation*}
because $\eta^2$ is bounded and $\bm{\left(C_{[M]}\right)}$ is assumed to hold. Therefore, we obtain by Novikov's criterion that $L$ is a true martingale. Thus, it follows from Doob's maximal inequality that 
\begin{equation}\label{eq:Doob_L_squared}
\left(E_t\left[  \sup_{s\in[t,T]} L_s^2  \right]\right)^{1/2} \leq 2 \left(E_t\left[  L_T^2  \right]\right)^{1/2}.
\end{equation}
We define $\wt{L} = \left( \wt{L}_s \right)_{s\in[t,T]}$ by 
\begin{equation*}
\wt{L}_s = \exp\left( \int_t^s 4 \eta_r dM_r - \frac12 \int_t^s \left( 4 \eta_r \right)^2 d[M]_r \right), \quad s\in [t,T],
\end{equation*}
and observe that by the Cauchy-Schwarz inequality it holds 
\begin{equation}\label{eq:L_squared_CS_L_tilde}
\begin{split}
E_t\left[  L_T^2  \right] 
& = E_t\left[  \exp\left( \int_t^T 2\eta_r dM_r - \int_t^T 4 \eta_r^2 d[M]_r \right) \exp\left( \int_t^T 3\eta_r^2 d[M]_r \right) \right] \\
& \leq \left(  E_t\left[ \wt{L}_T \right]  \right)^{1/2} \left(  E_t\left[ \exp\left( \int_t^T 6\eta_r^2 d[M]_r \right) \right]  \right)^{1/2} .
\end{split}
\end{equation}
As a nonnegative local martingale, $\wt L$ is a supermartingale, hence $E_t[ \wt{L}_T ]\le\wt L_t=1$. 
Together with the fact that $\eta^2$ is bounded by $c_{\eta}^2$, we obtain from \eqref{eq:L_squared_CS_L_tilde} that 
\begin{equation}\label{eq:L_squared_leq_exp_6eta_squared}
E_t\left[  L_T^2  \right] 
 \leq \left(  E_t\left[ e^{ 6 c_{\eta}^2 \left( [M]_T - [M]_t \right)} \right] \right)^{1/2} .
\end{equation}
It follows from \eqref{eq:supN_CS_L_squared}, \eqref{eq:sup_exp_2nu_plus_eta_squared_inequ}, \eqref{eq:Doob_L_squared} and \eqref{eq:L_squared_leq_exp_6eta_squared} that 
\begin{equation*}
E_t\left[ \sup_{s\in[t,T]} N_s \right] 
 \leq 2 \left( E_t\left[ e^{ 6 c_{\eta}^2 \left( [M]_T - [M]_t \right)} \right]  \right)^{1/4} \left( E_t\left[ e^{( 2 c_{\nu} + c_{\eta}^2 ) \left( [M]_T - [M]_t \right)} \right]  \right)^{1/2} .
\end{equation*}
This is a.s. finite due to $\bm{\left(C_{[M]}\right)}$ .
\end{proof}

\begin{remark}\label{rem:26032021a1}
The same computations also yield $E_t\left[ \sup_{s\in[t,T]} N_s \right] < \infty$ a.s.\ 
whenever we replace $\bm{\left(C_{[M]}\right)}$ and the boundedness assumption on $\eta$ and $\nu$ by 
\begin{itemize}
\item $E_t\left[\exp\left\{\int_t^T6\eta_s^2\,d[M]_s\right\}\right]<\infty$ a.s.\ and
\item $E_t\left[\exp\left\{\int_t^T(2\nu_s+\eta_s^2)^+\,d[M]_s\right\}\right]<\infty$ a.s.
\end{itemize}
This observation is used in Remark~\ref{rem:26032021a2}
to justify the claim in part~(b) of Remark~\ref{rem:interpretation_Y}.
\end{remark}

In the next lemma we show how to construct from a $\PdM$-a.e.\ bounded sequence $(\beta^n)_{n \in \N}$ of c\`adl\`ag semimartingales a sequence of admissible strategies $(X^n)_{n\in\N}$ (see~\eqref{eq:def_X_n} below) with the additional properties~\eqref{eq:D_n_equals_minus_beta_n_bracket} and~\eqref{eq:E_t_sup_gamma4bracket8_finite}. We use this result in the proof of Theorem~\ref{thm:sol_val_fct}.

\begin{lemma}\label{lem:Xifbetancadlagsemimart}
Suppose that  $\bm{\left(C_{>0}\right)}$, $\bm{\left(C_{\textbf{BSDE}}\right)}$ and $\bm{\left(C_{[M]}\right)}$ are satisfied. 
Assume that $\rho$ and $\mu$ are $\PdM$-a.e.\ bounded. 
Let $(\beta^n)_{n \in \N}$ be a sequence of c\`adl\`ag semimartingales $\beta^n = (\beta^n_s)_{s\in[0,T]}$ that are $\PdM$-a.e. bounded uniformly in $n$. 
Let $t\in [0,T]$ and $x,d \in \R$. 
Define for each $n\in\N$ the process $X^n=(X^n_s)_{s\in [t,T]}$ by 
 $X^n_{t-}=x$, 
\begin{equation}\label{eq:def_X_n}
X^n_s = \left(x-\frac{d}{\gamma_t}\right)\cE(Q^n)_{t,s}\,(1-\beta_s^n),\quad s\in [t,T),
\end{equation}
 and $X^n_T=0$, where 
\begin{equation*}
Q_s^n = - \int_0^s \beta_r^n \sigma_r dM_r - \int_0^s \beta_r^n (\mu_r + \rho_r - \sigma_r^2) d[M]_r, \quad s\in [0,T].
\end{equation*} 
Then:
\begin{enumerate}
\item $X^n \in \cA_t(x,d)$ for all $n\in\N$;
\item For all $n\in\N$ the associated deviation process $D^n$
(i.e., the one satisfying~\eqref{eq:deviation_dyn} with $X$ replaced by $X^n$) a.s.\ has the representations
\begin{equation}\label{eq:D_n_equals_minus_beta_n_bracket}
D^n_s = -\beta^n_s (\gamma_s X^n_s - D^n_s), 
\quad s\in[t,T),
\end{equation}
and
\begin{equation}\label{eq:07052020a1}
D^n_s = \left(x-\frac{d}{\gamma_t}\right)\cE(Q^n)_{t,s}\,(-\gamma_s\beta_s^n),\quad s\in [t,T),
\end{equation}
and, for the terminal value $D^n_T$, we have
$D^n_T=\left(x-\frac{d}{\gamma_t}\right)\cE(Q^n)_{t,T}\,(-\gamma_T)$.
\item It holds
\begin{equation}\label{eq:E_t_sup_gamma4bracket8_finite}
\sup_{n\in\N} E_t\left[ \sup_{s\in[t,T]} \left( \gamma_s^4 (X_s^n-\alpha_s D_s^n)^8 \right) \right] < \infty \text{\, a.s.}
\end{equation}
\end{enumerate}
\end{lemma}

\begin{proof}
Denote the constants that bound $|\rho|$ and $|\mu|$ $\PdM$-a.e.\ by $c_{\rho}$ and $c_{\mu}$ respectively. Note that, due to $\bm{\left(C_{>0}\right)}$, $\sigma$ is  $\PdM$-a.e.\ bounded by $c_{\sigma}=\sqrt{2c_{\rho}+c_{\mu}}$. 
Let $b \in (0,\infty)$ such that for all $n \in \N$ it holds $|\beta^n| \leq b$ $\PdM$-a.e. Now, fix $n\in\N$. 
Since $\beta^n$ is a c\`adl\`ag semimartingale, it holds that $X^n$ defined by \eqref{eq:def_X_n} is also a c\`adl\`ag semimartingale.  
Note that moreover $X^n_{t-}=x$ and $X^n_T=0$. We denote by $D^n=(D^n_s)_{s\in [t,T]}$ the associated deviation process.
Let $\wh{A}^n=(\wh{A}^n_s)_{s\in [t,T]}$ 
be the process defined by
\begin{equation*}
\begin{split}
\wh{A}^n_s & = \left(x-\frac{d}{\gamma_t}\right) \mathcal{E}(Q^n)_{t,s}
, \quad s\in [t,T].
\end{split}
\end{equation*}
Observe that 
for all $s\in [t,T)$ it holds $X^n_s=\wh{A}^n_s(1-\beta_s^n)$. Together with \eqref{eq:dyn_alpha} it follows that 
\begin{equation*}
d\wh{A}^n_s=\beta_s^n \wh{A}^n_s \left(\frac{d\alpha_s}{\alpha_s}-\rho_sd[M]_s\right)=(X^n_s-\wh{A}^n_s)\left(-\frac{d\alpha_s}{\alpha_s}+\rho_sd[M]_s\right), \quad s\in [t,T].
\end{equation*}
Let $A^n=(A^n_s)_{s\in [t,T]}$ be the process defined by $A^n_s=X^n_s-\frac{D^n_s}{\gamma_s}$, $s\in [t,T]$. 
Then it holds by Lemma~\ref{lem:unifieddAcalc} that 
$\wh{A}^n$ and $A^n$ satisfy the same dynamics and start in the same point $\wh{A}^n_t=x-\frac{d}{\gamma_t}=A^n_t$ at time $t$. Consequently, they are indistinguishable, i.e.,
almost surely, for all $s\in[t,T]$, it holds $A^n_s=\wh{A}^n_s$.
This implies that
\begin{equation}\label{eq:07052020a2}
D^n_s=\gamma_s (X_s^n-A_s^n)=\gamma_s(X^n_s-\wh{A}^n_s)=-\beta_s^n\gamma_s\wh{A}^n_s,\quad s\in[t,T),
\end{equation}
and, proceeding further,
$$
D^n_s=-\beta_s^n\gamma_s A^n_s=-\beta_s^n(\gamma_s X^n_s-D^n_s),\quad s\in[t,T).
$$
We thus establish~\eqref{eq:D_n_equals_minus_beta_n_bracket},
while \eqref{eq:07052020a1} follows from~\eqref{eq:07052020a2}.
For the terminal value $D^n_T$, we have
\begin{align*}
D^n_T
&=D^n_{T-}+\gamma_T\Delta X^n_T
=D^n_{T-}-\gamma_TX^n_{T-}\\
&=\left(x-\frac d{\gamma_t}\right)\cE(Q^n)_{t,T}\left[-\gamma_T\beta^n_{T-}-\gamma_T(1-\beta^n_{T-})\right]
=\left(x-\frac d{\gamma_t}\right)\cE(Q^n)_{t,T}\,(-\gamma_T).
\end{align*}
Furthermore, it follows from $A^n_s=\wh{A}^n_s$, $s\in [t,T]$, that 
\begin{equation*}
	\gamma_s^4 (X^n_s - \alpha_s D_s^n )^8 = \gamma_s^4 (x-\alpha_t d)^8 \left( \mathcal{E}(Q^n)_{t,s} \right)^8, \quad s \in [t,T]. 
\end{equation*}
Note that 
\begin{equation*}
\gamma_s = \gamma_t \exp\left( \int_t^s \mu_r d[M]_r + \int_t^s \sigma_r dM_r - \frac12 \int_t^s \sigma_r^2 d[M]_r \right), \quad s \in [t,T] .
\end{equation*}
Therefore, we have  
\begin{equation*}
\begin{split}
 E_t\Bigg[ \sup_{s\in [t,T]} ( \gamma_s^4  ( X_s^n -  \alpha_s & D_s^n )^8 ) \Bigg] 
 = \gamma_t^4 (x-\alpha_t d)^8  
E_t\Bigg[ \sup_{s\in [t,T]} \exp\Bigg(  \int_t^s (4\sigma_r - 8\beta_r^n \sigma_r) dM_r  \\
& + \int_t^s \left( -8\beta_r^n (\mu_r + \rho_r - \sigma_r^2 ) - 4 (\beta_r^n)^2 \sigma_r^2 + 4\mu_r - 2\sigma_r^2 \right) d[M]_r  \Bigg) \Bigg] .
\end{split}
\end{equation*} 
Define $\eta^n=4\sigma-8\beta^n\sigma$ and $\nu^n=-8\beta^n (\mu + \rho - \sigma^2 ) - 4 (\beta^n)^2 \sigma^2 + 4\mu - 2\sigma^2$. Since $\bm{\left(C_{[M]}\right)}$ holds and we have $|\eta^n| \leq 4c_{\sigma} + 8 b c_{\sigma}$ and $|\nu^n| \leq 8 b (c_{\mu}+c_{\rho} + c_{\sigma}^2) + 4 b^2 c_{\sigma}^2 + 4 c_{\mu} + 2 c_{\sigma}^2$, we obtain \eqref{eq:E_t_sup_gamma4bracket8_finite} from Lemma \ref{lem:expectation_sup_N_finite}. 
Observe furthermore that by Jensen's inequality it follows that (A1) holds true. 
In order to show (A2), note that 
$\bm{\left(C_{[M]}\right)}$ and $|\sigma|\leq c_{\sigma}$ $\PdM$-a.e.  
imply that $E_t\left[ \int_t^T \sigma_s^2 d[M]_s \right] < \infty$ a.s. This, (A1) and the Cauchy-Schwarz inequality prove that (A2) is satisfied. 
It follows from \eqref{eq:D_n_equals_minus_beta_n_bracket} that $\alpha_s^2 (D_s^n )^4 =  (\beta_s^n)^4 \gamma_s^2 \left( X_s^n - \alpha_s D_s^n \right)^4$, $s\in[t,T)$. Since $\beta^n$ is $\PdM$-a.e. bounded, the fact that (A2) is satisfied hence already implies that (A3) holds true as well. We have thus shown that $X^n \in \cA_t(x,d)$.
\end{proof}

In the following lemma we obtain the existence of an approximating c\`adl\`ag semimartingale sequence $(\beta^n)_{n \in \N}$ for any progressively measurable, $\PdM$-a.e. bounded process $\beta$. 
This enables us to exploit Lemma~\ref{lem:Xifbetancadlagsemimart} for the proof of the representation of the value function in Theorem~\ref{thm:sol_val_fct}.

\begin{lemma}\label{lem:approx_beta_by_cadlag_semimart_beta_n} 
Assume that $E\left[ [M]_T \right] < \infty$ and suppose that $\beta=(\beta_s)_{s\in[0,T]}$ is a progressively measurable process that is bounded $\PdM$-a.e. 

Then there exists a sequence $(\beta^n)_{n \in \N}$ of c\`adl\`ag semimartingales $\beta^n = (\beta^n_s)_{s\in[0,T]}$ that are $\PdM$-a.e. bounded uniformly in $n$ and such that for all $p\in[1,\infty)$ it holds 
$E\left[ \int_0^T | \beta_s - \beta^{n}_s |^p d[M]_s \right] \to 0$ as $n \to \infty$.
\end{lemma}

\begin{proof}
It follows from Lemma 2.7 in Section 3.2 of \cite{karatzasshreve} 
that there exists a sequence $(\wh{\beta}^n)_{n \in \N}$ of
(c\`agl\`ad)
simple processes $\wh{\beta}^n = (\wh{\beta}^n_s)_{s\in[0,T]}$ such that 
$E\left[ \int_0^T | \beta_s - \wh{\beta}^n_s |^2 d[M]_s \right] \to 0$ as $n \to \infty$. 
Define $\mathring{\beta}^n_s(\omega) = \lim_{r\downarrow s} \wh{\beta}^n_r(\omega)$, $s\in[0,T], \omega \in \Omega, n \in \N$.  
Let $b \in (0,\infty)$ be such that $|\beta|\leq b$ $\PdM$-a.e. and define, for each $n \in \mathbb{N}$, $\beta^n$ by  
$\beta^n_s (\omega) = \left( \mathring{\beta}^n_s(\omega) \wedge b \right) \vee (-b)$, $s\in[0,T]$, $\omega \in \Omega$. 
Since $|\beta^n_s(\omega)| \leq b$ for all $s\in[0,T], \omega \in \Omega, n \in \N$, and $\mathring{\beta}^n$ is c\`adl\`ag for all $n\in\N$, it follows that $(\beta^n)_{n \in \N}$ is a sequence of c\`adl\`ag semimartingales that are $\PdM$-a.e. bounded uniformly in $n$. 
Furthermore, since 
$| \beta - \beta^n| \leq | \beta - \mathring{\beta}^n |$ and 
$\mathring{\beta}^n = \wh{\beta}^n $ $\PdM$-a.e., 
we have that 
\begin{equation*}
E\left[ \int_0^T |\beta_s - \beta^n_s |^2 d[M]_s \right] \leq  E\left[ \int_0^T |\beta_s - \mathring{\beta}^n_s |^2 d[M]_s \right] = E\left[ \int_0^T |\beta_s - \wh{\beta}^n_s |^2 d[M]_s \right] \to 0
\end{equation*}
as $n \to \infty$. For $p\in[1,2)$, the convergence $E\left[ \int_0^T | \beta_s - \beta^{n}_s |^p d[M]_s \right] \xrightarrow[n\to\infty]{} 0$ follows from  Jensen's inequality, and for $p\in (2,\infty)$, the convergence holds due to $|\beta - \beta^n | \leq 2b$ $\PdM$-a.e.
\end{proof}

\begin{proof}[Proof of Theorem~\ref{thm:sol_val_fct}] 
(i) We first prove the representation of the value function. 

Let $t\in[0,T]$ and $x,d\in\R$. 
Since $\wt{\beta}$ is $\PdM$-a.e. bounded and we assume $\bm{\left(C_{[M]}\right)}$, it follows from Lemma  \ref{lem:approx_beta_by_cadlag_semimart_beta_n} that there exists a sequence $(\beta^n)_{n \in \N}$ of c\`adl\`ag semimartingales $\beta^n = (\beta^n_s)_{s\in[0,T]}$ that are $\PdM$-a.e. bounded uniformly in $n$ 
and such that for all $p\in[1,\infty)$ it holds 
\begin{equation}\label{eq:beta_n_approx_beta_tilde}
E_t\left[ \int_t^T | \tilde{\beta}_s - \beta^{n}_s |^p d[M]_s \right] \to 0 \text{ in } L^1(P) \text{ as } n \to \infty.
\end{equation}
In particular, by passing to a suitable subsequence, we can obtain the almost sure convergence in~\eqref{eq:beta_n_approx_beta_tilde}.
We further obtain from Lemma \ref{lem:Xifbetancadlagsemimart} 
that for each $n \in \N$ there exists $X^n \in \cA_t(x,d)$ such that
 $D^n_s = - \beta^n_s (\gamma_s X^n_s - D^n_s)$, $s\in[t,T)$, where $D^n$ denotes the deviation process associated to $X^n$, and that 
\begin{equation}\label{eq:uniformboundforexpectationofsuplikeA1squared}
\sup_{n\in\N} E_t\left[ \sup_{s\in[t,T]} \left( \gamma_s^4 (X_s^n-\alpha_s D_s^n)^8 \right) \right] < \infty \text{\, a.s.}
\end{equation} 
It then holds $\wt{\beta}_s (\gamma_s X^n_s - D^n_s) + D^n_s = (\wt{\beta}_s - \beta^n_s) (\gamma_s X^n_s - D^n_s)$, $s \in [t,T)$. Together with Theorem~\ref{thm:quadratic_exp_J} and $X^n\in\cA_t(x,d)$ this implies that for all $n\in\mathbb{N}$  
\begin{equation}\label{eq:upper_bound_for_value_fct_by_beta_n}
\begin{split}
 V_t(x,d)  \leq J_t(x,d,X^n) 
& = \frac{Y_{t}}{\gamma_t}\left(d-\gamma_t x\right)^2-\frac{d^2}{2\gamma_t} 
 + E_t\Bigg[\int_t^T  \frac{1}{\gamma_s} (\wt{\beta}_s - \beta_s^n )^2 (\gamma_s X^n_s - D^n_s )^2 \\
& \qquad \cdot 
\left( \sigma_s^2 Y_s + \frac12 (2\rho_s + \mu_s - \sigma_s^2 ) \right)
d[M]_s \Bigg] \text{ a.s.}
\end{split}
\end{equation}
By the Cauchy-Schwarz inequality we have that for all $n\in\mathbb{N}$ 
\begin{equation}\label{eq:R_n_estim}
\begin{split} 
& E_t\left[\int_t^T \frac{1}{\gamma_s} (\wt{\beta}_s - \beta_s^n )^2 (\gamma_s X^n_s - D^n_s )^2 d[M]_s \right] 
= E_t\left[\int_t^T  \gamma_s (\wt{\beta}_s - \beta_s^n )^2 ( X^n_s - \alpha_s D^n_s )^2
d[M]_s \right] \\
& \leq \left( E_t\left[ \int_t^T \gamma_s^2 (X_s^n-\alpha_s D_s^n)^4 d[M]_s \right] \right)^{1/2} \left( E_t\left[ \int_t^T (\wt{\beta}_s - \beta_s^n )^4 d[M]_s \right] \right)^{1/2} .
\end{split}
\end{equation}
Moreover, we have that for all $n \in \N$ 
\begin{equation}\label{eq:expectationbddunifinnas}
\begin{split}
E_t\Bigg[ \int_t^T \gamma_s^2 (X_s^n & - \alpha_s D_s^n)^4 d[M]_s \Bigg] 
\leq E_t\left[ \sup_{s\in[t,T]} \left( \gamma_s^2 (X_s^n-\alpha_s D_s^n)^4 \right) \left( [M]_T - [M]_t \right) \right] \\
& \leq \left(  E_t\left[  \sup_{s\in[t,T]} \left( \gamma_s^4 (X_s^n-\alpha_s D_s^n)^8 \right) \right] \right)^{1/2} \left( E_t\left[ \left( [M]_T - [M]_t \right)^2 \right]  \right)^{1/2} .
\end{split}
\end{equation}  
Since $\rho,\mu,\sigma$ and $Y$ are bounded, it follows from $\bm{\left(C_{[M]}\right)}$, \eqref{eq:uniformboundforexpectationofsuplikeA1squared}, \eqref{eq:expectationbddunifinnas},  \eqref{eq:beta_n_approx_beta_tilde} and \eqref{eq:R_n_estim}
that, along a suitable subsequence, the right-hand side
of~\eqref{eq:upper_bound_for_value_fct_by_beta_n} tends to $\frac{Y_{t}}{\gamma_t}\left(d-\gamma_t x\right)^2-\frac{d^2}{2\gamma_t}$ a.s., as $n\to \infty$.
We obtain the inequality
$V_t(x,d) \leq \frac{Y_{t}}{\gamma_t}\left(d-\gamma_t x\right)^2-\frac{d^2}{2\gamma_t}$ a.s.
The reverse inequality is provided 
in Theorem~\ref{thm:quadratic_exp_J}.

\medskip
(ii) Let $x\neq \frac{d}{\gamma_0}$. In this step, we prove that if there exists an optimal strategy, then there is a c\`adl\`ag semimartingale $\beta=(\beta_s)_{s\in [0,T]}$ such that $\wt{\beta} = \beta$ $\PdM$-a.e.  
For the other implication and for the uniqueness statement
consider $t=0$ in step~(iii) below.

Assume that there exists an optimal strategy $X^*=(X^*_s)_{s\in [0,T]}\in \mathcal A_0(x,d)$. 
It then follows from $V_0(x,d)= \frac{Y_{0}}{\gamma_0}\left(d-\gamma_0 x\right)^2-\frac{d^2}{2\gamma_0}$ and Theorem \ref{thm:quadratic_exp_J} that 
\begin{equation}\label{eq:squared_bracket_in_expectation_in_J_equals_zero}
\wt{\beta} \left( \gamma X^* - D^* \right) + D^* = 0 \quad \PdM\text{-a.e.},
\end{equation} 
where $D^*=(D^*_s)_{s\in [0,T]}$ denotes the 
deviation process associated to $X^*$ 
(see~\eqref{eq:deviation_dyn}).
This yields for the process $A^*=(A^*_s)_{s\in[0,T]}$ defined by 
$A^*_s = X^*_s - \alpha_s D^*_s$, $s \in [0,T],$ that 
by Lemma~\ref{lem:unifieddAcalc} 
\begin{equation*}
	dA^*_s = \wt{\beta}_s A^*_s\left(\frac{d\alpha_s}{\alpha_s} - \rho_s d[M]_s\right), \quad s \in [0,T], 
\end{equation*} 
and $A^*_0 = x - \frac{d}{\gamma_0}$. It follows that 
\begin{equation*}
A^*_s 
=\left(x-\frac{d}{\gamma_0}\right) 
\mathcal E\left( \int_0^\cdot \frac{\wt{\beta}_r}{\alpha_r}d\alpha_r-\int_0^\cdot \rho_r \wt{\beta}_r d[M]_r \right)_s 
= \left(x-\frac{d}{\gamma_0}\right) \mathcal E( \wt{Q} )_s, 
\quad s\in [0,T],
\end{equation*}
where $\wt{Q}_s =  \int_0^s \frac{\wt{\beta}_r}{\alpha_r}d\alpha_r-\int_0^s \rho_r \wt{\beta}_r d[M]_r$, $s\in [0,T]$. Since $\wt{Q} = (\wt{Q}_s)_{s\in[0,T]}$ is a continuous semimartingale, its stochastic exponential $\mathcal{E}(\wt{Q})$ is strictly positive. 
Together with the assumption $x\neq \frac{d}{\gamma_0}$ 
it follows that $A^*$ is nonvanishing. 
Consequently, $\beta = -\frac{D^*}{\gamma A^*}$ is a c\`adl\`ag semimartingale,
whereas \eqref{eq:squared_bracket_in_expectation_in_J_equals_zero} proves that $\wt{\beta} = \beta$ $\PdM$-a.e.

\medskip
(iii) Suppose that there exists a c\`adl\`ag semimartingale $\beta=(\beta_s)_{s\in [0,T]}$ such that $\wt{\beta} = \beta$ $\PdM$-a.e., and let $t\in[0,T]$. 
It then follows from Lemma~\ref{lem:Xifbetancadlagsemimart} that \eqref{eq:opt_strat_representation} defines 
a strategy $X^* \in \cA_t(x,d)$
such that, for the associated deviation process $D^*$,
we have representation~\eqref{eq:opt_dev_representation} and, moreover,
$D^*=-\beta (\gamma X^*-D^*) = -\wt{\beta} (\gamma X^*-D^*)$ 
$\PdM|_{[t,T]}$-a.e.
Then Theorem~\ref{thm:quadratic_exp_J} implies that $J_t(x,d,X^*)=\frac{Y_{t-}}{\gamma_t}(d-\gamma_t x)^2-\frac{d^2}{2\gamma_t}$, and since $V_t(x,d)=\frac{Y_{t-}}{\gamma_t}(d-\gamma_t x)^2-\frac{d^2}{2\gamma_t}$, the strategy $X^*$ is optimal. 
The uniqueness up to $\PdM|_{[t,T]}$-null sets follows from Lemma~\ref{lem:uniqueness_of_opt_strat_up_to_Dm_null_sets}.
\end{proof}

\begin{remark}\label{rem:26032021a2}
In order to justify the claim in part~(b) of Remark~\ref{rem:interpretation_Y} it now remains to verify that, for any $x,d\in\bbR$ and $t\in[0,T]$, the strategy $X^*$ in~\eqref{eq:opt_strat_representation} belongs to $\cA_t(x,d)$, i.e., satisfies (A1)--(A3), under conditions \ref{it:26032021a4}--\ref{it:26032021a2} in part~(b) of Remark~\ref{rem:interpretation_Y} only.
(A1) and~(A2) are verified by essentially the same argument (use Remark~\ref{rem:26032021a1} and condition~\ref{it:26032021a2} in part~(b) of Remark~\ref{rem:interpretation_Y}).
For~(A3), we notice that $D^*=-\beta(\gamma X^*-D^*)$ $\PdM|_{[t,T]}$-a.e.\ implies $(D^*)^4\alpha^2=\gamma^2(X^*-\alpha D^*)^4\beta^4$ $\PdM|_{[t,T]}$-a.e., hence (A3) follows, via the Cauchy-Schwarz inequality, from (A1) and condition~\ref{it:26032021a1} in part~(b) of Remark~\ref{rem:interpretation_Y}.
\end{remark}

\begin{proof}[Proof of Corollary~\ref{cor:07052020a1}]
The assumptions allow to apply Theorem~\ref{thm:sol_val_fct}.
First notice that, for any strategy $X$, the associated process
$X-\frac D\gamma$
is continuous (although both $X$ and $D$ can have jumps),
which follows from~\eqref{eq:deviation_dyn}.
Together with
\eqref{eq:opt_strat_representation}
and~\eqref{eq:opt_dev_representation}
this yields that, for any $u\in(t,T)$, we have
\begin{equation}\label{eq:07052020a3}
X_{u-}-\frac{D_{u-}}{\gamma_u}
=X_u-\frac{D_u}{\gamma_u}
=\left(x-\frac d{\gamma_t}\right)\cE(Q)_{t,u}.
\end{equation}
All statements of the corollary now follow from~\eqref{eq:07052020a3}
and the trivial fact that, for all $s\in[u,T]$,
we have $\cE(Q)_{t,u}\,\cE(Q)_{u,s}=\cE(Q)_{t,s}$.
\end{proof}

\begin{proof}[Proof of Proposition~\ref{propo:withoutresiliencecloseimmediately}]
	In the case $\rho \equiv 0$, the driver~\eqref{eq:driver_of_bsde} of BSDE~\eqref{eq:bsde} equals $0$ for $(Y,Z)=\left( \frac12 , 0\right)$. 
	Hence, $\left( Y,Z,M^\perp \right) = \left( \frac12, 0, 0 \right)$ solves BSDE~\eqref{eq:bsde} and $\bm{\left(C_{\textbf{BSDE}}\right)}$ is clearly satisfied. 
	We then obtain that $\wt{\beta}_s = 1$ for all $s\in[0,T]$. 
	By Theorem~\ref{thm:quadratic_exp_J} it holds for all $x,d \in \R$, $t \in [0,T]$ and $X \in \mathcal{A}_t(x,d)$ that 
	\begin{equation}\label{eq:18022020a2}
	J_t(x,d,X) = \frac{1}{2\gamma_t} (d-\gamma_t x)^2 - \frac{d^2}{2\gamma_t} + E_t\left[ \int_t^T \frac12 \gamma_s \mu_s X_s^2 d[M]_s \right] .
	\end{equation}
	Notice that, due to $\bm{\left(C_{>0}\right)}$ and $\rho \equiv 0$, the process $\mu$ is positive. 
	The optimality of closing the position immediately and the formula for the value function now follow from \eqref{eq:18022020a2}. 
	The uniqueness up to $\PdM|_{[t,T]}$-null sets follows from Lemma~\ref{lem:uniqueness_of_opt_strat_up_to_Dm_null_sets}.
\end{proof}

\begin{proof}[Proof of Proposition~\ref{propo:bsdeuniquenessfrommaintheo}]
	By Theorem~\ref{thm:sol_val_fct} we have 
	$\gamma_t Y^{(1)}_t = V_t(1,0) = \gamma_t Y^{(2)}_t$, $t\in[0,T]$. 
	Therefore, $Y^{(1)}$ and $Y^{(2)}$ are indistinguishable.
	Compare the following canonical decompositions (see Section~I.4c in \cite{jacodshiryaev}) of the special semimartingale $Y=Y^{(1)}=Y^{(2)}$, where $f$ denotes driver~\eqref{eq:driver_of_bsde} of BSDE~\eqref{eq:bsde}:
	\begin{equation*}
	\begin{split}
	Y_t & = Y_0 - \int_0^t f\left(s,Y_s,Z^{(1)}_s\right) d[M]_s + \int_0^t Z^{(1)}_s dM_s + M^{\perp,(1)}_t \\
	& = Y_0 - \int_0^t f\left(s,Y_s,Z^{(2)}_s\right) d[M]_s + \int_0^t Z^{(2)}_s dM_s + M^{\perp,(2)}_t, \quad t \in [0,T].
	\end{split}
	\end{equation*}
	For the local martingale parts we have 
	\begin{equation}\label{eq:18022020a3}
	\int_0^{\cdot} Z^{(1)}_s dM_s + M^{\perp, (1)}_{\cdot} 
	= \int_0^{\cdot} Z^{(2)}_s dM_s + M^{\perp, (2)}_{\cdot}.  
	\end{equation} 
	This implies that 
	\begin{equation*}
	\begin{split}
	\left[ M^{\perp, (1)} - M^{\perp, (2)} \right]_t 
	& = \left[ M^{\perp, (1)} - M^{\perp, (2)}, \int_0^{\cdot} \left( Z^{(2)}_s - Z^{(1)}_s \right) dM_s \right]_t \\
	& =\int_0^t \left( Z^{(2)}_s - Z^{(1)}_s \right) d\left[  M^{\perp, (1)} - M^{\perp, (2)}, M \right]_s 
	= 0, \quad t\in [0,T].
	\end{split}
	\end{equation*}
	Thus $M^{\perp, (1)} - M^{\perp, (2)}$ is a local martingale starting in $0$ with $\left[ M^{\perp, (1)} - M^{\perp, (2)} \right] = 0$. It follows from the Burkholder-Davis-Gundy inequality that $M^{\perp, (1)}$ and $M^{\perp, (2)}$ are indistinguishable. 
	Then, \eqref{eq:18022020a3} implies further that 
	$\int_0^{\cdot} \left( Z^{(2)}_s - Z^{(1)}_s \right) dM_s = 0 $ 
	and hence 
	\begin{equation*}
	\int_0^{\cdot} \left( Z^{(2)}_s - Z^{(1)}_s \right)^2 d[M]_s 
	= \left[ \int_0^{\cdot} \left( Z^{(2)}_s - Z^{(1)}_s \right) dM_s \right] 
	= 0 .
	\end{equation*} 
	It follows that $Z^{(1)} = Z^{(2)}$ $\PdM$-a.e.
\end{proof}

\appendix

\section{Once again to the cost functional and the dynamics of the deviation process}\label{sec:formalderivationdeviationcostfct}

Here we motivate dynamics~\eqref{eq:deviation_dyn}
of the deviation process
and cost functional~\eqref{eq:cost_fct}
via a limiting procedure from the discrete-time setting.

Without loss of generality we consider the starting time $t=0$.
We fix an initial position $x\in\bbR$ and an initial deviation $d\in\bbR$
and consider a continuous-time execution strategy
$X \in \mathcal{A}_0(x,d)$.
For any (large) $N\in\bbN$, we set $h=\frac TN$
and consider discrete-time trading at points
of the grid $\{kh\colon k=0,\ldots,N\}$.
More precisely, the continuous-time strategy $X$
is approximated by the discrete-time strategy
that consists of trades $\xi_{kh}$, $k\in\{0,\ldots,N\}$,
at the grid points, where
$$
\xi_0=X_0-x,\quad
\xi_{kh}=X_{kh}-X_{(k-1)h},\quad
k\in\{1,\ldots,N\}.
$$
Notice that $\xi_{kh}$ is $\cF_{kh}$-measurable, $k=0,\ldots,N$.
Further, for $k\in\{1,\ldots,N\}$, we define
$\beta_{kh} = \exp\left( -\int_{(k-1)h}^{kh} \rho_s d[M]_s \right)$
and introduce the notations
$\eta_r = \exp\left( - \int_0^r \rho_s d[M]_s \right)$ and
$\nu_r=\gamma_{r}\exp\left( \int_0^r \rho_s d[M]_s \right)$,
$r \in [0,T]$.

In the discrete-time setting of \cite{ackermann2020optimal},
the deviation process (now denoted by $\wt{D}^{(h)}$) is defined by 
\begin{equation*}
\wt{D}^{(h)}_{0-}=d, \quad 
\wt{D}^{(h)}_{(kh)-} = \left( \wt{D}^{(h)}_{((k-1)h)-} + \gamma_{(k-1)h} \xi_{(k-1)h} \right)\beta_{kh}, \quad  k\in\{1,\ldots,N\}.
\end{equation*}
The minus in the subscript of $\wt{D}^{(h)}_{(kh)-}$ is purely notational
(this is a discrete-time process),
the meaning of $\wt{D}^{(h)}_{(kh)-}$ is that
this is the deviation at time $kh$
directly \emph{prior} to the trade $\xi_{kh}$ at time $kh$,
and we preserve the minus sign in order to make the notation consistent with \cite{ackermann2020optimal}.
A straightforward calculation shows that
$$
\wt{D}^{(h)}_{(kh)-} =  
d \prod_{l=1}^{k} \beta_{lh} + \sum_{i=1}^{k} \gamma_{(i-1)h} \xi_{(i-1)h} \prod_{l=i}^{k} \beta_{lh}, \quad k \in \{1,\ldots,N\} .
$$
Substituting the definition of $\beta_{kh}$, we obtain that, for all $k\in\{1,\ldots,N\}$,
\begin{equation}\label{eq:dt_deviation_as_prod_exp_L}
\begin{split}
\wt{D}_{(kh)-}^{(h)} 
& = \exp\left( - \int_{0}^{kh} \rho_s d[M]_s \right) d + \sum_{i=1}^k \gamma_{(i-1)h} \xi_{(i-1)h} \exp\left( - \int_{(i-1)h}^{kh} \rho_s d[M]_s \right) \\[1mm]
& = \eta_{kh} \left( d + \sum_{i=1}^k \nu_{(i-1)h} \xi_{(i-1)h} \right)
=\eta_{kh}L^{(h)}_{(k-1)h},
\end{split}
\end{equation}
where, for $k\in\{0,\ldots,N\}$, we set
\begin{equation*}
\begin{split}
L^{(h)}_{kh}
&=d + \sum_{j=0}^{k} \nu_{jh} \xi_{jh}
= d + \gamma_0 ( X_0 - x ) + \sum_{j=1}^{k} \nu_{jh} \left( X_{jh} - X_{(j-1)h} \right)\\
&=d + \gamma_0 (X_0-x) + \sum_{j=1}^{k} \nu_{(j-1)h} \left( X_{jh} - X_{(j-1)h} \right) 
+ \sum_{j=1}^{k} \left( \nu_{jh} - \nu_{(j-1)h} \right) \left( X_{jh} - X_{(j-1)h} \right).
\end{split}
\end{equation*}
The last expression shows that the continuous-time limit of the processes
$(L^{(h)}_{kh})_{k\in\{0,\ldots,N\}}$,
as $N\to\infty$ (and $h=\frac TN\to0$),
is the process $(L_s)_{s\in[0,T]}$ given by
\begin{equation*}
L_s = d+ \int_{[0,s]} \nu_r\,dX_r + \int_{[0,s]} d[\nu,X]_r, \quad s\in[0,T]
\end{equation*} 
(apply Proposition~I.4.44 and Theorem~I.4.47 in \cite{jacodshiryaev}).
Combining this with~\eqref{eq:dt_deviation_as_prod_exp_L} and the definition of $\nu_r$, $r \in[0,T]$, recovers
that the continuous-time limit of the processes
$(\wt{D}^{(h)}_{(kh)-})_{k\in\{0,\ldots,N\}}$
is the process $(D_s)_{s\in[0,T]}$ given by
$$
D_s=\eta_s L_s,\quad s\in[0,T]
$$
(and $D_{0-}=d$),
which is nothing else but~\eqref{eq:deviation_def}
or, equivalently,~\eqref{eq:deviation_dyn}.

\smallskip
We now turn to the cost functional.
In the discrete-time setting the cost is 
$\sum_{j=0}^N \left( \wt{D}^{(h)}_{(jh)-} + \frac{\gamma_{jh}}{2} \xi_{jh} \right) \xi_{jh}$.
Set $X_{-h}=X_{0-}\;(=x)$. Then it holds
\begin{equation}\label{eq:dt_cost_fct_split_sum}
\begin{split}
\sum_{j=0}^N \left( \wt{D}^{(h)}_{(jh)-} + \frac{\gamma_{jh}}{2} \xi_{jh} \right) \xi_{jh} 
& = \sum_{j=0}^N \wt{D}^{(h)}_{(jh)-} \left( X_{jh} - X_{(j-1)h} \right) 
 + \sum_{j=0}^N \frac{\gamma_{(j-1)h}}{2} \left( X_{jh} - X_{(j-1)h} \right)^2 \\
& \quad + \sum_{j=0}^N \frac12 \left( \gamma_{jh} - \gamma_{(j-1)h} \right) \left( X_{jh} - X_{(j-1)h} \right)^2 .
\end{split}
\end{equation} 
For the first term on the right-hand side of~\eqref{eq:dt_cost_fct_split_sum}, we have
\begin{equation*} 
\begin{split}
 & \sum_{j=0}^N \wt{D}^{(h)}_{(jh)-} \left( X_{jh} - X_{(j-1)h} \right) 
= \sum_{j=0}^N \eta_{jh} 
L^{(h)}_{(j-1)h} \left( X_{jh} - X_{(j-1)h} \right) \\
& \quad =  \sum_{j=0}^N \eta_{(j-1)h} L^{(h)}_{(j-1)h} \left( X_{jh} - X_{(j-1)h} \right) 
+ \sum_{j=0}^N L^{(h)}_{(j-1)h} \left( \eta_{jh} - \eta_{(j-1)h} \right) 
 \left( X_{jh} - X_{(j-1)h} \right),
\end{split}
\end{equation*}
which has the continuous-time limit
\begin{equation*}
\int_{[0,T]}\eta_sL_{s-}\,dX_s+\int_{[0,T]}L_{s-}\,d[\eta,X]_s
=\int_{[0,T]}D_{s-}\,dX_s,
\end{equation*}
as $\eta$ is a continuous process of finite variation.
Further, the second term on the right-hand side of~\eqref{eq:dt_cost_fct_split_sum} tends to
$\int_{[0,T]}\frac{\gamma_s}2\,d[X]_s$
and the third term to $\frac12[\gamma,[X]]_T=0$
because $\gamma$ is continuous.
As the continuous-time limit of the discrete-time cost we thus obtain
$$
\int_{[0,T]}D_{s-}\,dX_s+\int_{[0,T]}\frac{\gamma_s}2\,d[X]_s,
$$
which motivates our form of the cost functional in continuous time.

\section{Heuristic derivation of the BSDE}\label{sec:formalderivationbsde}

We have seen that BSDE~\eqref{eq:bsde} plays a central role both in our results and in the proofs. But where does it come from?
In this appendix we motivate BSDE~\eqref{eq:bsde} via a heuristic limiting procedure from discrete time.

To this end we consider a discrete-time version of the stochastic control problem~\eqref{eq:def_value_fct}. 
For $h>0$ such that $h=\frac{T}{N}$ for some $N\in \N$, $t\in[0,T]$ and $x, d \in \R$ let $\cA_t^h(x,d)$ be the subset of all $X = (X_s)_{s\in[t,T]} \in \cA_t(x,d)$ with 
$X_s= \sum_{k=0}^N X_{(kh) \vee t} 1_{[kh,(k+1)h)}(s)$ for all $s \in [t,T]$. 
Moreover, let $V_t^h(x,d) = \essinf_{X \in \cA_t^h(x,d)}J_t(x,d,X)$ for all $x,d \in \R, t\in [0,T]$, $h>0$ with $h=\frac{T}{N}$ for some $N\in\N$. Then it follows from \cite{ackermann2020optimal} that for each $h>0$ with $h=\frac{T}{N}$ for some $N\in\N$ there exists a process $Y^h=(Y_t^h)_{t\in\{0,h,\ldots,T\}}$ such that $V_t^h(x,d) = \frac{Y_t^h}{\gamma_t} \left( d-\gamma_t x\right)^2 - \frac{d^2}{2\gamma_t}$, $x,d \in \mathbb{R}$, $t\in\{0,h,\ldots,T\}$. The discrete-time process $Y^h$ is given by the backward recursion $Y_T^h=\frac12$ and, for $t \in \{0,h,\ldots,T-h\}$, 
\begin{equation}\label{eq:def_discrete_time_Y}
Y_t^h = E_t\left[ \frac{\gamma_{t+h}}{\gamma_t} Y_{t+h}^h \right] - \frac{\left(  E_t\left[ Y_{t+h}^h \left( e^{-\int_t^{t+h} \rho_s d[M]_s} - \frac{\gamma_{t+h}}{\gamma_t} \right) \right]  \right)^2}{E_t\left[ Y_{t+h}^h \frac{\gamma_t}{\gamma_{t+h}} \left( e^{-\int_t^{t+h} \rho_s d[M]_s} - \frac{\gamma_{t+h}}{\gamma_t} \right) ^2 + \frac12 \left( 1- \frac{\gamma_t}{\gamma_{t+h}} e^{-2\int_t^{t+h} \rho_s d[M]_s} \right) \right]} .
\end{equation}
We aim at deriving --- at least heuristically --- the dynamics of the continuous-time limit $Y=(Y_t)_{t\in[0,T]}$ of $Y^h$.
To this end, we suppose that $Y$ can be decomposed as follows 
\begin{equation}\label{eq:proposed_decomp_of_Y}
dY_t = a_t d[M]_t + Z_tdM_t + dM^\perp_t, \quad t \in [0,T],
\end{equation} 
where $(a_t)_{t\in[0,T]}$, $(Z_t)_{t\in[0,T]}$ are progressively measurable processes
($(a_t)_{t\in[0,T]}$ is to be determined)
and $M^\perp = (M^\perp_t)_{t\in[0,T]}$ is a local martingale orthogonal to $M$. 
From~\eqref{eq:proposed_decomp_of_Y} we deduce that $(a_t)_{t\in[0,T]}$ should be identified as the limit 
\begin{equation*}
	\begin{split}
		a_t & = \lim_{h \to 0} \frac{E_t[Y_{t+h}]-Y_t}{E_t\left[ [M]_{t+h} \right] - [M]_t}, \quad t\in [0,T].
	\end{split}
\end{equation*}
Assume that replacing $Y^h$ with $Y$ in~\eqref{eq:def_discrete_time_Y} introduces an error only of the magnitude $o\left(E_t\left[ [M]_{t+h} \right] - [M]_t\right)$. Then we can get the expression for $a_t$ by evaluating the limit 
\begin{equation}\label{eq:identifyabylimit}
\begin{split}
a_t & = \lim_{h \to 0} \frac{1}{E_t\left[ [M]_{t+h} \right] - [M]_t} 
\left(\rule{0cm}{1.4cm}\right.
E_t\left[ Y_{t+h} \right]
- E_t\left[ \frac{\gamma_{t+h}}{\gamma_t} Y_{t+h} \right] \\
& \quad + \frac{\left(  E_t\left[ Y_{t+h} \left( e^{-\int_t^{t+h} \rho_s d[M]_s} - \frac{\gamma_{t+h}}{\gamma_t} \right) \right]  \right)^2}{E_t\left[ Y_{t+h} \frac{\gamma_t}{\gamma_{t+h}} \left( e^{-\int_t^{t+h} \rho_s d[M]_s} - \frac{\gamma_{t+h}}{\gamma_t} \right) ^2 + \frac12 \left( 1- \frac{\gamma_t}{\gamma_{t+h}} e^{-2\int_t^{t+h} \rho_s d[M]_s} \right) \right]} 
\left.\rule{0cm}{1.4cm}\right), \, t\in [0,T].
\end{split}
\end{equation} 
For the remainder of this section we fix $t\in[0,T]$ and assume that all stochastic integrals with respect to $dM$ and $dM^\perp$ that appear are true martingales. 
We define the process $\Gamma=(\Gamma_s)_{s\in[t,T]}$ by $\Gamma_s=\frac{\gamma_s}{\gamma_t}$ for $s\in[t,T]$. 
Since 
\begin{equation*}
d(\Gamma_sY_s) 
= \left( Y_{s}\Gamma_s \mu_s + \Gamma_s a_s + \Gamma_s\sigma_sZ_s \right) d[M]_s 
+ \left( Y_{s}\Gamma_s \sigma_s + \Gamma_sZ_s \right) dM_s
+ \Gamma_s dM^\perp_s, \, s\in[t,T],
\end{equation*}
it holds that for all $h\in (0,T-t)$, 
\begin{equation}\label{eq:GammaYErw}
E_t\left[ \Gamma_{t+h} Y_{t+h} \right] = Y_t + E_t\left[ \int_t^{t+h} \left( Y_{s}\Gamma_s \mu_s + \Gamma_s a_s + \Gamma_s\sigma_sZ_s  \right) d[M]_s \right] .
\end{equation}
Together with 
\begin{equation*}
E_t\left[ Y_{t+h}\right] = Y_t + E_t\left[ \int_t^{t+h} a_sd[M]_s \right], \, h\in (0,T-t), 
\end{equation*}
we obtain heuristically that  
\begin{equation}\label{eq:firstlimitfora}
\begin{split}
\frac{E_t\left[Y_{t+h}\right]-E_t\left[ \Gamma_{t+h} Y_{t+h} \right]}{E_t\left[ [M]_{t+h} \right] - [M]_t} 
& = \frac{E_t\left[ \int_t^{t+h} \left( a_s (1-\Gamma_s) - Y_{s}\Gamma_s \mu_s - \Gamma_s\sigma_sZ_s \right)d[M]_s \right]}{E_t\left[ \int_t^{t+h} d[M]_{s} \right] } \\ 
& \xrightarrow[h \to 0]{}  -Y_t\mu_t-\sigma_tZ_t .
\end{split}
\end{equation}
Furthermore, it holds for all $h\in (0,T-t)$ that 
\begin{equation}\label{eq:Yexprho}
\begin{split}
Y_{t+h} e^{- \int_t^{t+h} \rho_s d[M]_s} 
& = Y_t + \int_t^{t+h}\left( a_s - \rho_s Y_{s} \right) e^{- \int_t^{s} \rho_r d[M]_r} d[M]_s 
+ \int_t^{t+h}  Z_s e^{- \int_t^{s} \rho_r d[M]_r} dM_s \\
& \quad	+ \int_{(t,t+h]} e^{- \int_t^{s} \rho_r d[M]_r}  dM^\perp_s .
\end{split}
\end{equation}
From \eqref{eq:GammaYErw} and \eqref{eq:Yexprho} we derive heuristically that 
\begin{equation}\label{eq:secondlimitfora}
\begin{split}
& \frac{E_t\left[ Y_{t+h} \left( e^{-\int_t^{t+h} \rho_s d[M]_s} - \Gamma_{t+h} \right) \right]}{E_t\left[ [M]_{t+h} \right] - [M]_t} \\
& = \frac{  E_t\left[ \int_t^{t+h} \left( (a_s - \rho_s Y_{s} ) e^{- \int_t^{s} \rho_r d[M]_r} -\left( Y_{s}\Gamma_s \mu_s + \Gamma_s a_s + \Gamma_s\sigma_sZ_s  \right) \right) d[M]_s \right]  }{ E_t\left[ \int_t^{t+h} d[M]_s \right] }  \\
& \xrightarrow[h \to 0]{}  - \rho_t Y_t - Y_t \mu_t -\sigma_t Z_t .
\end{split}
\end{equation}
Recall that $\Gamma_s^{-1} = \frac{\alpha_s}{\alpha_t}$, $s\in[t,T]$, with 
\begin{equation*}
d\Gamma_s^{-1} = \Gamma_s^{-1} \left( - \left(\mu_s-\sigma_s^2\right) d[M]_s - \sigma_s dM_s  \right), \quad s \in [t,T].
\end{equation*} 
Therefore, it holds that 
\begin{equation}\label{eq:YGammaInv}
\begin{split}
d\left(Y_s \Gamma_{s}^{-1}\right) 
& =  \left( - Y_{s} \Gamma_s^{-1} \left(\mu_s-\sigma_s^2\right) + \Gamma_s^{-1} a_s - Z_s \Gamma_s^{-1} \sigma_s  \right) d[M]_s \\
& \quad + \left( \Gamma_s^{-1}Z_s - Y_{s}\Gamma_s^{-1} \sigma_s \right) dM_s 
+ \Gamma_s^{-1} dM^\perp_s , \quad s\in [t,T].
\end{split}
\end{equation}
Moreover, we have that for all $h\in (0,T-t)$, 
\begin{equation}\label{eq:expMinGammaSquared}
\begin{split}
\Big( e^{-\int_t^{t+h} \rho_s d[M]_s} & - \Gamma_{t+h} \Big)^2 
= \int_t^{t+h} \Big(  \Gamma_s^2\sigma_s^2 - 2 \left( e^{-\int_t^{s} \rho_r d[M]_r} - \Gamma_{s} \right) \\
& \cdot \left( \rho_se^{-\int_t^{s} \rho_r d[M]_r} + \Gamma_s \mu_s \right) \Big) d[M]_s 
- 2 \int_t^{t+h}   \left( e^{-\int_t^{s} \rho_r d[M]_r} - \Gamma_{s} \right) \Gamma_s \sigma_s dM_s .
\end{split}
\end{equation}
It follows from \eqref{eq:YGammaInv} and \eqref{eq:expMinGammaSquared} that 
\begin{equation*}
\begin{split}
& Y_{t+h}\Gamma_{t+h}^{-1} \left( e^{-\int_t^{t+h} \rho_s d[M]_s} - \Gamma_{t+h} \right)^2 \\
& = \int_t^{t+h} \Bigg(  Y_{s} \Gamma_s^{-1} \left( \Gamma_s^2 \sigma_s^2 - 2 \left( e^{-\int_t^{s} \rho_r d[M]_r} - \Gamma_{s} \right) \left( \rho_s e^{-\int_t^s \rho_r d[M]_r} + \Gamma_s \mu_s \right) \right)  \\
& \quad \quad \quad + \left( e^{-\int_t^{s} \rho_r d[M]_r} - \Gamma_{s} \right)^2 \Gamma_s^{-1} \left( -Y_{s} \left(\mu_s-\sigma_s^2\right) + a_s - Z_s \sigma_s \right)  \\
& \quad \quad \quad - 2 \sigma_s \left( Z_s - Y_{s} \sigma_s \right) \left( e^{-\int_t^{s} \rho_r d[M]_r} - \Gamma_{s} \right)  \Bigg) d[M]_s \\
& \quad + \int_t^{t+h} \left(  \Gamma_s^{-1} \left( e^{-\int_t^{s} \rho_r d[M]_r} - \Gamma_{s} \right)^2 \left( Z_s - Y_{s} \sigma_s \right) - 2 Y_{s} \left( e^{-\int_t^{s} \rho_r d[M]_r} - \Gamma_{s} \right) \sigma_s  \right) dM_s \\
& \quad + \int_{(t,t+h]} \left( e^{-\int_t^{s} \rho_r d[M]_r} - \Gamma_{s} \right)^2 \Gamma_s^{-1} dM_s^\perp , \quad h\in (0,T-t), 
\end{split}
\end{equation*}
and hence 
\begin{equation*}
\begin{split}
E_t\Big[ Y_{t+h} & \Gamma_{t+h}^{-1} \left( e^{-\int_t^{t+h} \rho_s d[M]_s} - \Gamma_{t+h} \right)^2 \Big] 
=
E_t\Bigg[ \int_t^{t+h}  \Bigg(  Y_{s} \Gamma_s^{-1} \Big( \Gamma_s^2 \sigma_s^2 - 2 \left( e^{-\int_t^{s} \rho_r d[M]_r} - \Gamma_{s} \right) \\
& \cdot \left( \rho_s e^{-\int_t^s \rho_r d[M]_r} + \Gamma_s \mu_s \right) \Big)  
+ \left( e^{-\int_t^{s} \rho_r d[M]_r} - \Gamma_{s} \right)^2 \Gamma_s^{-1} \left( -Y_{s} \left(\mu_s-\sigma_s^2\right) + a_s - Z_s \sigma_s \right)  \\
& - 2 \sigma_s \left( Z_s - Y_{s} \sigma_s \right) \left( e^{-\int_t^{s} \rho_r d[M]_r} - \Gamma_{s} \right)  \Bigg) 
d[M]_s \Bigg] , \quad h\in (0,T-t) .
\end{split}
\end{equation*}
Therefore, we obtain heuristically that 
\begin{equation}\label{eq:thirdlimitfora}
\frac{E_t\left[ Y_{t+h}\Gamma_{t+h}^{-1} \left( e^{-\int_t^{t+h} \rho_s d[M]_s} - \Gamma_{t+h} \right)^2 \right] }{E_t\left[ [M]_{t+h}\right] - [M]_t } 
\xrightarrow[h \to 0]{} Y_t \sigma_t^2 .
\end{equation}
From 
\begin{equation*}
\begin{split}
\Gamma_{t+h}^{-1} e^{-2 \int_t^{t+h} \rho_s d[M]_s} 
& = 1 - \int_t^{t+h} \Gamma_s^{-1} e^{-2 \int_t^{s} \rho_r d[M]_r} \left( 2 \rho_s + \mu_s-\sigma_s^2 \right) d[M]_s \\
& \quad - \int_t^{t+h} e^{-2 \int_t^{s} \rho_r d[M]_r} \Gamma_s^{-1} \sigma_s dM_s , \quad h\in (0,T-t),  
\end{split}
\end{equation*}
we derive heuristically that 
\begin{equation}\label{eq:fourthlimitfora}
\begin{split}
\frac{  E_t\left[ \frac{1}{2} \left( 1 - \Gamma_{t+h}^{-1} e^{-2 \int_t^{t+h} \rho_s d[M]_s} \right) \right]  }{E_t\left[ [M]_{t+h}\right] - [M]_t } 
& = \frac{ E_t\left[  \int_t^{t+h} \frac12 \left(  \Gamma_s^{-1} e^{-2 \int_t^{s} \rho_r d[M]_r} \left( 2 \rho_s + \mu_s-\sigma_s^2 \right)  \right) d[M]_s \right]  }{  E_t\left[ \int_t^{t+h} d[M]_s \right] } \\
& \xrightarrow[h \to 0]{} \frac12 \left(2 \rho_t + \mu_t - \sigma_t^2 \right).
\end{split}
\end{equation}

We conclude from \eqref{eq:identifyabylimit}, \eqref{eq:firstlimitfora}, \eqref{eq:secondlimitfora}, \eqref{eq:thirdlimitfora} and \eqref{eq:fourthlimitfora} that  
$$
a_t  = -Y_t\mu_t - \sigma_t Z_t + \frac{\left( - \rho_t Y_t - Y_t\mu_t - \sigma_t Z_t\right)^2}{ Y_t\sigma_t^2 +  \frac12 \left(2 \rho_t + \mu_t - \sigma_t^2 \right)}
=-f(t,Y_t,Z_t)
$$
with $f$ given in~\eqref{eq:driver_of_bsde}.
Finally, the fact that the discrete-time processes $Y^h$, $h \in (0,T-t)$, are $(0,1/2]$-valued, which is proved in \cite{ackermann2020optimal},
explains the requirement in $\bm{\left(C_{\textbf{BSDE}}\right)}$ that $Y$ is $[0,1/2]$-valued.

\section{Comparison argument for Section~\protect\ref{propo:existence_bsde_general_filtration}}\label{sec:comp_bsde}

Here we justify via a certain comparison argument that, in Proposition~\ref{propo:existence_bsde_general_filtration}, we get that $Y$ is $[0,1/2]$-valued.

In the following proposition we are interested in a BSDE with driver $f$ and terminal value $\xi$ of the form 
\begin{equation}\label{eq:bsde_comp}
dY_s = - f(s,Y_s) d[M]_s + Z_s dM_s + dM^\perp_s, \quad s\in[0,T], \quad  Y_T=\xi,
\end{equation}
and denote such a BSDE by BSDE$(f,\xi)$. Recall that, in Proposition~\ref{propo:existence_bsde_general_filtration}, the driver does not depend on $Z$. Therefore, we do not consider dependence on $Z$ in~\eqref{eq:bsde_comp}.

\begin{propo}\label{propo:comparison}
	Assume $\bm{\left(C_{[M]}\right)}$. Let $f$ and $\wt f$ be progressively measurable and $f$ Lipschitz continuous, i.e., there exists some $L \in (0,\infty)$ such that for all $y,y'\in\mathbb{R}$ it holds $| f(s,y) - f(s,y') | \leq L |y-y'|$ $\PdM$-a.e. Moreover, suppose that $\xi$ and $\wt\xi$ are $\cF_T$-measurable  random variables. 
	Let $(Y,Z,M^\perp)$ be a solution of BSDE$(f,\xi)$ and $(\wt Y,\wt Z,\wt M^\perp)$ a solution of BSDE$(\wt f,\wt\xi)$ such that $E\left[ \int_0^T Z_s^2 d[M]_s \right] < \infty$, $E\left[ [M^\perp]_T \right] < \infty$ and $E\left[ \int_0^T \wt Z_s^2 d[M]_s \right] < \infty$, $E\left[ [\wt M^\perp]_T \right] < \infty$. 
	
	Denote $\delta Y_t = Y_t-\wt Y_t$ and $\delta f_t = f(t,\wt Y_t) - \wt f(t,\wt Y_t)$ for all $t\in[0,T]$. 
	Furthermore, define 	
	$
	b_t = 1_{\{Y_t\neq \wt Y_t\}} ( f(t,Y_t) - f(t,\wt Y_t) ) ( Y_t - \wt Y_t )^{-1},$ 
	$t\in [0,T],$ 
	and introduce the process $\Gamma=(\Gamma_t)_{t\in[0,T]}$ given by 
	$\Gamma_t = \exp\left( \int_0^t b_s d[M]_s \right)$, $t\in[0,T]$. 
	
	Then, $\delta Y$ admits the representation 
	\begin{equation}\label{eq:repres_delta_Y_comp}
	\delta Y_t = \Gamma_t^{-1} E_t\left[ \Gamma_T \delta Y_T + \int_t^T \Gamma_s \delta f_s d[M]_s \right], \quad t \in [0,T].
	\end{equation}
	
	In particular:
	\begin{enumerate}[(i)]
		\item If $\xi\geq \wt \xi$ a.s. and $f(s,\wt Y_s)\geq \wt f(s,\wt Y_s)$ $\PdM$-a.e., then $Y_t\geq \wt Y_t$ a.s. for all $t \in [0,T]$.
		\item If $\xi\leq \wt \xi$ a.s. and $f(s,\wt Y_s)\leq \wt f(s,\wt Y_s)$ $\PdM$-a.e., then $Y_t\leq \wt Y_t$ a.s. for all $t\in[0,T]$.
	\end{enumerate}
\end{propo}

\begin{proof}
	It holds for all $t\in[0,T]$ that  
	\begin{equation*}
	\begin{split}
	\delta Y_t 
	& = Y_T - \wt Y_T + \int_t^T \left( f(s,Y_s) - \wt f(s,\wt Y_s) \right) d[M]_s 
	 - \int_t^T Z_s dM_s
	  - \left( M^\perp_T - M^\perp_t \right) \\
	  & \quad 
	 + \int_t^T \wt Z_s dM_s + \left( \wt M^\perp_T - \wt M^\perp_t \right) .
	\end{split}
	\end{equation*}
	Since 
	\begin{equation*}
	f(s,Y_s) - \wt f(s,\wt Y_s)  = f(s,Y_s) - f(s,\wt Y_s) + f(s,\wt Y_s) - \wt f(s,\wt Y_s) 
	= b_s \delta Y_s + \delta f_s, \quad s\in[0,T],
	\end{equation*}
	it follows that 
	\begin{equation*}
	d\delta Y_s = - \left( b_s \delta Y_s + \delta f_s \right) d[M]_s + Z_s dM_s - \wt Z_s dM_s + dM^\perp_s - d\wt M^\perp_s, \quad s\in[0,T] .
	\end{equation*}
	Together with $d\Gamma_s = \Gamma_s b_s d[M]_s$, $s\in[0,T]$, we obtain by integration by parts that 
	\begin{equation*}
	\begin{split}
	\Gamma_T \delta Y_T & = \Gamma_t\delta Y_t - \int_t^T \Gamma_s \left( b_s \delta Y_s + \delta f_s \right) d[M]_s
	+ \int_t^T \Gamma_s Z_s dM_s 
	- \int_t^T \Gamma_s \wt Z_s dM_s \\
	& \quad  + \int_{(t,T]} \Gamma_s dM^\perp_s 
	- \int_{(t,T]} \Gamma_s d\wt M^\perp_s 
	+ \int_t^T \delta Y_s \Gamma_s b_s d[M]_s, 
	\quad t \in [0,T].
	\end{split}
	\end{equation*}
	If the local martingales $S=\int_0^{\cdot} \Gamma_s Z_s dM_s$, $\wt S=\int_0^{\cdot} \Gamma_s \wt Z_s dM_s$, 
	$U=\int_{(0,\cdot]} \Gamma_s dM^\perp_s$ and $\wt U=\int_{(0,\cdot]} \Gamma_s d\wt M^\perp_s$
	are true martingales, then it follows that 
	\begin{equation*}
	\Gamma_t \delta Y_t = E_t\left[ \Gamma_T \delta Y_T + \int_t^T \Gamma_s \delta f_s d[M]_s \right], \quad t \in [0,T],
	\end{equation*}
	which yields the representation~\eqref{eq:repres_delta_Y_comp} of $\delta Y$. 
	
	To show that $S$ is a martingale, note first that due to the Lipschitz continuity of~$f$, the process $b$ is bounded $\PdM$-a.e.\ by the corresponding Lipschitz constant. 
	By the Cauchy-Schwarz inequality it holds that 
	\begin{equation}\label{eq:S_is_mart_inequ2}
	\begin{split}
	E\left[ \left( \int_0^T \Gamma_s^2 Z_s^2 d[M]_s \right)^{\frac12}  \right] 
	& \leq E\left[ \left( \sup_{t\in[0,T]} \Gamma_t^2 \right)^{\frac12} \left( \int_0^T Z_s^2 d[M]_s \right)^{\frac12}  \right] \\
	& \leq \left( E\left[ \sup_{t\in[0,T]} \Gamma_t^2 \right] \right)^{\frac12} \left( E\left[ \int_0^T Z_s^2 d[M]_s \right] \right)^{\frac12} .
	\end{split}
	\end{equation}
	Since $b$ is bounded and $\bm{\left(C_{[M]}\right)}$ holds, 
	$E\left[ \sup_{t\in[0,T]} \Gamma_t^2 \right]<\infty$.
	We also have by assumption that $E\left[ \int_0^T Z_s^2 d[M]_s \right]<\infty$. 
	Therefore, it follows from~\eqref{eq:S_is_mart_inequ2} and the Burkholder-Davis-Gundy inequality that $E\left[ \sup_{t\in[0,T]} |S_t| \right] < \infty$. Thus, $S$ is a martingale. A similar reasoning applies also to $\wt S$, 
	$U$ and~$\wt U$.
	
	Finally, the claims (i) and~(ii) are straightforward consequences of~\eqref{eq:repres_delta_Y_comp}.
\end{proof}

We now apply Proposition~\ref{propo:comparison} to obtain $0\leq Y\leq \frac12$ in the proof of Proposition~\ref{propo:existence_bsde_general_filtration}.
Observe that $(\wt{Y},\wt{Z},\wt M^\perp)=\left( \frac{1}{2},0,0 \right)$ is a solution of BSDE$(0,\frac12)$, which obviously satisfies $E\left[ [\wt M^\perp]_T\right]<\infty$ and $E\left[ \int_0^T \wt Z_s^2 d[M]_s \right] <\infty$. 
Moreover,
with $\ol f$ as defined in the proof of Proposition~\ref{propo:existence_bsde_general_filtration},
it holds 
\begin{equation*}
\ol{f}\left(s,\frac12\right) = \frac{-\rho_s^2}{2\left( 2\rho_s + \mu_s \right)} \leq 0, \quad s\in[0,T], 
\end{equation*}
and both BSDEs have the same terminal value $\frac12$. 
Therefore, Proposition~\ref{propo:comparison} applies and yields $Y\leq \wt Y = \frac12$. 

For the other bound, note that $(\wt{Y},\wt{Z},\wt M^\perp)=\left( 0,0,0 \right)$ is a solution of BSDE$(0,0)$ with $E\left[ [\wt M^\perp]_T\right]<\infty$ and $E\left[ \int_0^T \wt Z_s^2 d[M]_s \right] <\infty$. Since $\ol{f}(s,0)=0$ for all $s\in[0,T]$ and $Y_T = \frac12 \geq 0 = \wt Y_T$, it follows from Proposition~\ref{propo:comparison} that $Y\geq \wt Y=0$.

\begin{remark}
Notice that in the proof of Proposition~\ref{propo:comparison} we need
$\bm{\left(C_{[M]}\right)}$ and the Lipschitz continuity of $f$ only to show that $E\left[ \sup_{t\in[0,T]} \Gamma_t^2 \right]$ is finite. 
Replace these two conditions by the assumption 
that there exists a predictable stochastic process $R$ such that for all $y,y'\in\mathbb{R}$, $|f(\omega,s,y) - f(\omega,s,y')|\leq R_s(\omega) |y-y'|$ $\PdM$-a.e.    
and for all $c\in (0,\infty)$, $E\left[ \exp\left( c\int_0^T R_s d[M]_s \right) \right] <\infty$. 
Then, we still have that 
\begin{equation*}
	\begin{split}
		E\left[ \sup_{t\in[0,T]} \Gamma_t^2 \right] 
		& = E\left[ \sup_{t\in[0,T]} \exp\left( 2 \int_0^t b_s d[M]_s \right) \right] 
		\leq E\left[  \exp\left( 2 \int_0^T R_s d[M]_s \right) \right] 
		< \infty . 
	\end{split} 
\end{equation*}
Hence, the claim of Proposition~\ref{propo:comparison} also applies to the setting mentioned in Remark~\ref{rem:28032021a1}.
\end{remark}

\hyphenation{par-ti-ci-pants}

\bigskip\noindent

\textbf{Acknowledgement:}
We are grateful to the Mathematical Finance session participants in the Bernoulli-IMS One World Symposium and to seminar participants in
Gie\ss{}en, Berlin, Moscow and Leeds for insightful discussions.
We thank Alexander Schied and two anonymous referees for their constructive comments and suggestions that helped us improve the manuscript.

\bibliographystyle{abbrv}
\bibliography{literature}

\begin{thebibliography}{10}

\bibitem{ackermann2020optimal}
J.~Ackermann, T.~Kruse, and M.~Urusov.
\newblock Optimal trade execution in an order book model with stochastic
  liquidity parameters.
\newblock {\em Preprint, arXiv:2006.05843}, 2020.

\bibitem{alfonsi2014optimal}
A.~Alfonsi and J.~I. Acevedo.
\newblock Optimal execution and price manipulations in time-varying limit order
  books.
\newblock {\em Applied Mathematical Finance}, 21(3):201--237, 2014.

\bibitem{alfonsi2008constrained}
A.~Alfonsi, A.~Fruth, and A.~Schied.
\newblock Constrained portfolio liquidation in a limit order book model.
\newblock {\em Banach Center Publ}, 83:9--25, 2008.

\bibitem{alfonsi2010optimal}
A.~Alfonsi, A.~Fruth, and A.~Schied.
\newblock Optimal execution strategies in limit order books with general shape
  functions.
\newblock {\em Quantitative Finance}, 10(2):143--157, 2010.

\bibitem{almgren2012optimal}
R.~Almgren.
\newblock Optimal trading with stochastic liquidity and volatility.
\newblock {\em SIAM Journal on Financial Mathematics}, 3(1):163--181, 2012.

\bibitem{almgren2001optimal}
R.~Almgren and N.~Chriss.
\newblock Optimal execution of portfolio transactions.
\newblock {\em Journal of Risk}, 3:5--40, 2001.

\bibitem{ankirchner2020optimal}
S.~Ankirchner, A.~Fromm, T.~Kruse, and A.~Popier.
\newblock Optimal position targeting via decoupling fields.
\newblock {\em To appear in Annals of Applied Probability}, 2020.

\bibitem{ankirchner2014bsdes}
S.~Ankirchner, M.~Jeanblanc, and T.~Kruse.
\newblock {BSDE}s with singular terminal condition and a control problem with
  constraints.
\newblock {\em SIAM Journal on Control and Optimization}, 52(2):893--913, 2014.

\bibitem{ankirchner2015optimal}
S.~Ankirchner and T.~Kruse.
\newblock Optimal position targeting with stochastic linear-quadratic costs.
\newblock {\em Advances in mathematics of finance}, 104:9--24, 2015.

\bibitem{bank2014optimal}
P.~Bank and A.~Fruth.
\newblock Optimal order scheduling for deterministic liquidity patterns.
\newblock {\em SIAM Journal on Financial Mathematics}, 5(1):137--152, 2014.

\bibitem{bank2018linear}
P.~Bank and M.~Vo{\ss}.
\newblock Linear quadratic stochastic control problems with stochastic terminal
  constraint.
\newblock {\em SIAM Journal on Control and Optimization}, 56(2):672--699, 2018.

\bibitem{becherer2019stability}
D.~Becherer, T.~Bilarev, and P.~Frentrup.
\newblock Stability for gains from large investors’ strategies in
  {$M_1$/$J_1$} topologies.
\newblock {\em Bernoulli}, 25(2):1105--1140, 2019.

\bibitem{bertsimas1998optimal}
D.~Bertsimas and A.~W. Lo.
\newblock Optimal control of execution costs.
\newblock {\em Journal of Financial Markets}, 1(1):1--50, 1998.

\bibitem{carmona2019selffinancing}
R.~Carmona and K.~Webster.
\newblock The self-financing equation in limit order book markets.
\newblock {\em Finance Stoch.}, 23(3):729--759, 2019.

\bibitem{cheridito2014optimal}
P.~Cheridito and T.~Sepin.
\newblock Optimal trade execution under stochastic volatility and liquidity.
\newblock {\em Applied Mathematical Finance}, 21(4):342--362, 2014.

\bibitem{fruth2014optimal}
A.~Fruth, T.~Sch\"{o}neborn, and M.~Urusov.
\newblock Optimal trade execution and price manipulation in order books with
  time-varying liquidity.
\newblock {\em Math. Finance}, 24(4):651--695, 2014.

\bibitem{fruth2019optimal}
A.~Fruth, T.~Sch\"{o}neborn, and M.~Urusov.
\newblock Optimal trade execution in order books with stochastic liquidity.
\newblock {\em Math. Finance}, 29(2):507--541, 2019.

\bibitem{garleanu2016dynamic}
N.~G{\^a}rleanu and L.~H. Pedersen.
\newblock Dynamic portfolio choice with frictions.
\newblock {\em Journal of Economic Theory}, 165:487--516, 2016.

\bibitem{graewe2017optimal}
P.~Graewe and U.~Horst.
\newblock Optimal trade execution with instantaneous price impact and
  stochastic resilience.
\newblock {\em SIAM Journal on Control and Optimization}, 55(6):3707--3725,
  2017.

\bibitem{graewe2015non}
P.~Graewe, U.~Horst, and J.~Qiu.
\newblock A non-{M}arkovian liquidation problem and backward {SPDE}s with
  singular terminal conditions.
\newblock {\em SIAM Journal on Control and Optimization}, 53(2):690--711, 2015.

\bibitem{graewe2018smooth}
P.~Graewe, U.~Horst, and E.~S{\'e}r{\'e}.
\newblock Smooth solutions to portfolio liquidation problems under
  price-sensitive market impact.
\newblock {\em Stochastic Processes and their Applications}, 128(3):979--1006,
  2018.

\bibitem{HorstKivman2021}
U.~Horst and E.~Kivman.
\newblock Optimal trade execution under small market impact and portfolio
  liquidation with semimartingale strategies.
\newblock {\em Preprint}, 2021.

\bibitem{horst2016constrained}
U.~Horst, J.~Qiu, and Q.~Zhang.
\newblock A constrained control problem with degenerate coefficients and
  degenerate backward {SPDE}s with singular terminal condition.
\newblock {\em SIAM Journal on Control and Optimization}, 54(2):946--963, 2016.

\bibitem{horst2019multi}
U.~Horst and X.~Xia.
\newblock Multi-dimensional optimal trade execution under stochastic
  resilience.
\newblock {\em Finance and Stochastics}, 23(4):889--923, 2019.

\bibitem{jacodshiryaev}
J.~Jacod and A.~N. Shiryaev.
\newblock {\em Limit theorems for stochastic processes}.
\newblock A Series of Comprehensive Studies in Mathematics; 288. Springer,
  Berlin, 2nd edition, 2003.

\bibitem{karatzasshreve}
I.~Karatzas and S.~E. Shreve.
\newblock {\em Brownian motion and stochastic calculus}.
\newblock Graduate texts in mathematics; 113. Springer, New York, 2nd edition,
  1991.

\bibitem{kruse2016minimal}
T.~Kruse and A.~Popier.
\newblock Minimal supersolutions for {BSDE}s with singular terminal condition
  and application to optimal position targeting.
\newblock {\em Stochastic Processes and their Applications}, 126(9):2554--2592,
  2016.

\bibitem{lorenz2013drift}
C.~Lorenz and A.~Schied.
\newblock Drift dependence of optimal trade execution strategies under
  transient price impact.
\newblock {\em Finance and Stochastics}, 17(4):743--770, 2013.

\bibitem{morlais2009quadratic}
M.-A. Morlais.
\newblock Quadratic {BSDE}s driven by a continuous martingale and applications
  to the utility maximization problem.
\newblock {\em Finance and Stochastics}, 13(1):121--150, 2009.

\bibitem{obizhaeva2013optimal}
A.~A. Obizhaeva and J.~Wang.
\newblock Optimal trading strategy and supply/demand dynamics.
\newblock {\em Journal of Financial Markets}, 16:1--32, 2013.

\bibitem{papapantoleon2018nineyards}
A.~Papapantoleon, D.~Possamaï, and A.~Saplaouras.
\newblock Existence and uniqueness results for {BSDE} with jumps: the whole
  nine yards.
\newblock {\em Electron. J. Probab.}, 23:68 pp., 2018.

\bibitem{popier2019second}
A.~Popier and C.~Zhou.
\newblock Second-order {BSDE} under monotonicity condition and liquidation
  problem under uncertainty.
\newblock {\em The Annals of Applied Probability}, 29(3):1685--1739, 2019.

\bibitem{predoiu2011optimal}
S.~Predoiu, G.~Shaikhet, and S.~Shreve.
\newblock Optimal execution in a general one-sided limit-order book.
\newblock {\em SIAM Journal on Financial Mathematics}, 2(1):183--212, 2011.

\bibitem{revuzyor}
D.~Revuz and M.~Yor.
\newblock {\em Continuous martingales and {B}rownian motion}, volume 293 of
  {\em Grundlehren der Mathematischen Wissenschaften}.
\newblock Springer-Verlag, Berlin, third edition, 1999.

\bibitem{schied2013control}
A.~Schied.
\newblock A control problem with fuel constraint and {D}awson--{W}atanabe
  superprocesses.
\newblock {\em The Annals of Applied Probability}, 23(6):2472--2499, 2013.

\end{thebibliography}

\end{document}